\documentclass{amsart}

\usepackage{epsfig}

\usepackage{amscd,amssymb,amsmath,url}
\usepackage{setspace}
\usepackage{amsthm}
\usepackage{array}
\usepackage{graphicx}
\usepackage[all]{xy}

\newcommand{\T}{\mathbb{T}}
\newcommand{\C}{\mathbb{C}}

\renewcommand{\P}{\mathbb{P}}

\newcommand{\Z}{\mathbb{Z}}
\renewcommand{\L}{\mathcal{L}}
\newcommand{\V}{\mathcal{V}}
\newcommand{\VV}{\mathbb{V}}
\newcommand{\WW}{\mathbb{W}}

\newcommand{\LL}{\mathbb{L}}
\renewcommand{\O}{\mathcal{O}}
\newcommand{\X}{\widetilde{X}}
\newcommand{\Tor}{\operatorname{Tor}}

\newcommand{\Hom}{\operatorname{Hom}}
\newcommand{\im}{\operatorname{im}}

\newcommand{\Span}{\operatorname{Span}}
\newcommand{\Spec}{\operatorname{Spec}}

\newcommand{\D}{\operatorname{D}}

\newcommand{\G}{\mathbb{G}}
\renewcommand{\u}{\mathfrak{u}}
\newdir^{ (}{{}*!/-5pt/@^{(}}%
\newcommand{\U}{\mathcal{U}}
\renewcommand{\k}{{\pmb{k}}}

\renewcommand{\S}{\mathcal{S}}
\newcommand{\op}{\operatorname{op}}
\newcommand{\vrho}{{\pmb{\rho}}}
\newcommand{\Lr}{\mathbf{L}}
\renewcommand{\biggl}{\bigg{(}}
\renewcommand{\biggr}{\bigg{)}}
\newcommand{\A}{\mathcal{A}}
\renewcommand{\a}{{\pmb{a}}}
\newcommand{\bomega}{{\pmb{\omega}}}
\newcommand{\x}{{\pmb{m}}}
\newcommand{\h}{\mathfrak{h}}
\newcommand{\bmu}{\pmb{\mu}}

\newcommand{\fin}{\operatorname{fin}}
\renewcommand{\deg}{\operatorname{deg}}

\newcommand{\un}{\operatorname{un}}
\newcommand{\w}{\omega}

\newcommand{\W}{\mathcal{W}}

\newif\ifCopR \CopRfalse

\newtheorem{theorem}{Theorem}[section]
\newtheorem{lemma}[theorem]{Lemma}
\newtheorem{proposition}[theorem]{Proposition}
\newtheorem{corollary}[theorem]{Corollary}

\theoremstyle{definition}
\newtheorem{definition}[theorem]{Definition}
\newtheorem{example}[theorem]{Example}

\newtheorem{notation}[theorem]{Notation}
\newtheorem{important}[theorem]{Important Fact}

\theoremstyle{remark}
\newtheorem{remark}[theorem]{Remark}

\author{Anthony Joseph Narkawicz}
\email{nark@math.duke.edu}

\title{Cohomology Jumping Loci and Relative Malcev Completion}
\address{Duke University Department Of Mathematics}

\begin{document}

\maketitle

\tableofcontents

\section{Introduction}
\label{chapterIntroduction}
\pagenumbering{arabic}

\noindent Two standard invariants used to study the fundamental group of the complement $X$ of a hyperplane arrangement are the Malcev (i.e. unipotent) completion of its fundamental group $\pi_1(X,x_0)$ and the cohomology groups $H^\bullet(X,\L_\vrho)$ with coefficients in rank one local systems. In this paper, we develop a tool that unifies these two approaches. This tool is the Malcev completion of $\pi_1(X,x_0)$ relative to a homomorphism $\vrho \colon \pi_1(X,x_0) \to (\C^*)^N$. This is a prosolvable group that generalizes the Malcev completion of $\pi_1(X,x_0)$ and is tightly controlled by the cohomology groups $H^1(X,\L_{\rho_1^{k_1}\cdots\rho_N^{k_N}})$.

\subsection{Relative Malcev Completion}
\label{RelativeMalcevCompletion}

\noindent Let $\vrho \colon \pi_1(X,x_0) \to (\C^*)^N$ be a representation. Let $D_\vrho$ denote the Zariski closure of the image of $\vrho$ in $\G_m^N$. The Malcev completion of $\pi_1(X,x_0)$ relative to $\vrho$ is a proalgebraic group over $\C$ (i.e., inverse limit of algebraic groups) that is an extension
$$
1 \longrightarrow \U_\vrho \longrightarrow \S_\vrho \longrightarrow D_\vrho \longrightarrow 1
$$
of $D_\vrho$ by a prounipotent group $\U_\vrho$. It is equipped with a Zariski dense homomorphism $\theta_\vrho \colon \pi_1(X,x_0) \to \S_\vrho$ that lifts $\vrho$, and it is a characterized by the following universal property. Suppose that the affine algebraic group $S$ is an extension
$$
1 \longrightarrow U \longrightarrow S \longrightarrow D_\vrho \longrightarrow 1
$$
of $D_\vrho$ by a unipotent group $U$ and that $\theta \colon \pi_1(X,x_0) \to S$ lifts $\vrho$:
$$
\xymatrix{\pi_1(X,x_0) \ar[d]_\theta \ar[dr]^\vrho \\ S \ar[r] & D_\vrho.}
$$
Then there is a unique map $\S_\vrho \to S$ such that the diagram
$$
\xymatrix{\pi_1(X,x_0) \ar[r]^{\theta_\vrho} \ar[dr]_\theta & \S_\vrho \ar[d] \\ & S}
$$
commutes.

If $\vrho$ is the trivial homomorphism, then $D_\vrho$ is the trivial group, and $\S_\vrho$ is the standard Malcev completion of $\pi_1(X,x_0)$. This is the prounipotent group whose Lie algebra is the completion $\h^\wedge$ of Kohno's \cite{Kohno} holonomy Lie algebra
$$
\h = \frac{\LL(H_1(X,\C))}{\langle\im\partial\rangle},
$$
where $\partial \colon H_2(X,\C) \to \bigwedge^2 H_1(X,\C)$ is the dual of the cup product, and $\langle\im\partial\rangle$ is the ideal generated by the image of $\partial$.

The relative Malcev completion of the fundamental group of a knot complement was studied by Miller \cite{Miller}.

\subsection{The Pronilpotent Lie Algebra}
\label{PronilpLie}

\noindent The exponential and logarithm maps determine an equivalence of categories between prounipotent algebraic groups and pronilpotent Lie algebras. In particular, the prounipotent group $\U_\vrho$ corresponds to a pronilpotent Lie algebra $\u_\vrho$. Each character $\alpha$ of $D_\vrho$ determines a one dimensional representation $V_\alpha$ of $D_\vrho$ and a rank one local system $\VV_\alpha$ on $X$ with monodromy given by the character $\alpha \circ \vrho \colon \pi_1(X,x_0) \to \C^*$.

If $\u$ is any pronilpotent Lie algebra, then $\u$ is strongly controlled by $H_1(\u)$ and $H_2(\u)$. As described in the proof of Proposition 7.1 of \cite{HainTorelli}, there is a $D_\vrho$-equivariant isomorphism
\begin{equation}
\label{Dp1}
\prod_{\alpha \in D_\vrho^\vee} H^1(X,\VV_\alpha)^*\otimes V_\alpha \hspace{.05 in} \cong \hspace{.05 in} H_1(\u_\vrho)
\end{equation}
and a $D_\vrho$-equivariant surjection
\begin{equation}
\label{Dp2}
\prod_{\alpha \in D_\vrho^\vee} H^2(X,\VV_\alpha)^* \otimes V_\alpha \longrightarrow H_2(\u_\vrho).
\end{equation}

Rational homotopy theory provides several equivalent methods for constructing a pronilpotent Lie algebra from a connected commutative differential graded algebra (e.g. bar construction, formal power series connections, and minimal models). The Lie algebra $\u_\vrho$ is the pronilpotent Lie algebra constructed from the commutative differential graded algebra
$$
E^\bullet(X,\O_\vrho) = \bigoplus_{\alpha \in D_\vrho^\vee} E^\bullet(X,\VV_\alpha) \otimes V_\alpha^*
$$
This algebra has a product that we describe in Section \ref{CommDiffE}. If $\vrho$ has Zariski dense image in $\G_m^N$, then
$$
E^\bullet(X,\O_\vrho) = \bigoplus_{\k \in \Z^N} E^\bullet(X,\L_{\vrho_1^{k_1}\cdots \vrho_N^{k_N}}) q_1^{-k_1}\cdots q_N^{-k_N}.
$$

Standard results of rational homotopy theory imply that the Lie algebra $\u_\vrho$ is quadratically presented if and only if the differential graded algebra $E^\bullet(X,\O_\vrho)$ is 1-formal. In particular, if $\u_\vrho$ is quadratically presented, then all Massey triple products of $1$-forms must vanish modulo their indeterminacies. If $\vrho$ is the trivial representation, then $D_\vrho$ is the trivial group and $E^\bullet(X,\O_\vrho) = E^\bullet(X,\C)$. Thus, $\u_\vrho = \u_\mathbf{1}$ is the completion $\h^\wedge$ of Kohno's \cite{Kohno} holonomy Lie algebra $\h$, which is quadratically presented. As the next theorem shows, $E^\bullet(X,\O_\vrho)$ is not, in general, 1-formal.

Let $X \subset \C^2$ denote the complement of the braid arrangement $\mathcal{B}$, and let $\T = H^1(X,\C^*)$ denote its character torus. The intersection of $\mathcal{B}$ with $\mathbb{R}^2$ is shown below. Let the hyperplanes be numbered as indicated.
\begin{figure}[!ht]
\label{Braiddf}
\centering\epsfig{file=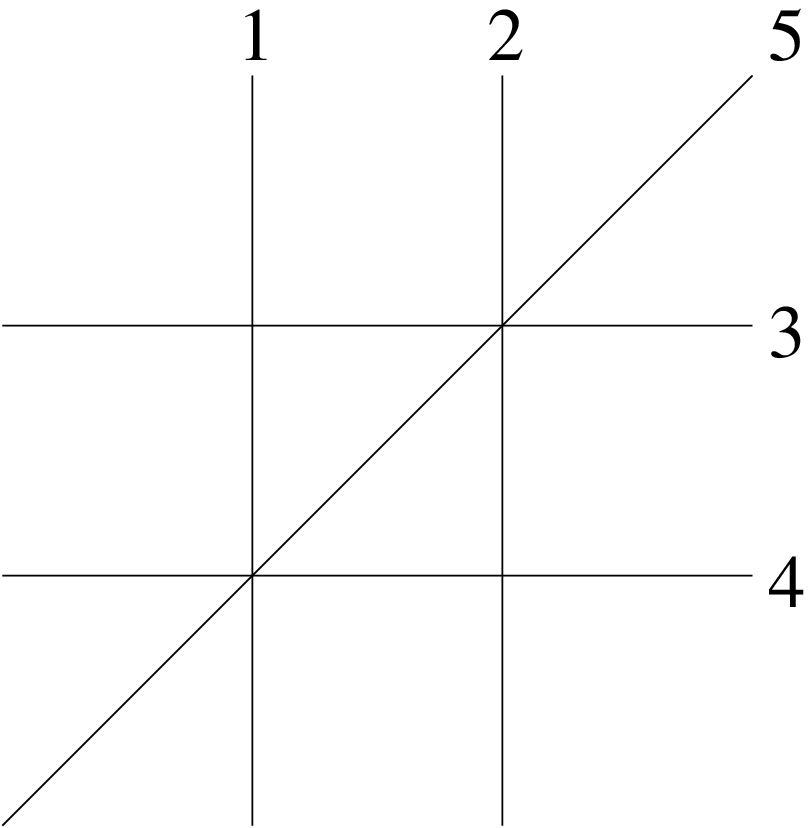, width=2.2in}
\caption{The braid arrangement $\mathcal{B}$}
\end{figure}

\begin{theorem}
\label{FirstMassey}
There exist infinitely many $\vrho \in \T^2$ for which $H^\bullet(X,\O_\vrho)$ has a nonzero Massey triple product of degree-one elements. Thus, the commutative differential graded algebra $E^\bullet(X,\O_\vrho)$ is not $1$-formal and therefore the pronilpotent Lie algebra $\u_\vrho$ is not quadratically presented. \qed
\end{theorem}

\subsection{Completion and Characteristic Varieties}
\label{CompletionCharVar}

\noindent For each $i \geq 0$, the character torus has a filtration $$\T = \V_0^i(X)\supset\V_1^i(X) \supset \dots$$ by {\em characteristic} (sub)varieties of $\T$, where $$\V_m^i(X) := \{\rho \in \T\colon \dim H^i(X,\L_\rho) \geq m\}.$$ There is a similar stratification of $\T^N$ defined by $$\V_{N,m}^i(X) := \{(\rho_1,\dots,\rho_N) \in \T^N \colon \dim H^i(X,\L_{\rho_1\cdots\rho_N}) \geq m\}.$$ In \cite{Arapura}, Arapura proved a general result that implies that the subvariety $\V_{N,m}^i(X)$ of $\T^N$ is the union of translates of subtori of $\T^N$ by torsion characters.

If $\vrho \in \T^N$, then the universal property of the relative Malcev completion $\S_\vrho$ gives a surjection $\S_\vrho \longrightarrow D_\vrho \times \pi_1(X,x_0)^{\un}$ of groups, where $\pi_1(X,x_0)^{\un}$ is the Malcev completion of $\pi_1(X,x_0)$. Thus, there is a surjection $\S_\vrho \to \pi_1(X,x_0)^{\un}$, which is an isomorphism if $\vrho$ is the trivial representation. Consequently, we may view $\S_\vrho$ as a kind of ``deformation" of $\pi_1(X,x_0)^{\un}$ over $\T^N$.

In Section \ref{SectionY}, we examine exactly how $\S_\vrho$ depends on $\vrho$. The first result in this direction is the following theorem.
\begin{theorem}
\label{SDU}
If $X$ is the complement of an arrangement of hyperplanes in a complex vector space and two distinct hyperplanes intersect, then $\S_\vrho \cong D_\vrho \times \pi_1(X,x_0)^{\un}$ for general\footnote{Those $\vrho \in Y$ which lie in an intersection of countably many Zariski open sets.} $\vrho \in \T^N$.
\end{theorem}
If $\vrho$ lies in the characteristic variety $\V_{N,1}^1(X)$, then $\S_\vrho$ is not isomorphic to $D_\vrho \times \pi_1(X,x_0)^{\un}$.

\subsection{The Problem with $(A^\bullet,-\a\bomega^T)$}
\label{IntroductionProblem}

\noindent Suppose that $X$ is the complement of an arrangement of $n$ hyperplanes in a complex vector space $V$. Choose a linear function $L_j$ on $V$ whose vanishing set is the $j$-th hyperplane. Set $\omega_j = (2\pi i)^{-1}d L_j/L_j$. This is a closed holomorphic $1$-form on $X$ with integral periods. Let $A^\bullet$ denote the subalgebra of $E^\bullet(X,\C)$ generated by the forms $\omega_j$. This is the complexified {\em Orlik-Solomon algebra} of the arrangement.

Define $\bomega = (\omega_1,\dots,\omega_n)$, and let $\bomega^T$ denote the transpose of $\bomega$. Given $\a \in \C^n$, set $\a \bomega^T = a_1 \omega_1 + \dots + a_n \omega_n$. This is an element of $H^1(X,\C)$. Its exponential $\vrho = \exp(\a\bomega^T)$ is an element of $\T$. Let $\nabla_{\a}$ denote the connection on the trivial line bundle $\C \times X \to X$ defined by $\nabla_\a\sigma = d\sigma - (\a \bomega^T)\sigma$ for $\sigma \in E^0(X)$. There is a natural inclusion
$$
(A^\bullet,-\a\bomega^T) \, \hookrightarrow \, E^\bullet(X,\L_\vrho).
$$
of complexes. Though the product in $A^\bullet$ induces a product in $H^\bullet(A^\bullet,-\a\bomega^T)$, the algebra $(A^\bullet,-\a\bomega^T)$ is not a differential graded algebra. The cup product of two elements of $H^\bullet(X,\L_{\vrho})$ lies in $H^\bullet(X,\L_{\vrho^2})$. If $\zeta$ and $\psi$ are elements of $(A^\bullet,-\a\bomega^T)$, then $\zeta \wedge \psi$ is an element of the complex $(A^\bullet,-2\a\bomega^T)$. The diagram
$$
\xymatrix@1{(A^\bullet,-\a\bomega^T) \otimes_\C (A^\bullet,-\a\bomega^T) \ar[d] \ar[r]^-\wedge & (A^\bullet,-2\a\bomega^T) \ar[d] \\ E^\bullet(X,\L_\vrho) \otimes_\C E^\bullet(X,\L_\vrho) \ar[r]^-\wedge & E^\bullet(X,\L_{\vrho^2})}
$$
commutes, and all maps are chain maps. That is, although the cup product of two elements in $(A^\bullet,-\a\bomega^T)$ is an element of this same complex, it is more naturally an element of the complex $(A^\bullet,-2\a\bomega^T)$. Thus, it is natural to define
$$
\A_\a^\bullet = \bigoplus_{k \in \Z} A^\bullet q^{-k}.
$$
This is a commutative differential graded $\C$-algebra, where the differential is given on the $k$-th component by left multiplication by $-k\a\bomega^T$. It is graded by degree of differential forms.

In general, if $\a$ is any element of $M_{N \times n}(\C)$, then $\a\bomega^T$ is an element of $H^1(X,\C^N)$. Thus, $\vrho = \exp(\a\bomega^T)$ is an element of $\T^N$. If $\k \in \Z^N$, then $\k\a\bomega^T$ is an element of $A^\bullet$. In this case, we define
$$
\A_\a^\bullet = \bigoplus_{\k \in \Z^N} A^\bullet q_1^{-k_1}\cdots q_N^{-k_N}.
$$
As above, this is a commutative differential graded $\C$-algebra, where the differential is given on the $\k$-th component by left multiplication by $-\k\a\bomega^T$. There is a canonical homomorphism
$$
\A_\a^\bullet \longrightarrow E^\bullet(X,\O_\vrho)
$$
of commutative differential graded algebras, which is an inclusion when $\vrho$ has Zariski dense image in $\G_m^N$.

\begin{theorem}
\label{VTheorem}
If $\mathbf{V}$ is a vector subspace of $M_{N \times n}(\C)$, then there is a countable collection $\{\W_j\}$ of proper affine subspaces of $\mathbf{V}$ that do not contain $0$ with the following property. If $\a \in \mathbf{V}-\bigcup_j\W_j$ and $\vrho = \exp(\a\bomega^T)$ has Zariski dense image in $\G_m^N$, then the induced homomorphism $H^\bullet(\A_{\a}^\bullet) \longrightarrow H^\bullet(X,\O_\vrho)$ is an isomorphism.
\end{theorem}

Recall that Theorem \ref{FirstMassey} says that when $X$ is the complement of the braid arrangement in $\C^2$, there exists $\vrho \in \T^2$ such that $H^1(X,\O_\vrho)$ has a nonvanishing Massey triple product of degree-one elements. To prove Theorem \ref{FirstMassey}, we apply Theorem \ref{VTheorem} and then exhibit a nonvanishing Massey triple product in $H^2(\A_\a^\bullet)$ for a $2$ by $5$ matrix $\a$.

In addition, we use Theorem \ref{VTheorem} to give conditions under which $\S_\vrho$ is combinatorially determined.

\subsection{When $\S_\vrho$ is Combinatorially Determined}
\label{IntroductionSpComb}

Let $X$ denote the complement of an arrangement of hyperplanes in a complex vector space, and let $\h$ denote its holonomy Lie algebra. Let $\h^\wedge$ denote its completion with respect to degree. Then $\h^\wedge$ is the pronilpotent Lie algebra constructed from the differential graded algebra $E^\bullet(X)$ by the methods of rational homotopy theory. Let $A^\bullet$ denote the complexified Orlik-Solomon algebra of $X$. The inclusion $A^\bullet \hookrightarrow E^\bullet(X)$ is a quasi-isomorphism \cite{OS1}. Thus, $\h^\wedge$ can also be constructed from the differential graded algebra $A^\bullet$. The Orlik-Solomon algebra is determined by the intersection poset of the hyperplane arrangement. Thus, the pronilpotent Lie algebra $\h^\wedge$ is also determined by the intersection poset. The Malcev completion $\pi_1(X,x_0)^{\un}$ is the unique prounipotent group whose Lie algebra is $\h^\wedge$. Thus, $\pi_1(X,x_0)^{\un}$ is determined by the intersection poset of the arrangement. For this reason, we say that $\pi_1(X,x_0)^{\un}$ is {\em combinatorially determined}.

It is natural to ask whether, in general, the isomorphism class of the relative Malcev completion $\S_\vrho$ is combinatorially determined. The first result in this direction is given in Theorem \ref{SDU}, which implies that if two distinct hyperplanes intersect, then $\S_\vrho \cong D_\vrho \times \pi_1(X,x_0)^{\un}$ for general $\vrho \in \T$. For such $\vrho$, the relative Malcev completion $\S_\vrho$ is combinatorially determined. This generalizes to any subtorus of the character torus $\T$ that contains the trivial character.

\begin{theorem}
If $Y$ is a subtorus of $\T$ that contains the trivial character, then the isomorphism class of the relative Malcev completion $\S_\vrho$ is combinatorially determined for general $\vrho \in Y$.
\end{theorem}

We do not know whether $\S_\vrho$ is always combinatorially determined. The isomorphism (\ref{Dp1}) and the surjection (\ref{Dp2}) of Lie algebra homologies suggest that the question of whether $\S_\vrho$ is combinatorially determined is related to the question of whether characteristic varieties are combinatorially determined.

\begin{theorem}
If the isomorphism class of the relative Malcev completion $\S_\vrho$ is combinatorially determined for all $\vrho \in \T$, then the characteristic variety $\V_m^1(X) = \{\vrho \in \T \, | \, \dim_\C H^1(X,\L_\vrho) \geq m\} = 0$ is combinatorially determined.
\end{theorem}

\subsection{Constancy Over Characteristic Varieties}
\label{IntroductionConstancy}

\noindent Let $Y$ be an irreducible subvariety of $\T^N$. In Section \ref{SectionY}, we examine how $\S_\vrho$ deforms as $\vrho \in Y$ varies. One natural and interesting choice for $Y$ is an irreducible component of the characteristic variety $\V_{N,m}^i(X)$.

There is an affine group scheme $\S_Y$ over $Y$ such that for $Y = \{\vrho\}$, $\S_Y$ is the Malcev completion of $\pi_1(X,x_0)$ relative to $\vrho$. Let $\O(Y)$ denote the coordinate ring of $Y$. There is a homomorphism
$$
\theta_Y \colon \pi_1(X,x_0) \longrightarrow \S_Y(\O(Y))
$$
into the group of $\O(Y)$-rational points of $\S_Y$. For each $\vrho \in Y$, there is a homomorphism $\S_\vrho \to \S_Y \otimes_{\O(Y)} \C_\vrho$ of affine group schemes over $\C$, where $\C_\vrho$ is the residue field associated to $\vrho$. This residue field is naturally a quotient of $\O(Y)$. The diagram
$$
\xymatrix{\pi_1(X,x_0) \ar[d]_{\theta_Y} \ar[r]^{\theta_\vrho} & \S_\vrho(\C) \ar[d] \\ \S_Y(\O(Y)) \ar[r] & \S_Y(\C_\vrho)}
$$
commutes.

\begin{theorem} If there exists $\pmb{\varrho} \in Y$ such that $D_{{\pmb{\varrho}}}$ contains $\im \vrho$ for all $\vrho \in Y$, then the homomorphism $$\S_\vrho \longrightarrow \S_Y \otimes_{\O(Y)} \C_\vrho$$ is an isomorphism for general $\vrho \in Y$.
\end{theorem}

\begin{remark}
If $Y$ is an irreducible subvariety of $\T$ that has positive dimension, then the general $\vrho \in Y$ has Zariski dense image in $\G_m$. Thus, if $N = 1$, then the hypotheses of the theorem are always satisfied.
\end{remark}

\subsection{Acknowledgements}

I would like to thank my advisor, Richard Hain, for his guidance and encouragement. He has helped to instill in me a love of mathematics and an appreciation for rigorous and creative accomplishment.

\section{Notation and Conventions}
\label{NotationAndConventions}

For the convenience of the reader, this section is an outline of the basic conventions that will be used.

All differential forms are assumed to be complex-valued. For a manifold $X$, $E^\bullet(X)$ denotes $E^\bullet(X,\C)$.

We multiply paths in their natural order. That is, if $\gamma, \beta \colon [0,1] \to X$ are paths in a topological space $X$ such that $\gamma(1) = \beta(0)$, then the path $\gamma\beta$ is given by $(\gamma\beta)(t) = \gamma(2t)$ for $0 \leq t \leq \frac{1}{2}$ and $(\gamma\beta)(t) = \beta(2t-1)$ for $\frac{1}{2} \leq t \leq 1$. If $(\X,\tilde{x}_0) \to (X,x_0)$ is a pointed universal covering of $X$, then $\pi_1(X,x_0)$ acts on the {\em left} of $\X$.

All schemes and varieties are assumed to be affine.

If $G$ is an affine group scheme over a commutative ring $R$, then all $G$-modules are assumed to be {\em right} $G$-modules.

If $n$ is an integer, elements of $\C^n$ are $1$ by $n$ vectors with entries in $\C$. If $\pmb{w} \in \C^n$, we let $\pmb{w}^T$ denote its transpose. Thus, if $\k \in \C^N$, $\pmb{w} \in \C^n$, and $\a$ is an $N$ by $n$ matrix with entries in $\C$, then $\k\a\pmb{w}^T$ is a complex number.

Suppose that $X$ is the complement of an arrangement of $n$ hyperplanes in a complex vector space $V$. For $j = 1,\dots,n$, choose a linear function $L_j$ on $V$ whose vanishing set is the $j$-th hyperplane. Set $\omega_j = (2\pi i)^{-1} d L_j /L_j$. This is a closed holomorphic $1$-form on $X$. Set $\bomega = (\omega_1,\dots,\omega_n)$, and let $\bomega^T$ denote the transpose of $\bomega$. If $\a$ is an $N$ by $n$ matrix with entries in $\C$, then $\a\bomega^T$ is a closed holomorphic $1$-form on $X$ with entries in $\C^N$. If $\k \in \Z^N$, then $\k\a\bomega^T$ is a closed holomorphic $1$-form on $X$ with entries in $\C$.

\section{Hyperplane Arrangements and Characteristic Varieties}
\label{SectionHyp}

This section is a review of some basic facts about the topology of complements of hyperplane arrangements.

\subsection{Hyperplane Arrangements}
\label{HypArr}
\label{OSAlg}

Let $\{K_1,\dots,K_n\}$ be an affine hyperplane arrangement in a complex vector space $V$ of dimension $\ell$. The complement $X = V - \bigcup_{j=1}^n K_j$ is affine and therefore has the homotopy type of a CW complex of dimension at most $\ell$ \cite[Theorem 7.2]{MilnorMorse}.

Suppose that $x_0 \in X$. Let $L_j$ be a defining equation for $K_j$, and set $\omega_j = (2\pi i)^{-1} d L_j/L_j$. This is a closed holomorphic $1$-form on $X$ whose periods are integers. Set $\bomega = (\omega_1,\dots,\omega_n)$, and let $\bomega^T$ denote the transpose of $\bomega$. If $\a$ is an $N$ by $n$ matrix with entries in an additive abelian group $A$, and if $\k \in \Z^N$, then $\k\a\bomega^T$ is a closed $1$-form on $X$ with values in $A$. The Zariski-Lefschetz theorem \cite{GorMac} implies that there are generators $\gamma_1,\dots,\gamma_n$ of $\pi_1(X,x_0)$ such that
$$
\int_{\gamma_j} \omega_k = \delta_{jk}.
$$
It follows from work of Brieskorn \cite{Brieskorn} that if $A$ is an abelian group under addition, then there is a natural isomorphism $A^n \overset{\simeq}{\longrightarrow} H^1(X,A)$ given by $\a \mapsto \a\bomega^T$. In particular, the character group $H^1(X,\C^*)$ is naturally isomorphic to the torus $(\C^*)^n$.

Denote by $A^\bullet$ the subalgebra of $E^\bullet(X)$ generated by the forms $\omega_1,\dots,\omega_n$. This algebra is due to Brieskorn and is known as the complexified Orlik-Solomon algebra. Brieskorn's theorem \cite{Brieskorn} implies that the homomorphism $A^\bullet \to H^\bullet(X,\C)$ is an isomorphism.

\subsection{Local Systems and Characteristic Varieties}
\label{ROLS}
\label{SectionESV}
\label{CharVar}

Define $\T$ to be the character torus $\T := H^1(X,\C^*)$. For each $\rho \in \T$, define $\L_\rho$ to be the rank one local system on $X$ with monodromy $\rho$. Choose $\a \in \C^n$ such that $\rho = \exp(\a\bomega^T)$. Then $\a$ determines a connection $\nabla_\a$ on the trivial line bundle $\C \times X \to X$ via the formula
$$
\nabla_\a\sigma = d\sigma - (\a\bomega^T)\sigma
$$
for $\sigma \in E^0(X,\C)$. It is flat because $d (\a\bomega^T) = (\a\bomega^T)\wedge(\a\bomega^T) = 0$. The monodromy representation of $\nabla_\a$ is $\vrho$. Consequently, there is an isomorphism $\L_\rho \cong (\C \times X, \nabla_\a)$ of flat line bundles on $X$.

If $\a \in \C^n$, then $\a\bomega^T$ is an element of $A^\bullet$. Thus, left multiplication by $-\a \bomega^T$ determines a differential on $A^\bullet$, which will be denoted $-\a\bomega^T$. The resulting complex $(A^\bullet,-\a\bomega^T)$ is then a subcomplex of $(E^\bullet(X),\nabla_\a)$. Let $\widehat{K}_j$ denote the closure of $K_j$ in $\P(V \oplus \C)$. Define $\widehat{K}_0 = \P(V)$ to be the hyperplane at infinity in $\P(V \oplus \C)$. Set $a_0 = -a_1 - \dots - a_n$. A subset $S$ of $\P(V \oplus \C)$ that is the intersection of a collection of the hyperplanes $\widehat{K}_j$ is said to be {\em dense} if the subarrangement consisting of all $\widehat{K}_j$ containing $S$ is not the product of two nonempty subarrangements. Given such a subset $S$ and a positive integer $M$, define a linear polynomial $\lambda_{S,M}$ on $\C^n$ by $\lambda_{S,M}(\a) = M - \sum_{S \subset \widehat{K}_j} a_j$.
\begin{theorem}[Esnault, Schechtman, and Viehweg, \cite{ESV}]
\label{ESVTheorem}
If $\a \in \C^n$ and $\lambda_{S,M}(\a) \neq 0$ for all dense $S$ and positive integers $M$, then the inclusion $$
(A^\bullet,-\a\bomega^T) \hookrightarrow (E^\bullet(X),\nabla_\a)$$ of complexes induces an isomorphism on cohomology. \qed
\end{theorem}

For each $i \geq 0$, the character torus has a filtration $$\T = \V_0^i(X)\supset\V_1^i(X) \supset \dots$$ by {\em characteristic} (sub)varieties of $\T$, where $$\V_m^i(X) := \{\rho \in \T\colon \dim H^i(X,\L_\rho) \geq m\}.$$ There is a similar stratification of $\T^N$ defined by $$\V_{N,m}^i(X) := \{(\rho_1,\dots,\rho_N) \in \T^N \colon \dim H^i(X,\L_{\rho_1\cdots\rho_N}) \geq m\}.$$ In \cite{Arapura}, Arapura proved the following theorem, which relies on and is closely related to work of Green and Lazarsfeld \cite{GreenLazarsfeld} and Simpson \cite{Simpson}.

\begin{theorem}
\label{ArapuraTheorem}
If $X$ is a smooth quasi-projective variety, then the characteristic variety $\V_{N,m}^i(X)$ is the union of translates of subtori of $\T^N$ by torsion characters.
\end{theorem}

It follows from the results in Yuzvinsky's paper \cite{Yuz1} that if $X$ is the complement of an arrangement of hyperplanes in a complex vector space and two distinct hyperplanes intersect, then $\V_{N,1}^1(X)$ is a proper subvariety of $\T^N$.

\section{Affine Group Schemes}
\label{AGS}
\label{AffineGroupSchemes}

\noindent Several of the basic objects in this paper are affine group schemes, either because their coordinate rings are not finitely generated or because they are defined over the ring of functions on a subvariety of a characteristic variety.

Let $R$ be a commutative ring with identity. Throughout this paper, we assume that all algebras are commutative and have a multiplicative identity. The category of {\em affine schemes} over $R$ is by definition the opposite category of the category of commutative $R$-algebras:
$$
\{\textrm{affine schemes over $R$}\} \, = \, \{\textrm{commutative $R$-algebras}\}^{\op}.
$$
If $A$ is an $R$-algebra, the corresponding affine scheme is denoted $\Spec A$. If $X$ is an affine scheme over $R$, the corresponding $R$-algebra is denoted $\O(X)$. If $X$ and $Y$ are affine schemes over $R$, morphisms $X \to Y$ are $R$-algebra homomorphisms $\O(Y) \to \O(X)$. Each affine scheme $X$ over $R$ gives rise to a set-valued functor
$$
X \colon {\mathsf{Alg}}_R \longrightarrow {\mathsf{Sets}}
$$
that takes the $R$-algebra $A$ to its set
$$
X(A) = \Hom_R(\O(X),A)\, ,
$$
of $A$-rational points.

An {\em affine group scheme} over $R$ is a group object in the category of schemes over $R$. The category of affine schemes over $R$ is opposite to the category of Hopf algebras over $R$. That is, if $G$ is an affine group scheme over $R$, then $\O(G)$ is a Hopf algebra over $R$. Conversely, if $A$ is a Hopf algebra over $R$, then $\Spec A$ is an affine group scheme over $R$.

If $G$ is an affine group scheme over $R$ and $A$ is any $R$-algebra, then the set $G(A)$ of $A$-rational points of $G$ is a group.

\begin{example}
The ring $R[q_j^{\pm 1}] = R[q_1^{\pm 1},\dots,q_N^{\pm 1}]$ of Laurent polynomials is a Hopf algebra over $R$. The comultiplication is given by $\Delta (q_j) = q_j \otimes q_j$, the antipode by $\lambda (q_j) = q_j^{-1}$, and the counit by $\epsilon (q_j) = 1$. Let $\G_{m/R}^N = \Spec R[q_j^{\pm 1}]$ denote the corresponding affine group scheme. For each $R$-algebra $A$, there is a canonical isomorphism $\G_{m/R}^N(A) \cong (A^\times)^N$ of groups. In particular, when $R = \C$, the group $\G_{m/\C}^N(\C)$ is $(\C^*)^N$. If $R$ is a field, we denote the affine group scheme $\G_{m/R}^N$ by $\G_m^N$. If $R$ is the coordinate ring of an affine variety $Y$, we write $\G_{m/Y}^N$ for this group scheme.
\end{example}

\begin{example}
\label{ExampleMu}
Let $r$ be a positive integer, and let $F$ be a field. Consider the Hopf algebra $F[q]/(q^r-1)$, which has comultiplication $\Delta \bar{q}^j = \bar{q}^j \otimes \bar{q}^j$, antipode $\lambda(\bar{q}^j) = \bar{q}^{-j}$, and counit $\epsilon(\bar{q}^j) = 1$. The affine algebraic group scheme $\bmu_r = \Spec F[q^{\pm 1}]/(q^r-1)$ sends each $F$-algebra to its group of $r$-th roots of unity.
\end{example}

Let $G$ be an affine group scheme over $R$. If $A$ is an $R$-algebra, then $\O(G) \otimes_R A$ is a Hopf algebra over $A$. We define
$$
G \otimes_R A = \Spec (\O(G) \otimes_R A).
$$
This is an affine group scheme over $A$. Equivalently, the functor $G \otimes_R A$ from $A$-algebras to groups is the restriction of the functor $G$ to $A$-algebras \cite[Section 1.6]{Waterhouse}.

If $\O(G)$ is a finitely generated $R$-algebra, then $G$ is called {\em algebraic}. If $G$ is the limit of an inverse system of affine algebraic group schemes, then $G$ is called {\em proalgebraic}.

Suppose that $A$ is an $R$-algebra and that $K$ is a subset of $G(A)$. We define the {\em Zariski closure} of $K$ in $G$ to be the intersection of all affine subschemes $G_\alpha$ of $G$ over $R$ such that $K \subset G_\alpha(A)$. This is a group subscheme of $G$.

\subsection{Right Modules and Representations}
\label{GroupSchemeReps}

An $R$-module $M$ is a (right) {\em $G$-module} if it is equipped with an $R$-module map $\phi \colon M \to \O(G) \otimes_R M$ such that $(I \otimes \phi) \circ \phi = (\Delta \otimes I) \circ \phi$ and $(\epsilon \otimes I) \circ \phi = I$. If $R$ is a field, then $M$ is said to be a (right) {\em representation} of $G$. In what follows, all modules and representations are assumed to be right modules and representations, respectively. If $M$ is a $G$-module, there is a (right) action of $G(A)$ on $A \otimes_R M$.

\begin{example}
The algebra $\O(G)$ has the structure of a (right) $G$-module, with structure map $\O(G) \to \O(G) \otimes_R \O(G)$ given by the coproduct. Consider the affine group scheme $\G_m$ over a field $F$, with $\O(\G_m) = R[q^{\pm 1}]$. The coproduct in $F[q^{\pm 1}]$ sends $q$ to $q \otimes q$. An element $\zeta \in F^\times$ of the $F$-rational points of $\G_m$ acts on this algebra via $(h\cdot \zeta)(q) = h(\zeta q)$ for $h \in F[q^{\pm 1}]$. It acts on $\Span_F q^j$ by multiplication by $\zeta^j$.
\end{example}

Over a field, an affine algebraic group scheme $G$ is {\em reductive} if for each representation $V$ of $G$, every subrepresentation of $V$ has a $G$-invariant complement.

\begin{example}
Suppose that $N$ is a positive integer and $\k \in \Z^N$. The ring $R[q_j^{\pm 1}] = R[q_1^{\pm 1},\dots,q_N^{\pm 1}]$ of Laurent polynomials is a Hopf algebra. The comultiplication sends each $q_j$ to $q_j \otimes q_j$. Its spectrum is the affine algebraic group scheme $\G_{m/R}^N$. The ring $R$ is itself a free $R$-module of rank one. There is a homomorphism $R \to R[q_j^{\pm 1}] \otimes_R R$ of $R$-modules given by $1 \longmapsto q_1^{k_1}\cdots q_N^{k_N} \otimes 1$. Thus, we may view $R$ as a $\G_{m/R}^N$-module. The action of an element $(r_1,\dots,r_N)$ of the $R$-rational points $(R^\times)^N$ of this group scheme is given by multiplication by $r_1^{k_1}\cdots r_N^{k_N}$. When $R$ is a field, we call this module the $\k$-th standard representation of $\G_m^N$. The group scheme $\G_m^N$ is reductive, and each irreducible representation is isomorphic to the $\k$-th standard representation for some $\k$.
\end{example}

Suppose that $F$ is a field and that $G$ is an affine group scheme over $F$. If $V$ is a representation of $G$, we say that a vector $v \in V$ is {\em fixed} by $G$ if the structure map $V \to \O(G) \otimes_F V$ sends $v$ to $1 \otimes v$. An affine algebraic group scheme over a field is {\em unipotent} if every nonzero representation has a nonzero fixed vector.
Over a field of characteristic zero, there is a bijection between unipotent group schemes and finite dimensional nilpotent Lie algebras. In Section \ref{RelCompBasics}, we will discuss inverse limits of unipotent group schemes. Thus, we define an affine group scheme over a field to be {\em prounipotent} if it is the limit of an inverse system of unipotent group schemes.

\subsection{The Dual Group and Coordinate Ring}
\label{TheDualGroup}

If $F$ is a field and $G$ is an affine group scheme over $F$, a {\em character} of $G$ is a homomorphism $\alpha \colon G \to \G_m$ of group schemes. The set $G^\vee$ of characters of $G$ forms a group. Given characters $\alpha,\beta \colon G \to \G_m$, let the product $\alpha\beta$ be the character given by composition $G \overset{(\alpha,\beta)}{\longrightarrow} \G_m \times \G_m \to \G_m$, where the map $\G_m \times \G_m \to \G_m$ is multiplication. The group $G^\vee$ is known as the {\em dual group} of $G$. The elements of $G^\vee$ correspond to the one-dimensional representations of $G$.

Let $F$ be a field, and let $G$ be an affine group scheme over $F$. The coproduct $\O(G) \to \O(G) \otimes_F \O(G)$ gives $\O(G)$ the structure of a (right) representation of $G$. Each character $\alpha \colon G \to \G_m$ of $G$ corresponds to a Hopf algebra homomorphism $\alpha^* \colon \C[q^{\pm 1}] \to \O(G)$. There is an injective group homomorphism
$$
G^\vee \hookrightarrow \O(G)^\times
$$
defined by $\alpha \mapsto \alpha^*(q)$. Thus, we may view $G^\vee$ as a subset of $\O(G)$.

\subsection{Diagonalizable Group Schemes}
\label{DiagGroups}

An affine algebraic group scheme over a field $F$ is {\em diagonalizable} if it is isomorphic to a group subscheme of $\G_m^N$ for some positive integer $N$. Over an algebraically closed field, every diagonalizable group scheme is reductive. It is well known that if $D$ is a diagonalizable group scheme over a field, then there is an isomorphism
$$D \overset{\cong}{\longrightarrow} \G_m^s \times \bmu_{r_1} \times \dots \times \bmu_{r_t}$$
of affine algebraic group schemes, where $\bmu_{r_j}$ is the group scheme of $r_j$-th roots of unity, the $r_j$ are integers greater than 1 such that $r_j | r_{j+1}$, and $s$ is a nonnegative integer \cite[Section 2.2]{Waterhouse}.

The irreducible rational representations of a diagonalizable group $D$ are all one-dimensional. Thus, they are in bijective correspondence with the dual group $D^\vee$. If $F$ is an algebraically closed field and $D$ is a group subscheme of $\G_m^N$, then the induced map $[\G_m^N]^\vee \to D^\vee$ is surjective \cite[Page 111]{Borel}. Consequently, every irreducible representation of $D$ extends to an irreducible representation of $\G_m^N$.

\section{Relative Malcev Completion}
\label{RelCompBasics}

\noindent The relative Malcev completion of a discrete group $\pi$ with respect to a reductive representation over a field $F$ is a proalgebraic group scheme that generalizes the Malcev (or unipotent) completion of $\pi$. Here, we restrict to the field $F = \C$. The prounipotent radical of the relative completion is determined by its pronilpotent Lie algebra. We describe the homology of this Lie algebra and give a commutative differential graded algebra that determines this Lie algebra via rational homotopy theory. Finally, we show that if $\pi$ is the fundamental group of the complement of an affine hyperplane arrangement and $\T$ is the character torus of $\pi$, then the relative completion is generally constant over $\T^N$.

\subsection{Relative Malcev Completion}
\label{RelCompDef}

\noindent The concept of relative Malcev completion is due to Deligne. It and generalizations of it have been extensively developed by Hain (\cite{HainCompletions}, \cite{HainDeRham}, \cite{HainTorelli}) and Hain and Matsumoto \cite{HainMatsu}. The data for the relative Malcev completion over $\C$ are
\begin{itemize}
\item a discrete group $\pi$;
\item an algebraic group scheme $G$ over $\C$;
\item a Zariski dense homomorphism $\rho \colon \pi \to G(\C)$.
\end{itemize}

The {\em Malcev completion of $\pi$ relative to $\rho$} is an extension $1 \to \U \to \mathcal{G} \to G \to 1$ in the category of proalgebraic group schemes, where $\U$ is prounipotent. It is equipped with a homomorphism $\theta_\rho \colon \pi \to \mathcal{G}(\C)$ lifting $\rho$, and it is characterized by the following universal property. If $E$ is an extension of $G$ by a prounipotent group scheme and $\theta \colon \pi \to E(\C)$ is a homomorphism lifting $\rho$, then there is a unique homomorphism $\mathcal{G} \to E$ such that the diagrams
$$
\xymatrix{ & \mathcal{G} \ar[dl] \ar[d] \\ E \ar[r] & G} \hspace{1 in}
\xymatrix{\pi \ar[r]^{\theta_\rho} \ar[d]_{\theta} & \mathcal{G}(\C) \ar[dl] \\ E(\C) &}
$$
commute.  If $\theta$ is Zariski dense, then the homomorphism $\mathcal{G} \to E$ is surjective. When $G$ is reductive, $\U$ is called the {\em prounipotent radical} of $\mathcal{G}$.

\vspace{.1 in}

In what follows, we will sometimes suppress the word ``Malcev" and refer to $\mathcal{G}$ as the completion of $\pi$ relative to $\rho$ or simply as the relative completion. To see that the relative completion exists, consider all extensions
\begin{equation}
\label{exts}
1 \to U \to E \to G \to 1
\end{equation}
of $G$ by a unipotent group scheme $U$ that are equipped with a Zariski dense homomorphism $\theta \colon \pi \to E(\C)$ that lifts $\rho$:
$$
\xymatrix{\pi \ar[d]_{\theta} \ar[dr]^{\rho} \\ E(\C) \ar[r] & G(\C).}
$$

Given two extensions $E_1$ and $E_2$ of $G$ by unipotent group schemes with lifts $\theta_1 \colon \pi \to E_1(\C)$ and $\theta_2 \colon \pi \to E_2(\C)$ of $\rho$, a {\em morphism} from the first to the second is a homomorphism $E_1 \to E_2$ such that the diagrams
$$
\xymatrix{ & E_1 \ar[dl] \ar[d] \\ E_2 \ar[r] & G} \hspace{1 in}
\xymatrix{\pi \ar[r]^{\theta_1} \ar[d]_{\theta_2} & E_1(\C)\ar[dl] \\ E_2(\C) &}
$$
commute. One can define a partial order on these extensions. Given two such extensions $E_1$ and $E_2$, we say that $E_1 \succeq E_2$ if there is a surjective morphism $E_1 \to E_2$. A proof of the following proposition can be found in \cite{HainCompletions}.
\begin{proposition}
The set of extensions (\ref{exts}) forms an inverse system, and the completion of $\pi$ relative to $\rho$ is the inverse limit
$$
\mathcal{G} = \underleftarrow{\lim} \hspace{.05 in} E
$$
taken over all such extensions. \qed
\end{proposition}
This is a proalgebraic group scheme, and the Zariski dense homomorphisms $\theta \colon \pi \to E(\C)$ induce a Zariski dense homomorphism $\theta_\rho \colon \pi \to \mathcal{G}(\C)$ that lifts $\rho$. The prounipotent group scheme $\mathcal{U}$ is given by
$$
\U = \underleftarrow{\lim} \hspace{.05 in} U
$$
taken over all extensions (\ref{exts}). If $G$ is reductive, then $\mathcal{U}$ is called the {\em prounipotent radical} of $\mathcal{G}$.

If the homomorphism $\rho \colon \pi \to G(\C)$ is not Zariski dense, we can define the completion of $\pi$ relative to $\rho$ as follows. Let $\overline{\im\rho}$ denote the Zariski closure of the image of $\rho$ in $G$. This is the smallest algebraic group subscheme of $G$ whose group of $\C$-rational points contains the image of $\rho$. The induced homomorphism $\pi \overset{\rho}{\longrightarrow} \overline{\im\rho}(\C)$ is Zariski dense, and we define the {\em completion of $\pi$ relative to $\rho$} to be the completion of $\pi$ with respect to this map. This is an extension of $\overline{\im\rho}$ by a prounipotent group scheme.

\begin{example}
Let $X$ be the complement of a hyperplane arrangement in a complex vector space $V$, and let $$\vrho \colon \pi_1(X,x_0) \to (\C^*)^N$$ be a homomorphism. Let $D_\vrho$ denote the Zariski closure of the image of $\vrho$ in $\G_m^N$. This is a group subscheme of $\G_m^N$. Let $\S_\vrho$ denote the completion of $\pi_1(X,x_0)$ relative to $\vrho$, and let $\U_\vrho$ denote its prounipotent radical. There is a short exact sequence
$$
1 \longrightarrow \U_\vrho \longrightarrow \S_\vrho \longrightarrow D_\vrho \longrightarrow 1
$$
of proalgebraic group schemes. In the following sections, we develop tools that will help us to understand the completion $\S_\vrho$.
\end{example}

\subsection{Unipotent Completion}

The completion of $\pi$ relative to the trivial representation $\pi \to \{1\}$ is the standard Malcev (i.e. unipotent) completion $\pi^{\un}$ of $\pi$. It is the prounipotent group scheme that is the inverse limit of all unipotent group schemes $U$ over $\C$ for which there is a Zariski dense homomorphism $\pi \to U(\C)$. The homomorphisms $\pi \to U(\C)$ induce a Zariski-dense homomorphism $\pi \to \pi^{\un}(\C)$. This agrees with other standard definitions \cite[Section 3]{HainCompletions}.

If $X$ is the complement of an arrangement of hyperplanes in a complex vector space and $\pi = \pi_1(X,x_0)$, then $\pi^{\un}$ is the is the unique prounipotent group scheme whose Lie algebra is the completion $\h^\wedge$ of Kohno's \cite{Kohno} holonomy Lie algebra
$$
\h = \frac{\LL(H_1(X,\C))}{\langle\im\partial\rangle},
$$
where $\partial \colon H_2(X,\C) \to \bigwedge^2 H_1(X,\C)$ is the dual of the cup product, and $\langle\im\partial\rangle$ is the ideal generated by the image of $\partial$.

\subsection{Properties of Relative Malcev Completion}
\label{RMCProperties}

Some of the basic properties presented here can be found in \cite{HainCompletions} and \cite{HainDeRham}. Suppose that $\rho \colon \pi \to G(\C)$ is Zariski dense, where $G$ is an algebraic group scheme over $\C$. Let $H$ be an algebraic group scheme with a surjection $G \to H$. Suppose that $E$ is an extension of $H$ by a prounipotent group scheme and that $\theta \colon \pi \to E(\C)$ is a homomorphism such that the diagram
$$
\xymatrix{\pi \ar[r]^\rho \ar[d]_{\theta} & G(\C)\ar[d] \\ E(\C) \ar[r] & H(\C)}
$$
commutes. The following proposition can easily be proved using the universal property of the relative completion $\mathcal{G}$.
\begin{proposition}
\label{UnivPropGen}
There is a unique homomorphism $\mathcal{G} \to E$ such that the diagrams
$$
\xymatrix{ \mathcal{G} \ar[d] \ar[r]& G \ar[d] \\ E \ar[r] & H} \hspace{1 in}
\xymatrix{\pi \ar[r]^{\theta_\rho} \ar[dr]_\theta & \mathcal{G}(\C) \ar[d] \\ & E(\C)}
$$
commute. The homomorphism $\mathcal{G} \to E$ is surjective if $\theta$ is Zariski dense. \qed
\end{proposition}

\begin{corollary}
If $\pi^{\un}$ is the Malcev completion of $\pi$, then there is a unique surjection $\mathcal{G} \to \pi^{\un}$ of group schemes such that the diagrams
$$
\xymatrix{ & \mathcal{G} \ar[dl] \ar[d] \\ \pi^{\un} \ar[r] & G} \hspace{1 in}
\xymatrix{\pi \ar[r]^{\theta_\rho} \ar[d] & \mathcal{G}(\C)\ar[dl] \\ \pi^{\un}(\C)}
$$
commute. \qed
\end{corollary}

This corollary holds even if $\rho \colon \pi \to G$ is not Zariski dense, because the Zariski closure $\overline{\im\rho}$ always surjects onto the trivial group scheme.

\begin{example}
Suppose that $X$ is the complement of a hyperplane arrangement in a complex vector space $V$ and that
$$
\vrho \colon \pi_1(X,x_0) \to (\C^*)^N
$$
is a homomorphism. Let $\S_\vrho$ denote the completion of $\pi_1(X,x_0)$ relative to $\vrho$. The corollary gives a surjection $\S_\vrho(\C) \to \pi_1(X,x_0)^{\un}$. When $\vrho$ is the trivial representation, this is an isomorphism. Thus, $\S_\vrho(\C)$ is a kind of ``deformation" of $\pi_1(X,x_0)^{\un}$ over $\T^N$ in a sense that we will make more precise in Section \ref{SectionY}.
\end{example}

When computing the Malcev completion relative to a representation $\vrho \colon \pi \to G(\C)$, the group scheme $G$ can be replaced by its maximal reductive quotient $R$. The composition $\pi \overset{\rho}{\longrightarrow} G \longrightarrow R$ is Zariski dense. Let $\mathcal{R}$ denote the completion of $\pi$ relative to this composition. Then Proposition \ref{UnivPropGen} gives a surjection $\mathcal{G} \to \mathcal{R}$ such that the diagrams
$$
\xymatrix{ & \mathcal{G} \ar[dl] \ar[d] \\ \mathcal{R} \ar[r] & G} \hspace{1 in}
\xymatrix{\pi \ar[r]^{\theta_\rho} \ar[d]_{\theta} & \mathcal{G}(\C)\ar[dl] \\ \mathcal{R}(\C)}
$$
commute. The proof of the following proposition is left as an exercise.

\begin{proposition}
The surjection $\mathcal{G} \to \mathcal{R}$ is an isomorphism of proalgebraic group schemes. \qed
\end{proposition}

This proposition allows us to replace $G$ with its maximal reductive quotient when studying the relative completion. Therefore, assume that $R$ is a reductive algebraic group scheme over $\C$ and that $\rho \colon \pi \to R(\C)$ is a Zariski dense homomorphism. Let $\mathcal{G}$ denote the completion of $\pi$ relative to $\rho$, and let $\mathcal{U}$ denote its prounipotent radical. Then there is an extension
$$
1 \longrightarrow \U \longrightarrow \mathcal{G} \longrightarrow R \longrightarrow 1
$$
in the category of proalgebraic group schemes. The following proposition generalizes the Levi decomposition for algebraic groups \cite[Page 158]{Borel}.
\begin{proposition}[Hain, \cite{HainCompletions}]
The sequence $1 \to \U \to \mathcal{G} \to R \to 1$ splits. \qed
\end{proposition}
This proposition implies that there is an isomorphism $\mathcal{G} \cong R \ltimes \U$ of proalgebraic group schemes. The relative completion $\mathcal{G}$ is therefore determined by its prounipotent radical $\U$ and the action of $R$ on $\U$.

Over an algebraically closed field of characteristic zero, the exponential and logarithm maps determine an equivalence of categories between prounipotent algebraic group schemes and pronilpotent Lie algebras. Thus, there is a unique pronilpotent Lie algebra $\u$ such that
$$
\U = \exp \u.
$$
The conjugation action of $R$ on $\U$ gives $\u$ the structure of a right representation of $R$.

\subsection{The Completion Relative to a Diagonal Representation}
\label{CompRelDiag}

\noindent Let $X$ be a smooth manifold, and set $\pi = \pi_1(X,x_0)$. Suppose that $\vrho \colon \pi \to (\C^*)^N$ is a representation. Let $D_\vrho$ denote the Zariski closure of the image of $\vrho$ in $\G_m^N$. Then $D_\vrho$ is a reductive algebraic group scheme. Let $\S_\vrho$ denote the completion of $\pi$ relative to $\vrho$, and let $\U_\vrho$ denote its prounipotent radical. Note that $\S_\vrho$ is prosolvable, as it is an extension
$$
1 \to \U_\vrho \to \S_\vrho \to D_\vrho \to 1,
$$
and $D_\vrho$ is diagonalizable. The irreducible representations of $D_\vrho$ correspond to elements of the dual group $D_\vrho^\vee$. Each $\alpha \in D_\vrho^\vee$ determines a one-dimensional irreducible representation $V_\alpha$ of $D_\vrho$ and a rank-one local system $\VV_\alpha$ on $X$ whose monodromy is given by the character $\alpha\circ\vrho$ of $\pi$.

Let $\u_\vrho$ denote the pronilpotent Lie algebra of $\U_\vrho$. Then $\u_\vrho$ is a representation of $D_\vrho$. Hain \cite[Proof of Proposition 7.1]{HainTorelli} shows that the $D_\vrho$-module structure on $\u_\vrho$ induces an $D_\vrho$-module structure on $H_\bullet(\u_\vrho)$, the Lie algebra homology of $\u_\vrho$. Recall that $H_1(\u_\vrho)$ is defined to the abelianization of $\u_\vrho$. A more general version of the following theorem is due to Hain and Matsumoto \cite[Theorems 4.8 and 4.9]{HainMatsu}.

\begin{theorem}
\label{HainMatsuTheorem}
There is a $D_\vrho$-equivariant isomorphism
$$
\prod_{\alpha \in D_\vrho^\vee} H^1(X,\VV_{\alpha})^*\otimes V_{\alpha} \hspace{.05 in} \cong \hspace{.05 in} H_1(\u_\vrho)
$$
and a $D_\vrho$-equivariant surjection
$$
\prod_{\alpha \in D_\vrho^\vee} H^2(X,\VV_{\alpha})^* \otimes V_{\alpha} \longrightarrow H_2(\u_\vrho).
$$ \qed
\end{theorem}

\begin{example}
\label{uHomNotDense}
Choose an isomorphism
$$
D_\vrho \overset{\phi}{\longrightarrow} \G_m^s \times \bmu_{r_1} \times \dots \times \bmu_{r_t}
$$
of algebraic group schemes, where $\bmu_{r_j}$ is the group scheme of $r_j$-th roots of unity, the $r_j$ are integers greater than 1 such that $r_j | r_{j+1}$, and $s$ is a nonnegative integer. Choose characters $q_1,\dots,q_s,\mathfrak{q}_1,\dots,\mathfrak{q}_t$ on $D_\vrho$ such that $\mathfrak{q}_j$ has order $r_j$ and the isomorphism $\phi$ is given by
$$
\phi = (q_1,\dots,q_s,\mathfrak{q}_1,\dots,\mathfrak{q}_t).
$$
Define characters $\rho_1,\dots,\rho_s, \varrho_1,\dots,\varrho_t$ of $\pi$ by $\rho_j = q_j \circ \vrho$ and $\varrho_j = \mathfrak{q}_j\circ \vrho$.

The irreducible representations of $D_\vrho$ correspond to elements in the dual group $D_\vrho^\vee$. The dual $\phi^\vee$ of $\phi$ is an isomorphism:
\begin{equation}
\label{IsoDual}
\Z^s \times (\Z/r_1\Z) \times \dots \times (\Z/r_t\Z) \overset{\phi^\vee}{\longrightarrow}D_\vrho^\vee.
\end{equation}
Given an element $\k = (k_1,\dots,k_s,\kappa_1,\dots,\kappa_t)$ of $\Z^s \times (\Z/r_1\Z) \times \dots \times (\Z/r_t\Z)$, let $\Lr_{\k}$ denote the one-dimensional representation of $D_\vrho$ given by the character $q_1^{k_1}\cdots q_s^{k_s}\cdot \mathfrak{q}_1^{\kappa_1}\cdots \mathfrak{q}_t^{\kappa_t}$, and let $\L_{\vrho^\k}$ denote the rank-one local system on $X$ with monodromy given by the character $\rho_1^{k_1}\cdots\rho_M^{k_s}\varrho_1^{\kappa_1}\cdots\varrho_t^{\kappa_t}$ of $\pi$. Then Theorem \ref{HainMatsuTheorem} implies that there is a $D_\vrho$-equivariant isomorphism
$$
\prod_{\k} H^1(X,\L_{\vrho^\k})^*\otimes \Lr_{\k} \hspace{.05 in} \cong \hspace{.05 in} H_1(\u_\vrho)
$$
and a $D_\vrho$-equivariant surjection
$$
\prod_{\k} H^2(X,\L_{\vrho^\k})^* \otimes \Lr_{\k} \longrightarrow H_2(\u_\vrho),
$$
where the products are taken over the elements $\k$ of $\Z^s\times (\Z/r_1\Z) \times \dots \times (\Z/r_t\Z)$.
\end{example}

\begin{example}
\label{uHomDense}
Suppose that  $\vrho = (\rho_1,\dots,\rho_N) \colon \pi \to (\C^*)^N$ has Zariski dense image in $\G_m^N$. That is, $D_\vrho = \G_m^N$. The irreducible representations of $\G_m^N$ correspond to characters, which are in bijective correspondence with $\Z^N$. Let $q_j$ denote the $j$-th standard character on $\G_m^N$. Given $\k = (k_1,\dots,k_N) \in \Z^N$, let $\Lr_{\k}$ denote the irreducible representation of $\G_m^N$ given by the character $q_1^{k_1}\cdots q_N^{k_N}$, and let $\L_{\vrho^\k}$ denote the rank-one local system on $X$ with monodromy $\rho_1^{k_1}\cdots \rho_N^{k_N}$. Then Theorem \ref{HainMatsuTheorem} implies that there is a $\G_m^N$-equivariant isomorphism
$$
\prod_{\k\in\Z^N} H^1(X,\L_{\vrho^\k})^*\otimes \Lr_{\k} \hspace{.05 in} \cong \hspace{.05 in} H_1(\u_\vrho)
$$
and a $\G_m^N$-equivariant surjection
$$
\prod_{\k\in\Z^N} H^2(X,\L_{\vrho^\k})^* \otimes \Lr_{\k} \longrightarrow H_2(\u_\vrho).
$$
\end{example}

\begin{remark}
Suppose that $X$ is the complement of a hyperplane arrangement in a complex vector space, and let $\T = H^1(X,\C^*)$ be the character torus. Each $\vrho \in \T^N$ can be viewed as a representation $\pi \to (\C^*)^N$. We will show that $\S_\vrho \cong D_\vrho \times \pi^{\un}$ for general $\vrho \in \T^N$, where $\pi^{\un}$ is the Malcev completion of $\pi$. To prove this, we will use the de Rham theory of relative completion. This result is nontrivial, and in fact it fails if $\vrho$ lies in the characteristic variety $\V_{N,1}^1(X)$.
\end{remark}

\subsection{A General Construction}
\label{GenConstructionAlg}

\noindent Suppose that $Z$ is an additive abelian group and that for each $\k \in Z$, $A_\k$ is a co-complex over $\C$ with differential $\nabla_\k$. Suppose further that there are chain maps $A_{\k_1} \otimes A_{\k_1} \to A_{\k_1 + \k_2}$ that are associative in the sense that $(a_1 \otimes a_2) \otimes a_3$ and $a \otimes (a_2 \otimes a_3)$ have the same image in $A_{\k_1 + \k_2 + \k_3}$, where $a_j \in A_{\k_j}$. This implies that there is a multiplication $A_0 \otimes A_0 \to A_0$, which gives $A_0$ the structure of an algebra over $\C$. We assume that $\C$ is a subalgebra of $A_0$.

We write $a_1\cdot a_2$ for the image of $a_1 \otimes a_2$ in $A_{\k_1 + \k_2}$. Assume that the differentials $\nabla_\k$ satisfy $$\nabla_{\k_1 + \k_2}(a_1\cdot a_2) = \nabla_{\k_1}(a_1)\cdot a_2 + (-1)^{\deg a_1} a_1 \cdot \nabla_{\k_2}(a_2).$$ The direct sum
$$
\bigoplus_{\k \in Z} A_\k
$$
is a graded $\C$-algebra, the grading determined by the degrees in the $A_\k$. The multiplication is defined componentwise by the chain maps $A_{\k_1} \otimes A_{\k_2} \to A_{\k_1 + \k_2}$. Define a differential $\nabla$ on this algebra by setting $\nabla(a_\k) = \nabla_\k(a_\k)$ for $a_\k \in A_\k$. The proof of the following proposition is trivial.

\begin{proposition}
The algebra $$\bigoplus_{\k \in Z} A_\k$$ is a commutative differential graded algebra over $\C$ with differential $\nabla$. \qed
\end{proposition}

\subsection{De Rham Theory of Relative Completion}
\label{CommDiffE}

\noindent Suppose that $X$ is a smooth manifold, and set $\pi = \pi_1(X,x_0)$. Let $\vrho \colon \pi \to (\C^*)^N$ be a representation, and let $D_\vrho$ denote the Zariski closure of the image of $\vrho$ in $\G_m^N$. The irreducible representations of $D_\vrho$ are given by the elements of the dual group $D_\vrho^\vee$. Given a character $\alpha \in D_\vrho^\vee$, let $V_\alpha$ be the corresponding irreducible representation of $D_\vrho$, and let $\VV_\alpha$ denote a rank-one local system on $X$ with monodromy given by the character $\alpha\circ\vrho$ of $\pi$. Let $\nabla_\alpha$ denote the differential on $E^\bullet(X,\VV_\alpha)$. For characters $\alpha$ and $\beta$ on $D_\vrho$, the cup product is a chain map
$$
E^\bullet(X,\VV_\alpha) \otimes E^\bullet(X,\VV_\beta) \overset{\wedge}{\longrightarrow} E^\bullet(X,\VV_{\alpha\beta}).
$$
In addition, if $\psi_\alpha \in E^j(X,\VV_\alpha)$ and $\psi_\beta \in E^\bullet(X,\VV_\beta)$, we have
$$
\nabla_{\alpha\beta}(\psi_\alpha \wedge \psi_\beta) = \nabla_\alpha(\psi_\alpha)\wedge \psi_\beta + (-1)^j \psi_\alpha \wedge \nabla_\beta(\psi_\beta).
$$

The construction in Section \ref{GenConstructionAlg} therefore allows us to define a commutative differential graded algebra $E^\bullet(X,\O_\vrho)$ by
$$
E^\bullet(X,\O_\vrho) = \bigoplus_{\alpha \in D_\vrho^\vee} E^\bullet(X,\VV_{\alpha}) \otimes V_\alpha^*\hspace{.05 in} .
$$
See \cite[Section 4]{HainDeRham} for an explanation of the notation $E^\bullet(X,\O_\vrho)$. This algebra is a $D_\vrho$-module. The differential is defined componentwise by the differential on $E^\bullet(X,\VV_\alpha)$. The product of an element in $E^\bullet(X,\VV_\alpha) \otimes V_\alpha^*$ with an element in $E^\bullet(X,\VV_\beta) \otimes V_\beta^*$ lies in $E^\bullet(X,\VV_{\alpha\beta}) \otimes V_{\alpha\beta}^*$.
The cohomology ring of $E^\bullet(X,\O_\vrho)$ is denoted $H^\bullet(X,\O_\vrho)$. We have
$$
H^\bullet(X,\O_\vrho) = \bigoplus_{\alpha \in D_\vrho^\vee} H^\bullet(X,\VV_\alpha) \otimes V_\alpha^*\hspace{.05 in} .
$$
Thus, this cohomology ring is a $D_\vrho$-module as well.

\begin{important}
The differential graded algebra $E^\bullet(X,\O_\vrho)$ determines the pronilpotent Lie algebra $\u_\vrho$ via the methods of rational homotopy theory. The $D_\vrho$-module structure on $E^\bullet(X,\O_\vrho)$ determines the $D_\vrho$-module structure on $\u_\vrho$. For a proof, see \cite{HainDeRham}. Thus, the Lie algebra $\u_\vrho$ is quadratically presented if and only if $E^\bullet(X,\O_\vrho)$ is 1-formal.
\end{important}

If $\vrho$ is the trivial homomorphism, then $D_\vrho$ is the trivial group scheme and $E^\bullet(X,\O_\vrho)$ is the de Rham complex $E^\bullet(X)$. When $X$ is the complement of a hyperplane arrangement, the Lie algebra $\u_\vrho = \u_1$ is therefore the completion of Kohno's \cite{Kohno} holonomy Lie algebra, since $X$ is 1-formal. We will show in Theorem \ref{MyTheorem} that when $X$ is the complement of the braid arrangement $\mathcal{B}$ in $\C^2$, there is a representation $\vrho \colon \pi \to (\C^*)^2$ for which $E^\bullet(X,\O_\vrho)$ is not 1-formal.

\begin{example}
\label{ExNotDense}
Choose an isomorphism
$$
D_\vrho \overset{\phi}{\longrightarrow} \G_m^s \times \bmu_{r_1} \times \dots \times \bmu_{r_t}
$$
of algebraic group schemes, where $\bmu_{r_j}$ is the group scheme of $r_j$-th roots of unity, the $r_j$ are integers greater than 1 such that $r_j | r_{j+1}$, and $s$ is a nonnegative integer. Choose characters $q_1,\dots,q_s,\mathfrak{q}_1,\dots,\mathfrak{q}_t$ of $D_\vrho$ such that $\mathfrak{q}_j$ has order $r_j$ and the isomorphism $\phi$ is given by
$$
\phi = (q_1,\dots,q_s,\mathfrak{q}_1,\dots,\mathfrak{q}_t).
$$
Set $\rho_j = q_j \circ \vrho$ and $\varrho_j = \mathfrak{q}_j\circ \vrho$. Given
$$
\k = (k_1,\dots,k_s,\kappa_1,\dots,\kappa_t) \hspace{.08 in}\in\hspace{.03 in} \Z^s \times (\Z/r_1\Z) \times \dots \times (\Z/r_t\Z),
$$
Let $\L_{\vrho^\k}$ be the rank-one local system on $X$ with monodromy $\rho_1^{k_1}\cdots\rho_s^{k_s}\varrho_1^{\kappa_1}\cdots\varrho_t^{\kappa_t}$.
The commutative differential graded algebra $E^\bullet(X,\O_\vrho)$ has the following description as a $D_\vrho$-module.
$$
E^\bullet(X,\O_\vrho) = \bigoplus_{\k} E^\bullet(X,\L_{\vrho^\k}) \cdot q_1^{-k_1}\cdots q_s^{-k_s}\cdot \mathfrak{q}_1^{-\kappa_1}\cdots\mathfrak{q}_t^{-\kappa_t}
$$
The $q_j$ and $\mathfrak{q}_j$ determine the $D_\vrho$-module structure on $E^\bullet(X,\O_\vrho)$ via the (right) action of $D_\vrho$ on its coordinate ring. The direct sum is taken over the elements $\k = (k_1,\dots,k_s,\kappa_1,\dots,\kappa_t)$ of $\Z^s \times (\Z/r_1\Z) \times \dots \times (\Z/r_t\Z).$
\end{example}

\begin{example}
\label{IsDenseExample}
Suppose now that $\vrho = (\rho_1,\dots,\rho_N) \colon \pi \to (\C^*)^N$ has Zariski dense image in $\G_m^N$. Then each $\rho_j$ is a character of $\pi$. Define $q_j$ to be the $j$-th standard character on $\G_m^N$. For $\k \in \Z^N$, let $\L_{\vrho^\k}$ denote the rank one local system on $X$ with monodromy $\rho_1^{k_1}\cdots\rho_N^{k_N}$. Then $E^\bullet(X,\O_\vrho)$ has the following description as a $\G_m^N$-module.
$$
E^\bullet(X,\O_\vrho) = \bigoplus_{\k \in \Z^N} E^\bullet(X,\L_{\vrho^\k}) \cdot q_1^{-k_1}\cdots q_N^{-k_N}
$$
Here, the $q_j$ determine the $\G_m^N$-module structure on $E^\bullet(X,\O_\vrho)$ via the (right) action of $\G_m^N$ on its coordinate ring.
\end{example}

\subsection{Constancy of Relative Completion}
\label{FirstConstancy}

\noindent Theorem \ref{HainMatsuTheorem} suggests a relationship between the relative completion and the characteristic varieties $\V_{N,m}^i(X)$. In this section, we prove the simplest form of this relationship.

The following lemma is standard in rational homotopy theory. The concepts and the idea were developed by Stallings, Chen, and Sullivan. A proof can be found in \cite[Corollary 3.2]{HainRelative}.
\begin{lemma}
\label{usefullemma}
A homomorphism $\alpha \colon \mathfrak{n}_1 \to \mathfrak{n}_2$ of pronilpotent Lie algebras is an isomorphism if and only if it induces an isomorphism $H_1(\mathfrak{n}_1) \to H_1(\mathfrak{n}_2)$ and a surjection $H_2(\mathfrak{n}_1) \to H_2(\mathfrak{n}_2)$. \qed
\end{lemma}

Let $E_1^\bullet$ and $E_2^\bullet$ be commutative differential graded algebras which determine the pronilpotent Lie algebras $\mathfrak{n}_1$ and $\mathfrak{n}_2$ (respectively) via the methods of rational homotopy theory. The next lemma follows directly from Lemma \ref{usefullemma}.
\begin{lemma}
\label{usefulcorollary}
If $E_1^\bullet \to E_2^\bullet$ is a homomorphism of commutative differential graded algebras, then the induced map $\mathfrak{n}_1 \to \mathfrak{n}_2$ is an isomorphism if and only if the maps $H^1(E_1^\bullet) \to H^1(E_2^\bullet)$ and $H^2(E_1^\bullet) \to H^2(E_2^\bullet)$ are an isomorphism and an injection, respectively. \qed
\end{lemma}

Let $X$ be the complement of $n$ hyperplanes in a complex vector space $V$, and let $\T = H^1(X,\C^*)$ denote the character torus. Set $\pi = \pi_1(X,x_0)$. An element $\vrho \in \T^N$ may be viewed as a homomorphism $\pi \to (\C^*)^N$. Let $D_\vrho$ denote the Zariski closure of the image of $\vrho$ in $\G_m^N$. Let $\pi^{\un}$ denote the Malcev completion of $\pi$. The universal property of the relative Malcev completion $\S_\vrho$ gives a unique homomorphism $\S_\vrho \to D_\vrho \times \pi^{\un}$ of proalgebraic group schemes such that the diagrams
$$
\xymatrix{ & \mathcal{S}_\vrho \ar[dl] \ar[d] \\ D_\vrho \times \pi^{\un} \ar[r] & D_\vrho} \hspace{1 in}
\xymatrix{\pi \ar[r]^{\theta_\rho} \ar[d] & \S_\vrho(\C) \ar[dl] \\ D_\vrho(\C) \times \pi^{\un}(\C)&}
$$
commute.

Recall that a property is said to hold for a {\em general point} of an irreducible affine variety $\mathbf{V}$ if there are countably many proper subvarieties $\Sigma_1,\Sigma_2,\dots$ of $\mathbf{V}$ such that the property holds for each point not lying in any $\Sigma_\k$.

\begin{theorem}
\label{IsUnipTheorem1}
If $X$ is the complement of an arrangement of $n$ hyperplanes in a complex vector space and two distinct hyperplanes intersect, then for general $\vrho \in \T^N$, the homomorphism $\S_\vrho \to D_\vrho \times \pi^{\un}$ is an isomorphism.
\end{theorem}

\begin{remark}
If $\vrho$ lies in the characteristic variety $\V_{N,1}^1(X)$, then $\S_\vrho$ is not isomorphic to $D_\vrho \times \pi_1(X,x_0)^{\un}$. We show in Section \ref{SectionY} that $\S_\vrho$ is generally constant on each irreducible component of $V_{N,1}^1(X)$.
\end{remark}

\begin{proof}[Proof of Theorem \ref{IsUnipTheorem1}] The unipotent radical of $D_\vrho \times \pi^{\un}$ is simply $\pi^{\un}$. The Lie algebra $\u_1$ of $\pi^{\un}$ is Kohno's \cite{Kohno} holonomy Lie algebra, since the space $X$ is 1-formal. The conjugation action of $D_\vrho$ on $\pi^{\un}$ given by the extension $1 \to \pi^{\un} \to D_\vrho \times \U_1 \to D_\vrho \to 1$ is trivial. Hence, $D_\vrho$ acts trivially on $\u_1$. The induced map $\U_\vrho \to \pi^{\un}$ is a $D_\vrho$-equivariant isomorphism if and only if the corresponding Lie algebra homomorphism $\u_\vrho \to \u_1$ is an isomorphism.

The uniqueness of the maps $\S_\vrho \to D_\vrho \times \U_1$ and $\U_\vrho \to \U_1$ implies that the map $\u_\vrho \to \u_1$ is induced by the canonical inclusion
$$
E^\bullet(X) \longrightarrow E^\bullet(X,\O_\vrho)
$$
of commutative differential graded algebras. By Lemma \ref{usefulcorollary}, it suffices to show that for general $\vrho \in \T^N$, the map $H^1(X,\C) \to H^1(X,\O_\vrho)$ is an isomorphism and the map $H^2(X,\C) \to H^2(X,\O_\vrho)$ is injective. Recall that the definition of $E^\bullet(X,\O_\vrho)$ implies that
$$
H^\bullet(X,\O_\vrho) = \bigoplus_{\alpha \in D_\vrho^\vee} H^\bullet(X,\VV_\alpha) \otimes V_\alpha^*,
$$
where $V_\alpha$ is the representation of $D_\vrho$ given by the character $\alpha$, and $\VV_\alpha$ is the local system on $X$ with monodromy $\alpha\circ\vrho$. The fact that the induced map $H^2(X,\C) \to H^2(X,\O_\vrho)$ is injective is therefore trivial.

Thus, we only need to show that the induced map $H^1(X,\C) \to H^1(X,\O_\vrho)$ is an isomorphism for general $\vrho \in \T^N$. By the definition of $E^\bullet(X,\O_\vrho)$, it suffices to show that for general $\vrho \in \T^N$, we have $H^1(X,\VV_\alpha) = 0$ for all nontrivial characters $\alpha$ of $D_\vrho$.

Set $\vrho = (\rho_1,\dots,\rho_N)$, where each $\rho_j \in \T$. For $\k \in \Z^N$, let $\L_{\rho_1^{k_1}\dots\rho_N^{k_N}}$ denote the rank-one local system on $X$ with monodromy $\rho_1^{k_1}\cdots\rho_N^{k_N}$. It follows directly from Arapura's theorem (Theorem \ref{ArapuraTheorem}) and Yuzvinsky's result that $\V_{N,1}^1(X) \neq \T^N$ (Section \ref{CharVar}) that for general $\vrho \in \T^N$, we have $H^1(X,\L_{\rho_1^{k_1}\dots\rho_N^{k_N}}) = 0$ for all nonzero $\k \in \Z^N$. For any such $\vrho$, if $\alpha$ is a nontrivial character on $D_\vrho$, then $\alpha\circ\vrho \colon \pi \to \G_m^N(\C) = (\C^*)^N$ must be a nontrivial character of $\pi$, as $\vrho$ has Zariski dense image in $D_\vrho$. Thus, $\VV_\alpha$ is a nontrivial local system on $X$. Recall that every character on $D_\vrho$ extends to a character on $\G_m^N$. Each character of $\G_m^N$ is given by $q_1^{k_1}\cdots q_N^{k_N}$ for some $\k \in \Z^N$, where $q_j$ is the $j$-th standard character on $\G_m^N$. Since $\alpha\circ\vrho$ is a nontrivial character on $\pi$, this implies that there is some nonzero $\k \in \Z^N$ such that $\VV_\alpha$ is isomorphic to $\L_{\rho_1^{k_1}\cdots\rho_N^{k_N}}$ as local systems on $X$. Our choice of $\vrho$ then implies that $H^1(X,\VV_\alpha) = 0$. This completes the proof.
\end{proof}

\section{A Combinatorial Approximation of $E^\bullet(X,\O_\vrho)$}
\label{DeRham}

\noindent In this section, we assemble the twisted Orlik
-Solomon algebras $(A^\bullet,\a\bomega^T)$ into a new, infinite dimensional commutative differential graded algebra $\A_\a^\bullet$ that approximates $E^\bullet(X,\O_\vrho)$.

\subsection{Notation}
\label{ANotation}

\noindent Let $\{K_1,\dots,K_n\}$ be an affine hyperplane arrangement in an complex vector space $V$. Set $\mathcal{K} = \bigcup_{j=1}^n K_j$, and let $X$ denote the complement $X = V - \mathcal{K}$. Set $\pi = \pi_1(X,x_0)$. Let $\T = H^1(X,\C^*)$ denote the character torus. For $j = 1,\dots,n$, choose a linear function $L_j$ on $V$ whose vanishing set is the hyperplane $K_j$, and define a holomorphic $1$-form $\omega_j$ on $X$ by $\omega_j = (2\pi i)^{-1} d L_j/L_j$. Set $\bomega = (\omega_1,\dots,\omega_n)$, and let $\bomega^T$ denote the transpose of $\bomega$. Recall the notation of Section \ref{HypArr}: If $\a \in \C^n$, then the closed holomorphic $1$-form $\a \bomega^T$ on $X$ is defined as follows.
\begin{equation}
\label{adot}
\a\bomega^T = a_1 \omega_1 + \dots + a_n \omega_n
\end{equation}
Let $M_{N \times n}(\C)$ denote the set of $N$ by $n$ matrices with entries in $\C$. In general, if $\a$ is any element of $M_{N \times n}(\C)$, then $\a\bomega^T$ is a closed holomorphic $1$-form on $X$ with values in $\C^N$. Thus, $\vrho = \exp(\a\bomega^T)$ is an element of $\T^N$.

\subsection{The Problem with $(A^\bullet,-\a\bomega^T)$}
\label{AProblem}

Let $A^\bullet$ denote the complexified Orlik-Solomon algebra, defined in Section \ref{OSAlg}. If $\a \in \C^n$, then $\vrho = \exp(\a\bomega^T)$ is an element of the character torus $\T$. Let $\nabla_{\a}$ denote the connection on the trivial line bundle $\C \times X \to X$ defined by

$$
\nabla_\a\sigma = d\sigma - (\a \bomega^T)\sigma
$$
for $\sigma \in E^0(X)$. There is a natural inclusion
$$
(A^\bullet,-\a\bomega^T) \, \hookrightarrow \, E^\bullet(X,\L_\vrho).
$$
of complexes. Though the product in $A^\bullet$ induces a product in $H^\bullet(A^\bullet,-\a\bomega^T)$, the algebra $(A^\bullet,-\a\bomega^T)$ is not a differential graded algebra. If $\zeta$ and $\psi$ are elements of $(A^\bullet,-\a\bomega^T)$, then $\zeta \wedge \psi$ is an element of the complex $(A^\bullet,-2\a\bomega^T)$. The cup product of two elements of $H^\bullet(X,\L_{\vrho})$ lies in $H^\bullet(X,\L_{\vrho^2})$. The diagram
$$
\xymatrix@1{(A^\bullet,-\a\bomega^T) \otimes_\C (A^\bullet,-\a\bomega^T) \ar[d] \ar[r]^-\wedge & (A^\bullet,-2\a\bomega^T) \ar[d] \\ E^\bullet(X,\L_\vrho) \otimes_\C E^\bullet(X,\L_\vrho) \ar[r]^-\wedge & E^\bullet(X,\L_{\vrho^2})}
$$
commutes, and all maps are chain maps. That is, although the cup product of two elements in $(A^\bullet,-\a\bomega^T)$ is an element of this same complex, it is more naturally an element of the complex $(A^\bullet,-2\a\bomega^T)$. Thus, it is natural to define
$$
\A_\a^\bullet = \bigoplus_{k \in \Z} A^\bullet q^{-k}
$$
By the construction in Section \ref{GenConstructionAlg}, $\A_\a^\bullet$ is a commutative differential bigraded $\C$-algebra, where the differential is given on the $k$-th component by left multiplication by $-k\a\bomega^T$. It is graded by degree of differential forms and also by $k \in \Z$. In the next section, we generalize this construction to define $\A_\a^\bullet$, where $\a \in M_{N \times n}(\C)$. In this case, $\vrho = \exp(\a\bomega^T)$ is an element of $H^1(X,(\C^*)^N)$, which is naturally isomorphic to $\T^N$. The algebra $\A_\a^\bullet$ approximates $E^\bullet(X,\O_\vrho)$.

\subsection{The Algebra $\A_{\a}^\bullet$}
\label{Avrho}

\noindent Let $A^\bullet$ denote the complexified Orlik-Solomon algebra, defined in Section \ref{OSAlg}. If $\a \in M_{N \times n}(\C)$, then $\a\bomega^T$ is a closed holomorphic $1$-form with values in $\C^N$. Thus, $\vrho = \exp(\a\bomega^T)$ is an element of $\T^N$. If $\k \in \Z^N$, then $\k\a \in \C^n$. As in Section \ref{ROLS}, let $\nabla_{\k\a}$ be the flat connection on the trivial line bundle $\C \times X \to X$ defined by the formula
$$
\nabla_{\k \a}\sigma = d\sigma - (\k\a\bomega^T)\sigma
$$
for $\sigma \in E^0(X)$. Thus, $E^\bullet(X)$ is a chain complex with differential $\nabla_{\k \a}$.

For each $\k \in \Z^N$, the algebra $A^\bullet$ is a subcomplex of $(E^\bullet(X),\nabla_{\k\a})$. The restriction of $\nabla_{\k \a}$ to $A^\bullet$ is given by left multiplication by $-(\k\a\bomega^T)$. Moreover, for $\varphi_1 \in A^j$ and $\varphi_2 \in A^\bullet$,
$$
\nabla_{(\k_1+\k_2)\cdot \a} (\varphi_1\wedge \varphi_2) =
(\nabla_{\k_1\cdot \a} \varphi_1)\wedge \varphi_2 + (-1)^j\varphi_1 \wedge (\nabla_{\k_2\cdot \a} \varphi_2).
$$
By the general construction in Section \ref{GenConstructionAlg}, we can therefore define a commutative differential graded algebra $\A_{\a}^\bullet$ by
$$
\A_{\a}^\bullet = \bigoplus_{\k \in \Z^N}A^\bullet q_1^{-k_1}\cdots q_N^{-k_N},
$$
where the differential is given on the $\k$-th component by left multiplication by $-\k\a\bomega^T$ and where $q_j$ is the $j$-th standard character on $\G_m^N$. Then each $q_j$ can be viewed as an element of the coordinate ring of $\G_m^N$. The action of $\G_m^N$ on its coordinate ring gives $\A_{\a}^\bullet$ the structure of a (right) representation of $\G_m^N$.

Each character $\beta$ on $\G_m^N$ determines an irreducible representation $W_\beta$ of $\G_m^N$ and a rank-one local system $\WW_\beta$ on $X$ whose monodromy is given by the character $\beta\circ\vrho$ of $\pi$. Again by the construction in Section \ref{GenConstructionAlg}, the direct sum
$$
\bigoplus_{\beta \in [\G_m^N]^\vee} E^\bullet(X,\WW_\beta)\otimes W_\beta^*
$$
is a commutative differential graded algebra. The differential is defined componentwise, as is multiplication. The product of an element in $E^\bullet(X,\WW_\alpha) \otimes W_\alpha^*$ with an element in $E^\bullet(X,\WW_\beta) \otimes W_\beta^*$ lies in $E^\bullet(X,\WW_{\alpha\beta}) \otimes W_{\alpha\beta}^*$.

\begin{lemma}
\label{inclusionlemma}
There is a natural $\G_m^N$-equivariant inclusion
\begin{equation}
\label{inclusion1}
\A_{\a}^\bullet \hookrightarrow \bigoplus_{\beta \in [\G_m^N]^\vee} E^\bullet(X,\WW_\beta)\otimes W_\beta^*
\end{equation}
of commutative differential graded algebras.
\end{lemma}

\begin{proof}
For $\k \in \Z^N$, let $\L_{\vrho^\k}$ denote the rank one local system on $X$ with monodromy $\rho_1^{k_1}\cdots\rho_N^{k_N}$. There is a natural $\G_m^N$-equivariant isomorphism
$$
\bigoplus_{\beta \in [\G_m^N]^\vee} E^\bullet(X,\WW_\beta)\otimes W_\beta^* \cong \bigoplus_{\k \in \Z^N} E^\bullet(X,\L_{\vrho^\k}) \cdot q_1^{-k_1}\cdots q_N^{-k_N}
$$
of commutative differential graded algebras. For $\k \in \Z^N$, there is an isomorphism
$$
(E^\bullet(X), \nabla_{\k\a}) \cong E^\bullet(X,\L_{\vrho^\k})
$$
of complexes. Moreover, the following diagram commutes for all $\k$ and $\x$ in $\Z^N$, where the horizontal arrows are given by the cup product.
$$
\xymatrix{(E^\bullet(X),\nabla_{\k\a}) \otimes (E^\bullet(X),\nabla_{\x\a})  \ar[rr]^{\wedge} \ar[d]^\cong && (E^\bullet(X),\nabla_{(\k+\x)\a})\ar[d]^\cong\\
E^\bullet(X,\L_{\vrho^{\k}}) \otimes E^\bullet(X,\L_{\vrho^{\x}}) \ar[rr]^\wedge && E^\bullet(X,\L_{\vrho^{(\k+\x)}})}
$$
For each $\k \in \Z^N$, $(A^\bullet,-\k\a\bomega^T)$ is a subcomplex of $(E^\bullet(X),\nabla_{\k\a})$. The result follows.
\end{proof}

Recall that $\vrho$ is a representation $\vrho \colon \pi \to (\C^*)^N$. Let $D_\vrho$ denote the Zariski closure of the image of $\vrho$ in $\G_m^N$. Since $\A_{\a}^\bullet$ is a representation of $\G_m^N$, it is a representation of $D_\vrho$ as well.

\begin{theorem}
If $\vrho = \exp(\a\bomega^T)$, then there is a natural $D_\vrho$-equivariant homomorphism
$$
\A_{\a}^\bullet \longrightarrow E^\bullet(X,\O_\vrho)
$$
of commutative differential graded algebras. It is an inclusion when $\vrho$ is Zariski dense.
\end{theorem}

\begin{proof}
By Lemma \ref{inclusionlemma}, it suffices to prove that there is a natural $D_\vrho$-equivariant homomorphism
\begin{equation}
\label{hello1}
\bigoplus_{\beta \in [\G_m^N]^\vee} E^\bullet(X,\WW_\beta) \otimes W_\beta^* \longrightarrow E^\bullet(X,\O_{\pmb \rho})
\end{equation}
of commutative differential graded algebras, which is equality when $\vrho$ is Zariski dense. Recall that $E^\bullet(X,\O_\vrho)$ is defined by
$$
E^\bullet(X,\O_\vrho) = \bigoplus_{\alpha \in D_\vrho^\vee} E^\bullet(X,\VV_{\alpha}) \otimes V_\alpha^*
$$
where $V_\alpha$ is the irreducible representation of $D_\vrho$ given by the character $\alpha$ and $\VV_\alpha$ is the rank-one local system on $X$ with monodromy given by the character $\alpha\circ\vrho$ of $\pi$. For $\beta \in [\G_m^N]^\vee$, the restriction of $W_\beta$ to $D_\vrho$ is isomorphic to $V_\alpha$, where $\alpha$ is the restriction of the character $\beta$ to $D_\vrho$. Thus, the homomorphism (\ref{hello1}) is simply defined by restriction of the irreducible representations $W_\beta$ of $\G_m^N$ to $D_\vrho$.
\end{proof}

\subsection{The Induced Map $H^\bullet(\A_{\a}^\bullet) \to H^\bullet(X,\O_\vrho)$}
\label{OSIso}

If $\a \in M_{N \times n}(\C)$, then $\vrho = \exp(\a\bomega^T)$ is an element of $\T^N$. The homomorphism $\A_{\a}^\bullet \hookrightarrow E^\bullet(X,\O_\vrho)$ of commutative differential graded algebras induces a homomorphism
$$
H^\bullet(\A_{\a}^\bullet) \longrightarrow H^\bullet(X,\O_\vrho)
$$
of graded algebras.

\begin{theorem}
\label{OSATheorem}
If $\mathbf{V}$ is a vector subspace of $M_{N \times n}(\C)$, then there is a countable collection $\{\W_j\}$ of proper affine subspaces of $\mathbf{V}$ that do not contain $0$ with the following property. If $\a \in \mathbf{V}-\bigcup_j\W_j$ and $\vrho = \exp(\a\bomega^T)$ has Zariski dense image in $\G_m^N$, then the homomorphism $H^\bullet(\A_{\a}^\bullet) \longrightarrow H^\bullet(X,\O_\vrho)$ is an isomorphism.
\end{theorem}

\begin{proof} For each $\k \in \Z^N$, Theorem \ref{ESVTheorem} implies that there is a countable collection $\{\Psi_{\k,M}\}_{M \in \Z}$ of affine subspaces of $M_{N \times n}(\C)$ that do not contain $0$ such that the inclusion $$(A^\bullet,-\k\a\bomega^T) \hookrightarrow (E^\bullet(X),\nabla_{\k\a})$$ is a quasi-isomorphism for all $\a \in M_{N \times n}(\C) - \bigcup_{M\in\Z} \Psi_{\k,M}$. Set $\W_{\k,M} = \mathbf{V} \cap \Psi_{\k,M}$. Then $\W_{\k,M}$ does not contain $0$, so it is a proper affine subspace of $\mathbf{V}$. Suppose $\a \in \mathbf{V}-\bigcup_{\k,M} \W_{\k,M}$.

Set $\vrho = \exp(\a\bomega^T) \in \T^N$ and suppose that $\vrho$ has Zariski dense image in $\G_m^N$. Given $\k \in \Z^N$, let $\L_{\vrho^\k}$ denote the rank one local system on $X$ with monodromy $\rho_1^{k_1}\cdots \rho_N^{k_N}$. As in Example \ref{IsDenseExample}, $E^\bullet(X,\O_\vrho)$ has the following description as a $\G_m^N$-module.
$$
E^\bullet(X,\O_\vrho) = \bigoplus_{\k \in \Z^N} E^\bullet(X,\L_{\vrho^\k}) \cdot q_1^{-k_1}\cdots q_N^{-k_N}
$$
The map $\A_{\a}^\bullet \to E^\bullet(X,\O_\vrho)$ on the $\k$-th component is determined by the inclusion $(A^\bullet,-\k\a\bomega^T) \hookrightarrow (E^\bullet(X),\nabla_{\k\a}) \cong E^\bullet(X,\L_{\vrho^\k})$, which is a quasi-isomorphism since $\a \notin \bigcup_{M \in \Z}\W_{\k,M}$. Thus, the induced map $$H^\bullet(\A_{\a}^\bullet) \to H^\bullet(X,\O_\vrho)$$ is an isomorphism on the $\k$-th component for all $\k \in \Z^N$. The result follows.
\end{proof}

\subsection{Massey Triple Products and $1$-Formality}
\label{MTPOneFormality}

\noindent In the next section, we show that when $X$ is the complement of the braid arrangement in $\C^2$, there exists $\vrho \in \T^2$ such that the algebra $E^\bullet(X,\O_\vrho)$ is not $1$-formal. Thus, the pronilpotent Lie algebra $\u_\vrho$ is not quadratically presented. This is accomplished by exhibiting a nonzero Massey triple product in $H^\bullet(\A_\a^\bullet)$, where $\vrho = \exp(\a\bomega^T)$. In this section, we review Massey triple products and their relationship to $1$-formality.

Let $R^\bullet$ be a commutative differential graded $\C$-algebra, and let $d$ denote the differential on $R^\bullet$. Let $H^1(R^\bullet)^2$ denote the vector subspace of $H^2(R^\bullet)$ defined by
$$
H^1(R^\bullet)^2 = \{ \phi \wedge \varphi | \phi,\varphi \in H^1(R^\bullet)\}.
$$
If $\varphi_1$, $\varphi_2$, and $\varphi_3$ are elements of $H^1(R^\bullet)$ such that $\varphi_1 \wedge \varphi_2 = \varphi_2\wedge \varphi_3 = 0$, then the {\em Massey triple product} $\big{\langle} \varphi_1, \varphi_2, \varphi_3 \big{\rangle}$ is a element of $H^2(R^\bullet)/H^1(R^\bullet)^2$. Equivalently, it is a coset of $H^1(R^\bullet)^2$ in $H^2(R^\bullet)$. It is defined as follows.

Choose closed elements $r_1$, $r_2$, and $r_3$ of $R^1$ that represent the cohomology classes $\varphi_1$, $\varphi_2$, and $\varphi_3$, respectively. Since $\varphi_1 \wedge \varphi_2 = \varphi_2\wedge \varphi_3 = 0$, there exist elements $r_{12}$ and $r_{23}$ of $R^1$ such that $d r_{12} = r_1 \wedge r_2$ and $d r_{23} = r_2 \wedge r_3$. Then $r_{12} \wedge r_3 + r_1 \wedge r_{23}$ is a closed element of $R^2$. Let $[r_{12} \wedge r_3 + r_1 \wedge r_{23}]$ denote its cohomology class in $H^2(R^2)$. Define
$$
\big{\langle} \varphi_1, \varphi_2, \varphi_3 \big{\rangle} \, = \, [r_{12} \wedge r_3 + r_1 \wedge r_{23}]  \, \, + \, H^1(R^\bullet)^2.
$$
This is a well-defined coset of $H^1(R^\bullet)^2$ in $H^2(R^\bullet)$. It is referred to as a Massey triple product in $H^2(R^\bullet)$.

The cohomology $H^\bullet(R^\bullet)$ is a commutative differential graded $\C$-algebra with trivial differential. Recall that $R^\bullet$ is {\em $1$-formal} if there exists a commutative differential graded $\C$-algebra $S^\bullet$ and differential graded algebra homomorphisms $\theta \colon S^\bullet \to R^\bullet$ and $\phi \colon S^\bullet \to H^\bullet(R^\bullet)$ which induce isomorphisms
\begin{equation*}
\begin{aligned}
\theta_* \colon H^0(S^\bullet) &\overset{\cong}{\longrightarrow} H^0(R^\bullet) \hspace{.5 in} & \theta_* \colon H^1(S^\bullet) &\overset{\cong}{\longrightarrow} H^1(R^\bullet)\\
\phi_* \colon H^0(S^\bullet) &\overset{\cong}{\longrightarrow} H^0(R^\bullet) & \phi_* \colon H^1(S^\bullet) &\overset{\cong}{\longrightarrow} H^1(R^\bullet)
\end{aligned}
\end{equation*}
and injections
\begin{equation*}
\begin{aligned}
\theta_* \colon H^2(S^\bullet)  \hookrightarrow H^2(R^\bullet) \hspace{.7125 in} &  \phi_* \colon H^2(S^\bullet)  \hookrightarrow H^2(R^\bullet).
\end{aligned}
\end{equation*}

\begin{proposition}
\label{1FormalMassey}
If $R^\bullet$ is $1$-formal, then all Massey triple products of degree-one elements vanish. \qed
\end{proposition}

\subsection{A Nontrivial Massey Triple Product}
\label{MTPSection}

\noindent Let $X \subset \C^2$ denote the complement of the braid arrangement $\mathcal{B}$. Let $\T$ denote the character torus. The intersection of $\mathcal{B}$ with $\mathbb{R}^2$ is shown below.

\begin{figure}[!ht]
\label{Braid34}
\centering\epsfig{file=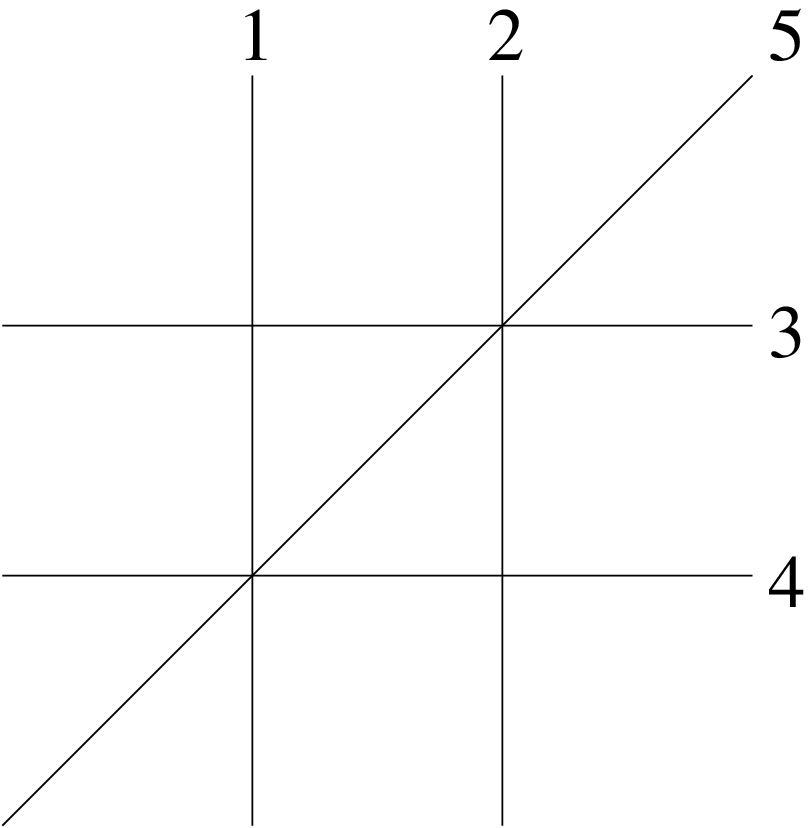, width=2.2in}
\caption{The braid arrangement $\mathcal{B}$}
\end{figure}

\begin{theorem}
\label{MyTheorem}
There exist infinitely many $\vrho \in \T^2$ for which $H^2(X,\O_\vrho)$ has a nonzero Massey triple product of degree-one elements. Thus, the commutative differential graded algebra $E^\bullet(X,\O_\vrho)$ is not $1$-formal and the pronilpotent Lie algebra $\u_\vrho$ is not quadratically presented.
\end{theorem}

\begin{proof}
Let the hyperplanes be numbered as indicated. Define elements $\lambda_1$ and $\lambda_2$ of the Orlik-Solomon algebra $A^\bullet$ by $\lambda_1 = \w_2 + \w_3 - 2\w_5$ and $\lambda_2 = -\w_1-\w_4 + 2\w_5$. Let $\mathbf{V}$ be the one dimensional vector subspace of $M_{2 \times 5}(\C)$ spanned by
$$
\pmb{b} \, = \, \begin{pmatrix}
0 & 1 & 1 & 0 & -2\\
-1 & 0 & 0 & -1 & 2
\end{pmatrix}.
$$
Given an element $\a = r \pmb{b}$ of $\mathbf{V}$, where $r \in \C$, the element $\vrho = \exp(\a\bomega^T)$ of $\T^2 = H^1(X,(\C^*)^2)$ is given by $\vrho = (\rho_1,\rho_2)$, where $\rho_1, \rho_2 \in H^1(X,\C^*) \cong (\C^*)^5$ and
\begin{equation*}
\begin{aligned}
\text{$\rho_1 = (1,e^{r},e^{r},1,e^{-2r})$\,\, and\,\, $\rho_2 = (e^{-r },1,1,e^{-r},e^{2r})$.}
\end{aligned}
\end{equation*}
The image of $\vrho \colon \pi_1(X,x_0) \to (\C^*)^2$ contains both $(1,e^{-r})$ and $(e^{r },1)$. Thus, by applying Theorem \ref{OSATheorem} to $\mathbf{V}$, there are infinitely many $r \in \C$ such that $\vrho$ has Zariski dense image in $\G_m^2$ and such that the map $H^\bullet(\A_{\a}^\bullet) \to H^\bullet(X,\O_\vrho)$ is an isomorphism. To prove the theorem, it therefore suffices to show that there exists a nonzero Massey triple product of 1-forms in $H^2(\A_{\a}^\bullet)/H^1(\A_{\a})^2$. Define closed holomorphic 1-forms $\alpha_1$ and $\alpha_2$ in $A^\bullet$ by $\alpha_j = r \lambda_j$. The cohomology classes $[\alpha_j]$ lie in $H^1(X,\C)$, and we have
$$
\a\bomega^T = \binom{\alpha_1}{\alpha_2}.
$$

The cup product in $A^\bullet$ will be written as juxtaposition. The elements $(\w_2-\w_3)\cdot q_1$ and $(\w_1-\w_4)\cdot q_2$ are closed in $\A_{\a}^\bullet$, as $\alpha_1 (\w_2-\w_3) = \alpha_2 (\w_1-\w_4) = 0$. The product $(\w_1-\w_4)(\w_2-\w_3)\cdot q_1q_2$ lies in $A^\bullet\cdot q_1q_2 \subset \A_{\a}^\bullet$. We have
$$
\begin{aligned}
\nabla_{\a}\bigg{[}\biggl-\frac{1}{r}\w_2 - \frac{1}{r}\w_3\biggr\cdot q_1q_2\bigg{]} &=(-\alpha_1-\alpha_2)\biggl-\frac{1}{r}\w_2 - \frac{1}{r}\w_3\biggr\cdot q_1q_2 \\
&= (r\w_1 - r\w_2 - r\w_3 + r\w_4)\biggl-\frac{1}{r}\w_2 - \frac{1}{r}\w_3\biggr \cdot q_1 q_2\\
&= (\w_2\w_4-\w_1\w_3)\cdot q_1q_2\\
&= (\w_1-\w_4)(\w_2-\w_3)\cdot q_1q_2.
\end{aligned}
$$
That is, the cohomology class $[(\w_1-\w_4)(\w_2-\w_3)\cdot q_1q_2]$ in $H^2(\A_{\a}^\bullet)$ is trivial. Trivially, we have $((\w_2-\w_3)\cdot q_1)^2 = 0$. The Massey triple product
\begin{equation}
\label{MTP333}
\big{\langle} [(\w_1-\w_4)\cdot q_2], [(\w_2-\w_3)\cdot q_1],[(\w_2-\w_3)\cdot q_1]\big{\rangle}
\end{equation}
is therefore defined. We show that it is nonzero in $H^2(\A^\bullet_{\a})/H^1(\A^\bullet_{\a})^2$. This Massey triple product is equal to
$$
\bigg{[}\biggl-\frac{1}{r}\w_2 - \frac{1}{r}\w_3\biggr\biggl\w_2-\w_3\biggr\cdot q_1^2q_2\bigg{]} = \frac{2}{r}\bigg{[}\w_2\w_3\cdot q_1^2q_2\bigg{]} \, \, + \, H^1(\A^\bullet_{\a})^2.
$$
To show that it is nonzero, it suffices to show that $[\w_1\w_2\cdot q_1^2q_2]$ is not equal a sum of cup products of elements in $H^1(\A_{\a}^\bullet)$.

First, we show that if this Massey triple product is trivial, then $[\w_1 \w_2 \cdot q_1^2q_2]$ is equal to a sum of cup products of cohomology classes of elements in $A^\bullet\cdot q_1^2$ with cohomology classes of elements in $A^\bullet\cdot q_2$. To prove this statement, it would suffice show to that if $H^1(X,\L_{\rho_1^s\rho_2^t}) \neq 0$, then $st = 0$. We thus consider the characteristic variety $\V_1^1(X)$, which is completely described by Suciu in \cite[Example 10.3]{Suciu1}. We have
$$
s\alpha_{1} + t\alpha_{2} = -rt\w_1 + rs\w_2 + rs\w_3 - rt\w_4 + 2r(t-s)\w_5.
$$
If the character $\rho_1^s\rho_2^t$ is an element of $V_1^1(X)$, then by Suciu's description of the variety $\V_1^1(X)$, $s$ or $t$ must be zero. Thus, we have shown that if the Massey triple product (\ref{MTP333}) is trivial, then $[\w_1 \w_2 \cdot q_1^2q_2]$ is equal to a sum of cup products of cohomology classes of elements in $A^1\cdot q_1^2$ with cohomology classes of elements in $A^1\cdot q_2$.

By \cite[Example 10.3]{Suciu1}, we know that $\dim_\C H^1(X,\L) \leq 1$ for all nontrivial rank-one local systems $\L$ on $X$. Thus, if the Massey triple product (\ref{MTP333}) is trivial, then $[\w_1 \w_2 \cdot q_1^2q_2]$ is equal to the cup product of the cohomology class of an element in $A^\bullet\cdot q_1^2$ with the cohomology class of an element in $A^\bullet\cdot q_2$. That is, we can write
$$
\w_2\w_3 = \varphi\beta + \psi(2\alpha_1 + \alpha_2)
$$
as forms in the Orlik-Solomon algebra $A^\bullet$, where $\alpha_1\wedge\varphi = 0$ and $\alpha_2\wedge\beta = 0$. By an elementary argument using the fact that $H^1(X,\L_{\rho_1^2})$ and $H^1(X,\L_{\rho_2})$ are both one-dimensional, there are complex numbers $f_1$, $f_2$, $f_3$, $f_4$, $f_5$, $x$, $y$, $g$, and $h$ such that the following equalities hold.
$$
\begin{aligned}
\psi &= \frac{f_1}{r}\w_1 + \frac{f_2}{r}\w_2 + \frac{f_3}{r}\w_3 + \frac{f_4}{r}\w_4 + \frac{f_5}{r}\w_5\\
\varphi &= x \w_2 + y\w_3 - (x+y)\w_5\\
\beta &= g \w_1 + h\w_4 - (g+h)\w_5
\end{aligned}
$$
We have
$$
\begin{aligned}
&\psi(2\alpha_1 + \alpha_2) =\\
 & (f_1\w_1 + f_2\w_2 + f_3\w_3 + f_4\w_4 + f_5\w_5)(-\w_1 + 2\w_2 + 2\w_3 - \w_4 - 2\w_5)\\
&= 2f_1\w_1\w_3 - f_1\w_1\w_4 - 2f_1\w_1\w_5 + 2f_2\w_2\w_3 - f_2\w_2\w_4 - 2f_2\w_2\w_5 \\
& \hspace{1 in} + f_3\w_1\w_3 - 2f_3\w_2\w_3 - 2f_2\w_3\w_5 + f_4\w_1\w_4 - 2f_4\w_2\w_4\\
& \hspace{1 in} - 2f_4\w_4\w_5 + f_5\w_1\w_5 - 2f_5\w_2\w_5 - 2f_5\w_3\w_5 + f_5\w_4\w_5\\
&= 2f_1\w_1\w_3 - f_1\w_1\w_4 - 2f_1\w_1\w_5 + 2f_2\w_2\w_3 - f_2\w_2\w_4 - 2f_2\w_2\w_5 \\
& \hspace{1 in} + f_3\w_1\w_3 - 2f_3\w_2\w_3 + 2f_3\w_2\w_3 - 2f_3\w_2\w_5\\
& \hspace{1 in} + f_4\w_1\w_4 - 2f_4\w_2\w_4 + 2f_4\w_1\w_4 - 2f_4\w_1\w_5\\
& \hspace{1 in} + f_5\w_1\w_5 - 2f_5\w_2\w_5 + 2f_5\w_2\w_3 - 2f_5\w_2\w_5\\
& \hspace{1 in} - f_5\w_1\w_4 + f_5 \w_1\w_5
\end{aligned}
$$
We also have
$$
\begin{aligned}
\varphi\beta &= (x \w_2 + y\w_3 - (x+y)\w_5)(g \w_1 + h\w_4 - (g+h)\w_5)\\
&= xh\w_2\w_4 - x(g+h)\w_2\w_5 - yg\w_1\w_3 - y(g+h)\w_3\w_5 \\
&\hspace{1.4 in} + g(x+y)\w_1\w_5 + h(x+y)\w_4\w_5\\
&= xh \w_2\w_4 - x(g+h)\w_2\w_5 - yg\w_1\w_3 -y(g+h)\w_2\w_5 + y(g+h)\w_2\w_3\\
&\hspace{1.4 in} + g(x+y)\w_1\w_5 + h(x+y)\w_1\w_5 - h(x+y)\w_1\w_4
\end{aligned}
$$
Since $\{\w_1\w_3,\w_1\w_4,\w_1\w_5,\w_2\w_3\,\w_2\w_4,\w_2\w_5\}$ is a basis for $A^2$ and $\w_2\w_3 = \varphi\beta + \psi(2\alpha_{\rho_1} + \alpha_{\rho_2})$, we have the following set of equations.
$$
\begin{aligned}
2f_1 + f_3 - y g &= 0 &\hspace{.5 in} \textrm{(1)}\\
-f_1 + 3f_4 - f_5 - (x+y)h &= 0 &\hspace{.5 in} \textrm{(2)}\\
-2f_1-2f_4+2f_5 + (x+y)(g+h) &= 0 &\hspace{.5 in} \textrm{(3)}\\
2f_2 + 2f_5 + y(g+h) &= 1 &\hspace{.5 in} \textrm{(4)}\\
-f_2 - 2 f_4 + xh &= 0 &\hspace{.5 in} \textrm{(5)}\\
-2f_2 - 2f_3 - 4f_5 - (x+y)(g+h) &= 0 &\hspace{.5 in} \textrm{(6)}
\end{aligned}
$$
Adding (3) and (6) together gives
\begin{equation}
\label{WrongEquation1}
f_1 + f_2 + f_3 + f_4 + f_5 = 0.
\end{equation}
Note that $y (g+h) = (x+y) h + yg - xh$. Thus, (1), (5), and (2) imply that
$$
\begin{aligned}
y(g+h) &= (-f_1 + 3 f_4 - f_5) + (2 f_1 +  f_3) + (-f_2 - 2 f_4)\\
&= f_1 - f_2 + f_3 + f_4 - f_5.
\end{aligned}
$$
Plugging this in for $y(g+h)$ in (4) yields
\begin{equation}
\label{WrongEquation2}
\begin{aligned}
2 f_2  + 2 f_5 + y(g+h) & = 1\\
f_1 + f_2 + f_3 + f_4 + f_5 &= 1.
\end{aligned}
\end{equation}
Equations (\ref{WrongEquation1}) and (\ref{WrongEquation2}) are incompatible. This completes the proof.
\end{proof}

\section{Chen's Reduced Bar Construction}
\label{Bar}

\noindent In this section, we review Chen's (\cite{Chen2}, \cite{Chen3}) reduced bar construction $B(M,R^\bullet,N)$. We prove that under certain conditions, it is a differential graded Hopf algebra. In particular, the Hopf algebra $H^0(B(\C,E^\bullet(X,\O_\vrho),\O(D_\vrho)))$ is the coordinate ring of the relative Malcev completion $\S_\vrho$.

\subsection{Definition of the Reduced Bar Construction}
\label{DefOfBarConstr}

\noindent Suppose that $R^\bullet$ is a commutative differential graded algebra over a commutative ring $F$ and that $M$ and $N$ are graded $R^\bullet$-algebras that only have degree-zero elements. Then $rm = 0$ and $rn = 0$ for all $r \in R^+$, $m \in M$, and $n \in N$. We assume that $R^\bullet$ is non-negatively weighted. Define $R^+[1]$ to be the $F$-module obtained from $R^+$ be reducing degrees by $1$. Chen's (\cite{Chen2}, \cite{Chen3}) {\em reduced bar construction} $B(M,R^\bullet,N)$ is a quotient of the $F$-module
$$T(M,R^\bullet,N) \colon = \bigoplus_{s \geq 0} M \otimes (R^+[1]^{\otimes s}) \otimes N.$$ For $m \otimes (r_1\otimes \cdots \otimes r_s) \otimes n \in T(M,R^\bullet,N)$, we write $m[r_1|\cdots |r_s]n$ for its equivalence class in $B(M,R^\bullet,N)$. The relations in $B(M,R^\bullet,N)$ are given below.
\begin{equation*}
\begin{aligned}
m [d g | r_1 | \cdots | r_s] n &=  m [g r_1 | \cdots | r_s] n - m \cdot g [r_1|\cdots| r_s] n;\\
m [r_1|\cdots |r_j | dg |r_{j+1} | \cdots | r_s] n &= m[r_1 |\cdots | r_j |g r_{j+1}|\cdots | r_s]n\\
&\hspace{.15 in} - m [r_1 |\cdots| r_j g | r_{j+1} |\cdots| r_s] n \hspace{.3 in} 1 \leq j < s;\\
m [r_1| \cdots |r_s |d g]n &= m[r_1|\cdots| r_s] g\cdot n - m [r_1|\cdots| r_s g] n;\\
m [ d g ] n &=  1 \otimes g\cdot n - m\cdot g \otimes 1.
\end{aligned}
\end{equation*}
Here each $r_j \in R^+$, $g \in R^0$, $m \in M$, and $n \in N$.

The reduced bar construction $B(M,R^\bullet,N)$ has the structure of a commutative differential graded algebra over $F$. The degree of the element $m[r_1|\cdots|r_s]n$ is defined to be $\deg(r_1) + \cdots + \deg(r_s) - s$. Define an endomorphism $J$ of $R^\bullet$ by $J(r) = (-1)^{\deg r} r$ for each homogeneous element $r$. The differential on $B(M,R^\bullet,N)$ is defined by
\begin{equation}
\label{ddef}
\begin{aligned}
d \hspace{.05 in} m[r_1|\cdots|r_s]n &= \sum_{1 \leq j \leq s} (-1)^j m[J r_1|\cdots|J r_{j-1}| d r_j | r_{j+1}|\cdots | r_s]n\\
&\hspace{.15 in} + \sum_{1 \leq j < s} (-1)^{i+1} m[J r_1|\cdots | J r_{j-1}| J r_j \wedge r_{j+1} | r_{j+1} | \cdots | r_s]n.
\end{aligned}
\end{equation}
The product in $B(M,R^\bullet,N)$ is given by
$$
(m[r_1|\cdots|r_p]n)\cdot (m'[r_{p+1}|\cdots|r_{p+q}]n') = \sum_{\sigma \in Sh(p,q)} m m' [r_{\sigma(1)}|\cdots |r_{\sigma(p+q)}] n n',
$$
where $Sh(p,q)$ denotes the set of shuffles of type $(p,q)$. With this product, the reduced bar construction $B(M,R^\bullet,N)$ is a commutative differential graded $F$-algebra. The map $F \to B(M,R^\bullet,N)$ is given by $1 \mapsto [\,\phantom{i}]$.

If $R^\bullet \to A^\bullet$ is a surjective homomorphism of commutative differential graded $F$-algebras, then $M \otimes_{R^\bullet} A^\bullet$ and $N \otimes_{R^\bullet} A^\bullet$ are $A^\bullet$-modules. Their annihilators contain $A^+$. Thus, we may form the reduced bar construction $B(M \otimes_{R^\bullet}A^\bullet ,A^\bullet, N \otimes_{R^\bullet}A^\bullet)$.

\begin{proposition}
\label{BarSpecializationProposition}
If $R^\bullet \to A^\bullet$ is a surjective homomorphism of commutative differential graded $F$-algebras, then the canonical homomorphism
\begin{equation}
\label{BarSpec222}
B(M,R^\bullet,N) \otimes_{R^\bullet} A^\bullet \, \longrightarrow \, B(M \otimes_{R^\bullet} A^\bullet, A^\bullet, N \otimes_{R^\bullet} A^\bullet)
\end{equation}
is an isomorphism of differential graded $A^\bullet$-algebras.
\end{proposition}

\begin{proof}
For simplicity of notation, all tensor product symbols in the proof are assumed to be over $R^\bullet$. Let $Rel_R$ denote the $R^\bullet$-submodule of $T(M,R^\bullet,N)$ consisting of all elements that have trivial equivalence class in $B(M,R^\bullet,N)$. Let $Rel_A$ denote the $A^\bullet$-submodule of $T(M \otimes A^\bullet, A^\bullet, N \otimes A^\bullet)$ consisting of all elements that have trivial equivalence class in $B(M \otimes A^\bullet, A^\bullet, N\otimes A^\bullet)$. The sequence
$$
0 \longrightarrow Rel_R \longrightarrow T(M,R^\bullet,N) \longrightarrow B(M,R^\bullet,N) \longrightarrow 0
$$
is exact. Thus, the following diagram commutes, and both rows are exact. All tensor product symbols indicate a tensor product over $R^\bullet$.
$$
\xymatrix{Rel_R \otimes A^\bullet \ar[d]^\phi \ar[r] & T(M,R^\bullet,N) \otimes A^\bullet \ar[d]^\cong \ar[r] & B(M,R^\bullet,N) \otimes A^\bullet \ar[d]\\ Rel_A \ar[r] & T(M \otimes A^\bullet, A^\bullet, N \otimes A^\bullet) \ar[r]^\theta & B(M \otimes A^\bullet, A^\bullet, N \otimes A^\bullet)}
$$
The homomorphism (\ref{BarSpec222}) is surjective because $\theta$ is surjective and injective because $\phi$ is surjective.
\end{proof}

\subsection{The Eilenberg-Moore Spectral Sequence}
\label{EMSS}

\noindent The reduced bar construction $B(M,R^\bullet,N)$ has a standard filtration
$$
F = B^0(M,R^\bullet,N) \subset B^{-1}(M,R^\bullet,N) \subset B^{-2}(M,R^\bullet,N) \dots
$$
The subspace $B^{-s}(M,R^\bullet,N)$ is defined to be the $F$-submodule of $B(M,R^\bullet,N)$ generated by those $m[r_1|\cdots|r_t]n$ with $t \leq s$. The second quadrant spectral sequence $E_n^{s,t}$ corresponding to this filtration is known as the {\em Eilenberg-Moore spectral sequence}. One always has $E_n^{-s,s} \Rightarrow H^0(B(M,R^\bullet,N))$.

\begin{proposition}[Chen, \cite{Chen2}]
\label{EMSS1}
If $H^0(R^\bullet) = F$, then the $E_1$ term of the Eilenberg-Moore spectral sequence is given by
$$
E_1 = B(M,H^\bullet(R^\bullet),N),
$$
where $H^\bullet(R^\bullet)$ is given the trivial differential. The differential $d_1$ is therefore given by the cup product. \qed
\end{proposition}

Note that the algebra $H^\bullet(R^\bullet)$ is a commutative differential graded algebra when equipped with the trivial differential. The algebra $R^\bullet$ is {\em formal} if there is a commutative differential graded algebra $S^\bullet$ and quasi-isomorphisms $S^\bullet \to R^\bullet$ and $S^\bullet \to H^\bullet(R^\bullet)$. The following proposition shows that the formality of $R^\bullet$ is closely related to the differentials in the Eilenberg-Moore spectral sequence.

\begin{proposition}
\label{prop2222}
If $R^\bullet$ is $1$-formal, then $E_2^{-s,s} = E_\infty^{-s,s}$, and if $R^\bullet$ is formal, then $E_2 = E_\infty$. \qed
\end{proposition}

\begin{corollary}
\label{BarQuism}
If $R^\bullet \hookrightarrow A^\bullet$ is a quasi-isomorphism of non-negatively weighted commutative differential graded $F$-algebras and $M$ and $N$ are graded $A^\bullet$-modules that have only degree-zero elements, then the map $B(M,R^\bullet,N) \to B(M,A^\bullet,N)$ induces an isomorphism
$$
H^\bullet B(M,R^\bullet,N) \overset{\cong}{\longrightarrow} H^\bullet B(M,A^\bullet,N)
$$
of graded $R^\bullet$-algebras. \qed
\end{corollary}

\subsection{The Hopf Algebra $B(F,R^\bullet,\O)$}
\label{BarHopfAlgebra}

\noindent In this section, we show that under certain hypotheses, the reduced bar construction $B(F,R^\bullet,\O)$ is a differential graded Hopf algebra. This generalizes a construction by Hain in \cite{HainDeRham}. The main purpose for this construction is that it implies that $B(\C,E^\bullet(X,\O_\vrho),\O(D_\vrho))$ is a differential graded Hopf algebra. The Hopf algebra $H^0B(\C,E^\bullet(X,\O_\vrho),\O(D_\vrho))$ is the coordinate ring of the relative Malcev completion $\S_\vrho$ \cite{HainDeRham}. Suppose that we have the following data.
\begin{itemize}
\item A differential graded algebra $R^\bullet$ over a commutative ring $F$.
\item A Hopf algebra $\O$ over $F$ with augmentation $\epsilon$.
\item A homomorphism $R^\bullet \to \O$ of $F$-algebras that vanishes on $R^+$.
\item An $F$-algebra homomorphism $\nu \colon R^\bullet \to \O \otimes_F R^\bullet$ with the following properties.

(1) If $r \in R^j$ and $\nu(r) = \sum_k \phi_k \otimes r_k$, then $r_k \in R^j$ for all $k$ and $\nu(d r) = \sum_k \phi_k \otimes d r_k$.

(2) We have $(\Delta \otimes I)\circ\nu = (I \otimes \nu)\circ\nu \colon R^\bullet \to \O \otimes_F \O \otimes_F \otimes_F R^\bullet$.

(3) We have $(\epsilon \otimes I)\circ \nu = I\colon R^\bullet \to R^\bullet$.

\end{itemize}

Define a homomorphism $\epsilon\colon R^\bullet \to F$ of rings by composition $R^\bullet \to \O \overset{\epsilon}{\to} F$, where $\epsilon$ is the counit of $\O$. Then $\epsilon$ vanishes on $R^+$. Via the map $\epsilon \colon R^\bullet \to F$, we may view $F$ as an $R^\bullet$-module. Thus, we may form the reduced bar construction $B(F,R^\bullet,\O)$. We show here that it has the structure of a differential graded Hopf algebra over $F$. Suppose that $r_1,\dots,r_s \in R^+$. Define $\epsilon \colon B(F,R^\bullet,\O) \to F$ by $\epsilon([r_1|\cdots|r_s]\varphi) = 0$ for $s>0$ and $\epsilon([\,\phantom{i}]\varphi) = \epsilon(\varphi)$. Suppose that $\varphi \in \O$ and that the comultiplication in $\O$ sends $\varphi$ to $\sum_j \varphi_j'\otimes \varphi_j''$. Suppose that $\nu(r_\ell) = \sum_{k_\ell} \phi_{k_\ell}^{(\ell)} \otimes r_{k_\ell}^{\ell}$. Define an $F$-algebra homomorphism $\Delta \colon B(F,R^\bullet,\O) \to B(F,R^\bullet,\O)\otimes B(F,R^\bullet,\O)$ by
$$
\begin{aligned}
\Delta\colon [r_1|\cdots|&r_s]\varphi \,\, \longmapsto\\
&\sum_{i = 1}^s\sum_j \sum_{k_{i+1}}\cdots\sum_{k_s} ([r_1|\cdots|r_i]\phi_{k_{i+1}}^{i+1}\cdots \phi_{k_s}^{s}\varphi_j')\otimes ([r_{k_{i+1}}^{i+1}|\cdots|r_{k_s}^{s}]\varphi_j'').
\end{aligned}
$$
Define $\lambda \colon B(F,R^\bullet,\O) \to B(F,R^\bullet,\O)$ by
$$
\lambda \colon [r_1|\cdots|r_s]\varphi \,\, \longmapsto \,\, (-1)^s\sum_{k_s}\cdots\sum_{k_1} [r_{k_s}^{s}|\cdots|r_{k_1}^{1}]\iota(\phi_{k_s}^{s})\cdots\iota(\phi_{k_1}^{1})\iota(\varphi),
$$
where $\iota$ is the antipode of $\O$. The maps $\epsilon$, $\Delta$, and $\lambda$ are $F$-algebra homomorphisms. The proof of the following theorem can be found in \cite{Narkawicz}.
\begin{theorem}
The reduced bar construction $B(F,R^\bullet,\O)$ is a differential graded Hopf algebra over $F$ with counit $\epsilon$, comultiplication $\Delta$, and antipode $\lambda$. \qed
\end{theorem}
\begin{corollary}
The Hopf algebra structure on $B(F,R^\bullet,\O)$ induces a Hopf algebra structure on $H^0B(F,R^\bullet,\O)$. \qed
\end{corollary}
The following proposition allows us to describe the Hopf algebra structure on $H^0B(F,R^\bullet,\O)$ via the Eilenberg-Moore spectral sequence.

\begin{proposition}[Chen]
\label{EnHopf}
For each $n \geq 0$, the Hopf algebra structure on $B(F,R^\bullet,\O)$ induces a differential graded Hopf algebra structure on the $n$-th term of the associated Eilenberg-Moore spectral sequence. \qed
\end{proposition}

\subsection{The Hopf Algebra $B(\C,E^\bullet(X,\O_\vrho),\O(D_\vrho))$}
\label{HopfExampleSection}

\noindent Of particular interest to us is the reduced bar construction when the differential graded algebra $R^\bullet$ is $E^\bullet(X,\O_\vrho)$, defined as in Section \ref{CommDiffE}. In this section, we prove that there is an algebra homomorphism $E^\bullet(X,\O_\vrho) \to \O(D_\vrho)$ that vanishes in positive degree. As in the previous section, this gives the reduced bar constructions $B(\C,E^\bullet(X,\O_\vrho),\C)$ and $B(\C,E^\bullet(X,\O_\vrho),\O(D_\vrho))$ the structure of differential graded Hopf algebras.

Let $X$ be a smooth manifold, and set $\pi = \pi_1(X,x_0)$. Let $\X \to X$ denote the universal cover of $X$. Suppose that $\vrho \colon \pi \to (\C^*)^N$ is a representation. Let $D_\vrho$ denote the Zariski closure of the image of $\vrho$ in $\G_m^N$, and let $\O(D_\vrho)$ denote the coordinate ring of $\O(D_\vrho)$. Each character $\alpha$ on $D_\vrho$ gives an irreducible representation $V_\alpha$ of $D_\vrho$ and a rank-one local system $\VV_\alpha$ on $X$ with monodromy given by the character $\alpha\circ\vrho$ of $\pi$. Recall the definition of $E^\bullet(X,\O_\vrho)$:
$$
E^\bullet(X,\O_\vrho) = \bigoplus_{\alpha \in D_\vrho^\vee} E^\bullet(X,\VV_\alpha) \otimes V_{\alpha}^*.
$$
For each $\alpha \in D_\vrho^\vee$, there is a left action of $\pi$ on the trivial line bundle $\X \times \C \to \X$ defined by $\gamma \cdot (z,u) = (\gamma \cdot z,\alpha(\vrho(\gamma^{-1})) u)$. This induces a right action of $\pi$ on $E^\bullet(\X)$ via $\psi\cdot \gamma = \alpha(\vrho(\gamma)) (\gamma^{-1})^*\psi$. The complex $E^\bullet(X,\VV_\alpha)$ is defined to be the $\pi$-invariants of $E^\bullet(\X)$.

Evaluation of a differential form $\psi$ in $E^\bullet(X,\VV_\alpha) \otimes V_\alpha^*$ at $x_0$ gives an element $\delta(\psi)$ of $V_\alpha \otimes V_\alpha^*$, which is the trivial representation $\C$ of $D_\vrho$. Note that the coproduct in $\O(D_\vrho)$ satisfies $\Delta(\alpha^{-1}) = \alpha^{-1} \otimes \alpha^{-1}$. Define a $\C$-algebra homomorphism $E^\bullet(X,\O_\vrho) \to \O(D_\vrho)$ by sending $\psi$ to $\delta(\psi)\alpha^{-1}$. Then this map vanishes on $E^+(X,\O_\vrho)$. Define an algebra homomorphism
\begin{equation}
\label{rhonu}
E^\bullet(X,\O_\vrho) \,\,\overset{\nu}{\longrightarrow}\,\, E^\bullet(X,\O_\vrho) \otimes \O(D_\vrho)
\end{equation}
via $\psi \longmapsto \psi \otimes \alpha^{-1}$. This extends in a unique way to $E^\bullet(X,\O_\vrho)$. It follows easily from definitions that $\nu$ satisfies the hypotheses at the beginning of Section \ref{BarHopfAlgebra}. Thus, the reduced bar construction $B(\C,E^\bullet(X,\O_\vrho),\O(D_\vrho))$ is a differential graded Hopf algebra.

Viewing $\C$ as an algebra over $E^\bullet(X,\O_\vrho)$ via the map $\delta$, we may form the reduced bar construction $B(\C,E^\bullet(X,\O_\vrho),\C)$ as well. By the construction in Section \ref{BarHopfAlgebra}, it is also a differential graded Hopf algebra.

The Hopf algebra $H^0B(\C,E^\bullet(X,\O_\vrho),\O(D_\vrho))$ is the coordinate ring of a proalgebraic group scheme, which is in fact the relative Malcev completion $\S_\vrho$ \cite{HainDeRham}. The Hopf algebra $H^0B(\C,E^\bullet(X,\O_\vrho),\C)$ is the coordinate ring of the prounipotent radical $\U_\vrho$ of $\S_\vrho$. The homomorphism $\theta_\vrho \colon \pi \to \S_\vrho$ will be described in Section \ref{Iter} using iterated integrals.

\section{Iterated Integrals}
\label{Iter}

\noindent In this section, we review Hain's \cite{HainDeRham} generalization of Chen's (\cite{Chen2},\cite{Chen3}) iterated integrals. We describe the relative Malcev completion using the bar construction $B(\C,E^\bullet(X,\O_\vrho),\O(D_\vrho))$ and iterated integrals.

\subsection{Twisted Iterated Integrals}
\label{TwistedIteratedIntegrals}

\noindent Suppose that $X$ is a smooth manifold and that $\psi_1,\dots,\psi_r$ are smooth 1-forms on $X$. Chen \cite{Chen3} defined
\begin{equation}
\label{ChenIterInt}
\int_\gamma \psi_1\cdots\psi_r = \int_{0\leq t_1 \leq \dots \leq t_r \leq 1} f_1(t_1)\cdots f_r(t_r) dt_1\cdots dt_r,
\end{equation}
where $\gamma\colon [0,1] \to X$ is a piecewise smooth path and $\gamma^*\psi_j = f_j(t)dt$. The integral $\int \psi_1\cdots\psi_r$ can therefore be viewed as a function $PX \to \C$, where $P X$ denotes the path space of $X$. A linear combination of such functions is called an {\em iterated intgegral}.

The following generalization of Chen's iterated integrals is due to Hain \cite{HainDeRham}. Set $\pi = \pi_1(X,x_0)$. Let $\vrho \colon \pi \to (\C^*)^N$ be a representation, and let $D_\vrho$ denote the Zariski closure of the image of $\vrho$ in $\G_m^N$. This is a group subscheme of $\G_m^N$. Each character $\alpha$ on $D_\vrho$ gives a one-dimensional irreducible representation $V_\alpha$ of $D_\vrho$ and a rank-one local system $\VV_\alpha$ on $X$ whose monodromy is given by the character $\alpha\circ\vrho$ of $\pi$. For each $\alpha$, the fiber of $\VV_\alpha$ over $x_0$ is canonically identified with $V_\alpha$. The local system $\VV_\alpha \otimes V_{\alpha}^*$ is isomorphic to $\VV_\alpha$, as the tensor product with $V_{\alpha}^*$ simply indicates an action by $D_\vrho$. The fiber of $\VV_{\alpha} \otimes V_\alpha^*$ over the point $x_0$, however, is canonically identified with $V_\alpha \otimes V_\alpha^*$. As a representation of $D_\vrho$, this is canonically isomorphic to the trivial representation $\C$.

Suppose that $\psi_j \in E^1(X,\O_\vrho)$ for $j = 1,\dots,r$, that $\varphi \in \O(D_\vrho)$, and that $\gamma$ is a piecewise smooth loop at $x_0$ in $X$: $\, \gamma \colon [0,1] \to X$. We will define the iterated integral
$$
\int_\gamma (\psi_1\cdots\psi_r|\varphi) \, \in \C.
$$

First, suppose that each $\psi_j \in E^1(X,\WW_{\alpha_j}) \otimes W_{\alpha_j}^*$, where each $\alpha_j$ is a character of $D_\vrho$. Let $\widetilde{X}$ denote the universal cover of $X$, with basepoint $\tilde{x_0}$ over $x_0$. The local system $\WW_{\alpha_j}$ is equal to the quotient $(\widetilde{X} \times \C)/\pi_1(X,x_0)$, where $\pi_1(X,x_0)$ acts on $\widetilde{X} \times \C$ on the left via $\eta \cdot (z,u) = (\eta \cdot z,\alpha_j(\vrho(\eta^{-1}))\cdot u)$. This action induces a right action of $\pi_1(X,x_0)$ on $E^\bullet(\X)$ via $\psi \cdot \eta = \alpha_j(\vrho(\eta))(\eta^{-1})^*\psi$, where $\eta \in \pi_1(X,x_0)$. By definition, $E^1(X,\WW_{\alpha_j})$ is the set of $\pi_1(X,x_0)$-invariants of $E^\bullet(\X)$. Let $\tilde{\gamma}$ denote any lift of $\gamma$ to $\X$. We define
$$
\int_\gamma (\psi_1\cdots\psi_r|\varphi) = \varphi(\vrho(\gamma))\int_{\tilde{\gamma}}\psi_1\cdots\psi_r.
$$
This definition extends uniquely to the case $\psi_1,\dots,\psi_r \in E^1(X,\O_\vrho)$ in such a way that the integral $\int_\gamma (\psi_1\cdots\psi_r|\varphi)$ is multi-linear in the forms $\psi_j$ and in $\varphi$. When $r = 0$, we set $\int_\gamma(\,|\varphi) = \varphi(\vrho(\gamma))$.

\begin{definition}
The set $I(X)_\vrho$ of {\em iterated integrals} with coefficients in $\O(D_\vrho)$ is defined to be the set of all linear combinations of integrals of the form $\int(\psi_1\cdots\psi_r | \varphi)$, where $r \geq 0$, $\psi_j \in E^1(X,\O_\vrho)$, and $\varphi \in \O(D_\vrho)$.
\end{definition}

The elements of $I(X)_\vrho$ will be regarded as functions $\Omega_{x_0} X \to \C$ on the loop space $\Omega_{x_0} X$.

\begin{definition}
We define $H^0(I(X)_\vrho)$ to be the subset of $I(X)_\vrho$ consisting of all elements that are constant on each homotopy class $[\gamma] \in \pi_1(X,x_0)$. We call the elements of $H^0(I(X)_\vrho)$ {\em locally constant iterated integrals} with coefficients in $\O(D_\vrho)$.
\end{definition}
We will see in Section \ref{TwistedII} that the set $H^0(I(X)_\vrho)$ has a purely algebraic description.
\begin{proposition}[\cite{Chen3}, \cite{HainDeRham}]
\label{commprop}
For $\psi_1,\dots,\psi_{p+q} \in E^1(X,\O_\vrho)$ and $\varphi,\theta \in \O(D_\vrho)$, we have
$$
\int (\psi_1\cdots\psi_p|\varphi) \int (\psi_{p+1}\cdots\psi_{p+q} | \theta) = \sum_{\sigma \in Sh(p,q)}\int (\psi_{\sigma(1)}\cdots\psi_{\sigma(p+q)} | \varphi\theta),
$$
where $Sh(p,q)$ denotes the set of shuffles of type $(p,q)$. \qed
\end{proposition}

\begin{corollary}
The sets $I(X)_\vrho$ and $H^0(I(X)_\vrho)$ of functions on $\Omega_{x_0} X$ are $\C$-algebras, where the map $\C \to H^0(I(X)_\vrho)$ is given by $1 \mapsto \int (\phantom{i}|1)$. \qed
\end{corollary}

\begin{remark}
Suppose that $\vrho \colon \pi \to (\C^*)^N$ is the trivial representation and that $\psi_1,\dots,\psi_r \in E^1(X,\O_\vrho)$. Then $D_\vrho$ is the trivial group scheme and $E^\bullet(X,\O_\vrho) = E^\bullet(X)$. Consequently, the iterated integral $\int_\gamma (\psi_1\cdots\psi_r|1)$ is Chen's iterated integral $\int_\gamma \psi_1\cdots\psi_r$, defined by (\ref{ChenIterInt}).
\end{remark}

\subsection{Relative Malcev Completion}
\label{TwistedII}

\noindent In this section, we recall several results by Hain \cite{HainDeRham}. They are generalizations of work by Chen \cite{Chen3}.

Let $X$ be a smooth manifold, and set $\pi = \pi_1(X,x_0)$. Suppose that $\vrho \colon \pi \to (\C^*)^N$ is a representation. Let $D_\vrho$ denote the Zariski closure of the image of $\vrho$ in $\G_m^N$. Consider the bar construction $B(\C,E^\bullet(X,\O_\vrho),\O(D_\vrho))$, which is described in Section \ref{HopfExampleSection}. This Hopf algebra is nonnegatively graded. Thus, $H^0B(\C,E^\bullet(X,\O_\vrho),\O(D_\vrho))$ is a Hopf subalgebra of $B(\C,E^\bullet(X,\O_\vrho),\O(D_\vrho))$. In \cite{HainDeRham}, Hain shows that it is the coordinate ring of a proalgebraic group scheme $\S_\vrho$.

\begin{theorem}
There is a $\C$-algebra isomorphism
$$H^0B(\C,E^\bullet(X,\O_\vrho),\O(D_\vrho)) \overset{\cong}{\longrightarrow} H^0(I(X)_\vrho)$$ given by $[\psi_1|\cdots|\psi_r]\varphi \,\, \longmapsto \,\, \int (\psi_1\cdots\psi_r|\varphi)$. \qed
\end{theorem}

\noindent Note that the group $\S_\vrho(\C)$ consists of the set of $\C$-algebra homomorphisms
$$
H^0B(\C,E^\bullet(X,\O_\vrho),\O(D_\vrho)) \,\, \longrightarrow \,\, \C.
$$
\begin{theorem}
\label{thetavrho}
The map $\theta_\vrho \colon \pi \to \S_\vrho(\C)$ given by
$$
\gamma \,\,\,\, \longmapsto \,\, \biggl [\psi_1|\cdots|\psi_r]\varphi \,\, \mapsto \, \int_\gamma (\psi_1\cdots\psi_r|\varphi)\biggr
$$
is a Zariski dense group homomorphism. \qed
\end{theorem}

In the construction of $B(\C,E^\bullet(X,\O_\vrho),\O(D_\vrho))$, we are viewing $\C$ as an algebra over $E^\bullet(X,\O_\vrho)$ via the composition $E^\bullet(X,\O_\vrho) \to \O(D_\vrho) \overset{\epsilon}{\to} \C$. The map $E^\bullet(X,\O_\vrho) \to \O(D_\vrho)$ is described in Section \ref{HopfExampleSection}. Thus, we may form the reduced bar construction $B(\C,E^\bullet(X,\O_\vrho),\C)$, which is also a differential graded Hopf algebra. Hain \cite{HainDeRham} shows that the Hopf algebra $H^0B(\C,E^\bullet(X,\O_\vrho),\C)$ is the coordinate ring of a prounipotent group scheme $\U_\vrho$. That is,
\begin{equation*}
\begin{aligned}
\O(\S_\vrho) &= H^0B(\C,E^\bullet(X,\O_\vrho),\O(D_\vrho))\\
\O(\U_\vrho) &= H^0B(\C,E^\bullet(X,\O_\vrho),\C).
\end{aligned}
\end{equation*}

\begin{theorem}
There is a natural short exact sequence
$$
1 \longrightarrow \U_\vrho \longrightarrow \S_\vrho \longrightarrow D_\vrho \longrightarrow 1
$$
of affine proalgebraic group schemes over $\C$. \qed
\end{theorem}

The homomorphism $\U_\vrho \to \S_\vrho$ corresponds to the Hopf algebra homomorphism $H^0B(\C,E^\bullet(X,\O_\vrho),\O(D_\vrho)) \to H^0B(\C,E^\bullet(X,\O_\vrho),\C)$ which sends $[\psi_1|\cdots|\psi_r]\varphi$ to $[\psi_1|\cdots|\psi_r]\epsilon(\varphi)$. The homomorphism $\S_\vrho \to D_\vrho$ corresponds to the Hopf algebra homomorphism $\O(D_\vrho) \to H^0B(C,E^\bullet(X,\O_\vrho),\O(D_\vrho))$ which sends $\varphi$ to $[\phantom{i}]\varphi$.

If $\gamma \in \pi$, then the element $\theta_\vrho(\gamma)$ of $\S_\vrho(\C)$ is a $\C$-algebra homomorphism $\O(\S_\vrho) \to \C$. Moreover, the element $\vrho(\gamma)$ of $D_\vrho(\C)$ is a $\C$-algebra homomorphism $\O(D_\vrho) \to \C$. If $\varphi \in \O(D_\vrho)$, then $[\phantom{i}]\varphi \in \O(\S_\vrho)$. Thus, by the definition of $\theta_\vrho$, we have $\theta_\vrho(\gamma)([\phantom{i}]\varphi) = \varphi(\vrho(\gamma))$. It follows that the diagram
$$
\xymatrix{\pi \ar[d]^{\theta_\vrho} \ar[dr]_\vrho \\ \S_\vrho(\C) \ar[r] &\D_\vrho(\C)}
$$
commutes.
\begin{theorem}[Hain, \cite{HainDeRham}]
The proalgebraic group scheme $\S_\vrho$ is the Malcev completion of $\pi$ relative to $\vrho$. \qed
\end{theorem}
By definition, the group scheme $\U_\vrho$ is the prounipotent radical of $\S_\vrho$. In Section \ref{SectionY}, we will generalize this construction to define the Malcev completion relative to any irreducible component of the characteristic variety $\V_{N,m}^i(X)$.

\subsection{When $\S_\vrho$ is Combinatorially Determined}
\label{WhenSpCombDet}

\noindent In this section, we give conditions under which $\S_\vrho$ is combinatorially determined, where $\vrho \in \T$ is a character. We do not know whether $\S_\vrho$ is always combinatorially determined. In Section \ref{CharVarIntPos}, we show that if $\S_\vrho$ is combinatorially determined for all $\vrho \in \T$, then the characteristic variety $\V_m^1(X)$ is combinatorially determined.

Let $X$ denote the complement of an arrangement of hyperplanes in a complex vector space $V$, and let $\h$ denote its holonomy Lie algebra. Set $\pi = \pi_1(X,x_0)$. Let $\h^\wedge$ denote its completion with respect to degree. Then $\h^\wedge$ is the pronilpotent Lie algebra constructed from the differential graded algebra $E^\bullet(X)$ by the methods of rational homotopy theory. Let $A^\bullet$ denote the complexified Orlik-Solomon algebra of $X$. The inclusion $A^\bullet \hookrightarrow E^\bullet(X)$ is a quasi-isomorphism \cite{OS1}. Thus, $\h^\wedge$ can also be constructed from the differential graded algebra $A^\bullet$. The Orlik-Solomon algebra is determined by the intersection poset of the hyperplane arrangement. Thus, the pronilpotent Lie algebra $\h^\wedge$ is also determined by the intersection poset. The Malcev completion $\pi^{\un}$ is the unique prounipotent group whose Lie algebra is $\h^\wedge$. Thus, $\pi^{\un}$ is determined by the intersection poset of the arrangement. For this reason, we say that $\pi^{\un}$ is {\em combinatorially determined}.

It is natural to ask whether, in general, the relative Malcev completion $\S_\vrho$ is combinatorially determined for $\vrho \in \T^N$. The first result in this direction was given in Theorem \ref{IsUnipTheorem1}, which says that if two distinct hyperplanes intersect, then $\S_\vrho \cong D_\vrho \times \pi^{\un}$ for general $\vrho \in \T^N$. For such $\vrho$, the relative Malcev completion $\S_\vrho$ is combinatorially determined.

This result extends to positive dimensional subvarieties of the character torus. Here, we only consider $\vrho \in \T$. That is, we consider characters $\vrho \colon \pi \to \C^*$. This simplifies notation, because each subtorus $Y$ of $\T$ which has positive dimension must contain a character $\vrho$ that has Zariski dense image in $\G_m$. Thus, we can apply Theorem \ref{OSATheorem}.

Choose linear functions $L_j$ on $V$ such that the $j$-th hyperplane is the vanishing set of $L_j$. Set $\omega_j = (2\pi i)^{-1}d L_j/L_j$. This is a closed holomorphic $1$-form on $X$. Set $\bomega = (\omega_1,\dots,\omega_n)$.

Suppose that $\a \in \C^n$, and set $\vrho = \exp(\a\bomega^T)$. This is an element of $\T$. In Section \ref{Avrho}, we constructed a commutative differential graded algebra $\A_{\a}^\bullet$. It is defined by
$$
\A_\a^\bullet = \bigoplus_{k \in \Z} A^\bullet q^{-k}.
$$
The differential is given on the $k$-th component by left multiplication by $-k\a\bomega^T$.
There is a natural inclusion $\A_{\a}^\bullet \hookrightarrow E^\bullet(X,\O_\vrho)$. By Theorem \ref{OSATheorem}, if $\mathbf{V}$ is a vector subspace of $\C^n$, then there is a countable collection $\{\W_j\}_{j \in \Z}$ of proper affine subspaces of $\mathbf{V}$ with the following property. If $\a \in \mathbf{V}-\bigcup_j \W_j$ and $\vrho = \exp(\a\bomega^T)$ has Zariski dense image in $\G_m$, then the induced homomorphism $H^\bullet(\A_{\a}^\bullet) \longrightarrow H^\bullet(X,\O_\vrho)$ is an isomorphism.

\begin{theorem}
\label{CombinatoriallyDetermined}
If $Y$ is a subtorus of $\T$ that contains the trivial character, then the relative Malcev completion $\S_\vrho$ is combinatorially determined for general $\vrho \in Y$.
\end{theorem}

\begin{proof}
If $Y$ has dimension $0$, then $Y$ is the trivial character, and $\S_\vrho$ is the standard Malcev completion of $\pi$. It is therefore combinatorially determined. If $Y$ has positive dimension, then the general $\vrho \in Y$ has Zariski dense image in $\G_m$. This is because the statement that $\vrho \colon \pi \to \C^*$ does not have Zariski dense image in $\G_m$ is equivalent to the statement that the image of $\vrho$ is contained in the roots of unity. Choose a countable collection $\{\Lambda_r\}$ of proper subvarieties of $Y$ such that every $\vrho \in Y - \bigcup_r \Lambda_r$ has Zariski dense image in $\G_m$.

Let $\mathbf{V}$ be the unique vector subspace of $\C^n$ such that the exponential map $\exp \colon H^1(X,\C) \to \T$ takes $\mathbf{V}_\bomega$ onto $Y$, where $\mathbf{V}_\bomega$ is the image of $\mathbf{V}$ under the natural isomorphism $\C^n \overset{\cong}{\longrightarrow} H^1(X,\C)$, given by $\a \mapsto \a\bomega^T$. Then $\mathbf{V}_\bomega$ is the universal cover of $Y$. Let $\V_j$ denote the image of $\W_j$ in $\mathbf{V}_\bomega$. Each affine subspace $\V_j$ exponentiates to a possibly-translated subtorus $\Omega_j$ of $Y$. Each $\Omega_j$ is a proper subtorus of $Y$, since $\dim Y = \dim_\C \mathbf{V} > \dim_\C \W_j = \dim_\C \V_j = \dim \Omega_j$. Suppose that $\vrho \in Y$ and that $\vrho$ does not lie in any $\Lambda_r$ or any $\Omega_j$. Since $\vrho$ does not lie in any $\Lambda_r$, it follows that $\vrho$ has Zariski dense image in $\G_m$. That is, $D_\vrho = \G_m$. We show that the isomorphism class of $\S_\vrho$ may be computed using the combinatorially determined algebra $\A_\a^\bullet$, where $\a \in \mathbf{V}$ satisfies $\vrho = \exp(\a\bomega^T)$. If $\vrho$ does not lie in any $\Lambda_r$ or $\Omega_j$, then $\a$ is necessarily an element of $\mathbf{V} - \bigcup_j \W_j$. Theorem \ref{OSATheorem} therefore implies that the map
$$
H^\bullet(\A_a^\bullet) \longrightarrow E^\bullet(X,\O_\vrho)
$$
is an isomorphism of graded algebras.

Consider the reduced bar construction $B(\C,E^\bullet(X,\O_\vrho),\O(\G_m))$. We have
$$
\O(\S_\vrho) = H^0B(\C,E^\bullet(X,\O_\vrho),\O(\G_m)).
$$
There is natural inclusion $\A_\a^\bullet \hookrightarrow E^\bullet(X,\O_\vrho)$, which is a quasi-isomorphism. Thus, $\O(\G_m)$ inherits the structure of a $\A_\a^\bullet$-module. Consider the reduced bar construction $B(\C,\A_\a^\bullet,\O(\G_m))$. Set
$$
\mathcal{G} = \Spec H^0B(\C,\A_\a^\bullet,\O(\G_m)).
$$
This is an affine group scheme. Proposition \ref{BarQuism} implies that the natural homomorphism $\O(\mathcal{G}) \longrightarrow \O(\S_\vrho)$ is an isomorphism, since $\A_\a^\bullet \hookrightarrow E^\bullet(X,\O_\vrho)$ is a quasi-isomorphism. Thus, there is a natural isomorphism
$$
\S_\vrho \overset{\cong}{\longrightarrow} \mathcal{G}
$$
of group schemes. The group scheme $\mathcal{G}$ is determined by the algebra $A_\a^\bullet$, which depends only on the intersection poset of the arrangement.
\end{proof}

We do not know whether $\S_\vrho$ is always combinatorially determined.

\subsubsection{Characteristic Varieties and the Intersection Poset}
\label{CharVarIntPos}

Theorem \ref{HainMatsuTheorem} suggests that the question of whether $\S_\vrho$ is combinatorially determined is related to the question of whether characteristic varieties are combinatorially determined.

\begin{theorem}
If the isomorphism class of the relative Malcev completion $\S_\vrho$ is combinatorially determined for all $\vrho \in \T$, then the characteristic variety $\V_m^1(X) = \{\vrho \in \T \, | \, \dim_\C H^1(X,\L_\vrho) \geq m\} = 0$ is combinatorially determined.
\end{theorem}

\begin{proof}
Suppose that the relative Malcev completion $\S_\vrho$ is always combinatorially determined for all characters $\vrho \colon \pi_1(X,x_0) \to \C^*$. Let $\{K_1,\dots,K_n\}$ and $\{H_1,\dots,H_n\}$ denote arrangements of hyperplanes in $V$ that have isomorphic intersection posets. We may assume that the ordering of the hyperplanes induces the isomorphism of intersection posets. Set $X = V - \bigcup_{j=1}^n K_j$ and $Z = V - \bigcup_{j=1}^n H_j$. There are natural isomorphisms
\begin{equation}
\label{rhoXY}
H^1(X,\C^*) \cong (\C^*)^n \cong H^1(Z,\C^*)
\end{equation}
determined by the ordering of the hyperplanes. Let $\vec{p}$ be an element of $(\C^*)^n$, and choose characters $\vrho_X \in H^1(X,\C^*)$ and $\vrho_Z \in H^1(Z,\C^*)$ which correspond to $\vec{p}$ via the isomorphisms (\ref{rhoXY}). We will show that $\dim_\C H^1(X,\L_{\vrho_X}) = \dim_\C H^1(Z,\L_{\vrho_Z})$.

Let $\S_{\vrho_X}$ denote the completion of $\pi_1(X,x_0)$ relative to $\vrho$, and let $\S_{\vrho_Z}$ denote the completion of $\pi_1(Z,z_0)$ relative to $\vrho_Z$. Let $D$ denote the Zariski closure of the image of $\vrho_X$ in $\G_m$. Then $D$ is also the Zariski closure of the image of $\vrho_Z$ in $\G_m$. Let $\psi\colon \S_{\vrho_X} \to D$ denote the surjection given by the definition of $\S_\vrho$. The composition $\S_{\vrho_X} \overset{\cong}{\longrightarrow} \S_{\vrho_Z} \to D$ is surjective. Since $D$ is the maximal reductive quotient of $\S_{\vrho_X}$, it follows that there is an automorphism $\phi \colon D \to D$ such that the diagram
$$
\xymatrix{\S_{\vrho_X} \ar[r]^\psi \ar[d]_\cong & D \ar[d]^\phi\\
\S_{\vrho_Z} \ar[r] & D}
$$
commutes. Thus, the diagram
$$
\xymatrix{\S_{\vrho_X} \ar[r]^{\phi \circ \psi} \ar[d]_\cong & D \ar@{=}[d]\\
\S_{\vrho_Z} \ar[r] & D}
$$
commutes. Let $\U_{\vrho_Z}$ denote the kernel of $\S_{\vrho_Z} \to D$, and let $\U_{\vrho_X}$ denote the kernel of $\phi\circ \psi \colon\S_{\vrho_X} \to D$. A simple diagram chase shows that there is an isomorphism $\U_{\vrho_X} \overset{\cong}{\longrightarrow} \U_{\vrho_Z}$ such that the diagram
$$
\xymatrix{1 \ar[r] & \U_{\vrho_X} \ar[d]^\cong \ar[r] & \S_{\vrho_X} \ar[d]^\cong \ar[r]^{\phi \circ \psi} & D \ar@{=}[d] \ar[r] & 1 \\ 1 \ar[r] & \U_{\vrho_Z} \ar[r] & \S_{\vrho_Z} \ar[r] & D \ar[r] & 1}
$$
commutes. Let $\u_{\vrho_X}$ and $\u_{\vrho_Z}$ denote the Lie algebras of $\U_{\vrho_X}$ and $\U_{\vrho_Z}$, respectively. Then the isomorphism $\U_{\vrho_X} \cong \U_{\vrho_Z}$ induces a $D$-equivariant isomorphism $$H_1(\u_{\vrho_X}) \cong H_1(\u_{\vrho_Z}).$$ Theorem \ref{HainMatsuTheorem} implies that as representations of $D$, both $H_1(\u_{\vrho_X})$ and $H_1(\u_{\vrho_Z})$ are direct products of irreducible representations. Recall that $D$ is a group subscheme of $\G_m$. The standard representation of $\G_m$ is the one dimensional representation corresponding to the identity $\G_m \to \G_m$, which is a character of $\G_m$. This restricts to an irreducible representation $\mathbf{L}_{{\pmb{1}}}$ of $D$. By Theorem \ref{HainMatsuTheorem}, the $\mathbf{L}_{{\pmb{1}}}$ isotypical part of the representation $H_1(\u_{\vrho_X})$ of $D$ has dimension $\dim_\C H^1(X,\L_{\vrho_X})$. By the same theorem, the $\mathbf{L}_{{\pmb{1}}}$-isotypical part of the representation $H_1(\u_{\vrho_Z})$ of $D$ has dimension $\dim_\C H^1(Z,\L_{\vrho_Z})$. The result follows.
\end{proof}

\section{Infinite Dimensional Flat Vector Bundles}
\label{IDFVB}

\noindent Let $X$ be a smooth manifold such that $H_1(X,\Z)$ is torsion-free, and set $\pi = \pi_1(X,x_0)$. Then $\T = H^1(X,\C^*)$ is isomorphic to $(\C^*)^{b_1(X)}$. Let $R$ be a commutative $\C$-algebra, and suppose that $\alpha \colon \pi \to R^\times$ is a group homomorphism. This gives $R$ the structure of a left $\pi$-module. There is a corresponding infinite dimensional flat vector bundle $\VV_{R,\alpha}$ over $X$. In this section, we define the de Rham complex $E^\bullet(X,\VV_{R,\alpha})$ and prove some of its basic properties. The monodromy homomorphism of $\VV_{R,\alpha}$ is $\alpha \colon \pi \to R^\times$. The natural example to keep in mind is where $Y$ is an irreducible subvariety of $\T^N$, $R = \O(Y)$, and the homomorphism $\alpha \colon \pi \to \O(Y)^\times$ is given by $\gamma \mapsto f_1^{k_1}\cdots f_N^{k_N}$, where the $k_j$ are integers and $f_j \in \O(Y)$ is defined by $f_j(\vrho) = \rho_j(\gamma)$.

\subsection{Definitions and Basic Properties}
\label{InfiniteDimensionalVec}

\noindent Suppose that $X$ is a smooth manifold, and set $\pi = \pi_1(X,x_0)$. Let $\tau \colon \X \to X$ denote the universal cover of $X$. We say that an open subset $U \subset X$ is {\em evenly covered} if $\tau^{-1}(U)$ is a disjoint union of open subsets of $\X$, each of which maps homeomorphically onto $U$ via $\tau$. Suppose that $R$ is a commutative $\C$-algebra.

Let $E^\bullet_{\fin}(\X,R)$ denote the set of sums
$$
\sum_{j \in J} \psi_j \otimes r_j,
$$
taken over any index set $J$, with $\psi_j \in E^\bullet(\X)$ and $r_j \in R$, such that locally, all but finitely many $\psi_j$ vanish. That is, every point in $\X$ has a neighborhood on which only finitely many $\psi_j$ take a nonzero value. We view the elements of $E^\bullet_{\fin}(\X,R)$ as differential forms on $\X$ with values in $R$. We obviously have $E^\bullet(\X) \otimes_\C R \subset E^\bullet_{\fin}(\X,R)$. The product on $E^\bullet(\X) \otimes
R$ extends to a product on $E^\bullet_{\fin}(\X,R)$. If $R$ has finite dimension, then $E^\bullet_{\fin}(\X,R) = E^\bullet(\X) \otimes_\C R$. The differential on $E^\bullet_{\fin}(\X,R)$ defined by $d(\sum_j \psi_j \otimes r_j) = \sum_j d(\psi_j) \otimes r_j$ is $R$-linear. Thus, $E^\bullet_{\fin}(\X,R)$ is a differential graded algebra over $R$.

Let $\alpha \colon \pi \to R^\times$ be a group homomorphism, where $R^\times$ denotes the group of units in $R$. This determines a left action of $\pi$ on $R$. There is an induced left action of $\pi$ on the trivial bundle $\X \times R \to \X$ via the formula $\gamma \cdot (z,r) = (\gamma \cdot z, \alpha(\gamma^{-1}) r)$. The quotient by this action is a flat vector bundle
$$
\xymatrix{\VV_{R,\alpha} \ar@{=}[r] & (\X \times R)/\pi \ar[d] \\ & X.}
$$
The fiber over each point is a free $R$-module of rank one.

\begin{notation}
If $Y$ is an affine variety and $\alpha \colon \pi \to \O(Y)^\times$ is a group homomorphism, then the resulting flat vector bundle is denoted $\VV_{Y,\alpha}$.
\end{notation}

The action of $\pi$ on the bundle $\X \times R \to \X$ induces a right action of $\pi$ on $E^\bullet_{\fin}(\X,R)$ via
$$
(\sum_j \psi_j \otimes r_j)\cdot \gamma = \sum_j (\gamma^{-1})^*\psi_j \otimes \alpha(\gamma) r_j.
$$

\begin{definition}
Define $E^\bullet(X,\VV_{R,\alpha})$ to be the $\pi$-invariants of $E^\bullet_{\fin}(X,R)$:
$$
E^\bullet(X,\VV_{R,\alpha}) = [E^\bullet_{\fin}(X,R)]^\pi.
$$
\end{definition}

This is a module over $R$, and the differential on $E^\bullet_{\fin}(\X,R)$ restricts to a differential on $E^\bullet(X,\VV_{R,\alpha})$. If $\alpha$ and $\beta$ are homomorphisms $\pi \to R^\times$, then the product $\alpha\beta$ is as well.

\begin{proposition}
The space $E^\bullet(X,\VV_{R,\alpha})$ is a cochain complex of $R$-modules, and the product on $E^\bullet_{\fin}(\X,R)$ restricts to a product
$$
E^\bullet(X,\VV_{R,\alpha}) \otimes_V E^\bullet(X,\VV_{R,\beta}) \, \longrightarrow \, E^\bullet(X,\VV_{R,\alpha\beta}). \qed
$$
\end{proposition}

\begin{example}
Let $R = \C$, and let $\alpha \colon \pi \to \C^*$ be a homomorphism. The flat bundle $\VV_{R,\alpha}$ is the rank one local system on $X$ with monodromy $\alpha$. The complex $E^\bullet(X,\VV_{R,\alpha})$ is the standard complex of differential forms on $X$ with coefficients in $\VV_{R,\alpha}$.
\end{example}

Let $\phi \colon R \to A$ be a $\C$-algebra homomorphism. Then $\phi \circ \alpha$ is a homomorphism $\pi \to A^\times$. Thus, we may form the complex $E^\bullet(X,\VV_{A,\phi \circ \alpha})$. The homomorphism $\phi$ induces a homomorphism
\begin{equation}
\label{ERestriction2}
E^\bullet(X,\VV_{R,\alpha}) \longrightarrow E^\bullet(X,\VV_{A,\phi \circ \alpha})
\end{equation}
of complexes of $R$-modules. The proof of the following proposition is a standard argument using a partition of unity, and it can be found in \cite{Narkawicz}.

\begin{proposition}
\label{PropositionRestrictionSurjection}
If the $\C$-algebra homomorphism $\phi \colon R \to A$ is surjective, then the homomorphism (\ref{ERestriction2}) is surjective. \qed
\end{proposition}

If $\phi \colon R \to A$ is a surjection of $\C$-algebras, then there is a canonical inclusion
$E^\bullet(X,\VV_{R,\alpha}) \otimes_R A \, \hookrightarrow \, E^\bullet(X,\VV_{A,\phi \circ \alpha})$ of complexes of $A$-modules.

\begin{corollary}
\label{DeRhamSpecialization}
If $\phi \colon R \to A$ is a surjection of $\C$-algebras, then the canonical inclusion $E^\bullet(X,\VV_{R,\alpha}) \otimes_R A \hookrightarrow E^\bullet(X,\VV_{A,\phi \circ \alpha})$ is an isomorphism of complexes of $A$-modules. \qed
\end{corollary}

The cohomology of the complex $E^\bullet(X,\VV_{R,\alpha})$ is denoted $H^\bullet(X,\VV_{R,\alpha})$.
In the next sections, we show that this cohomology may be computed using simplicial cochains.

\subsection{Simplicial Cohomology}
\label{SingSimpCoh}

\noindent Recall that $X$ is a smooth manifold and that $\pi = \pi_1(X,x_0)$. By \cite[pages 124-135]{Whitney}, there is a triangulation of $X$. In this section, we show that the cohomology $H^\bullet(X,\VV_{R,\alpha})$ may be computed by simplicial cochains. The proofs of these results can be found in \cite{Narkawicz}.

Let $R$ be a commutative $\C$-algebra, and let $\alpha \colon \pi \to R^\times$ be a group homomorphism. As in the previous section, there is an associated flat vector bundle $\VV_{R,\alpha}$ over $X$. The fiber over each point is a free $R$-module of rank one. Choose a triangulation of $X$ such that each simplex has a neighborhood that is evenly covered by $\tau \colon \X \to X$. This triangulation lifts to a triangulation of $\X$ that is invariant under the action of $\pi$ on $\X$. Let $C^\bullet_\Delta$ denote simplicial cochains. Given a subsimplicial complex $Z \subset X$, let $C^\bullet_\Delta(Z,\VV_{R,\alpha})$ denote the set of sums $\sum_{j \in J} \phi_j \otimes r_j$ taken over any index set $J$, with $\phi_j \in C^\bullet_\Delta(\tau^{-1}(Z))$ and $r_j \in R$, that have the following two properties.
\begin{itemize}
\item On each simplex of $\tau^{-1}(Z)$, all but finitely many $\phi_j$ vanish.
\item For $\gamma \in \pi$, one has $\sum_j (\gamma^{-1})^*\phi_j \otimes \alpha(\gamma) r_j = \sum_j \psi_j \otimes r_j$.
\end{itemize}
It is easy to see that $C^\bullet_\Delta(Z,\VV_{R,\alpha})$ is a cochain complex of $R$-modules. The differential is defined by $d(\sum_j \phi_j \otimes r_j) = \sum_j d(\phi_j) \otimes r_j$. Note that if $\sigma$ is a simplex in $X$, then $C^\bullet_\Delta(\sigma,\VV_{R,\alpha})$ is a free $R$-module of rank one.

If $\phi \colon R \to A$ is a $\C$-algebra homomorphism, there is an induced homomorphism $C^\bullet_\Delta(Z,\VV_{R,\alpha}) \to C^\bullet_\Delta(Z,\VV_{A,\phi \circ \alpha})$ of complexes of $R$-modules. The proof of the next proposition is similar to the proof of Proposition \ref{PropositionRestrictionSurjection}.

\begin{proposition}
\label{ComplexCSpec}
If $\phi \colon R \to A$ is a surjective $\C$-algebra homomorphism, then the induced homomorphism $C^\bullet_\Delta(Z,\VV_{R,\alpha}) \to C^\bullet_\Delta(Z,\VV_{A,\phi \circ \alpha})$ is surjective. \qed
\end{proposition}

\begin{corollary}
\label{SimpSpec}
If $\phi \colon R \to A$ is a surjective $\C$-algebra homomorphism, then the induced homomorphism $C^\bullet_\Delta(Z,\VV_{R,\alpha}) \otimes_{R} A \to C^\bullet_\Delta(Z,\VV_{A,\phi \circ \alpha})$ is an isomorphism of complexes of $A$-modules. \qed
\end{corollary}

\subsection{Equivalence of De Rham and Simplicial Cohomologies}

Suppose that $X$ is a smooth manifold such that $H_1(X,\Z)$ is torsion free, and let $\T = H^1(X,\C^*)$ denote the character torus. Each element of $\T^N$ can be viewed as a representation $\pi \to (\C^*)^N$. Let $Y$ be an irreducible subvariety of $\T^N$, and let $\O(Y)$ denote the coordinate ring of $Y$. Choose $\k \in \Z^N$. There is a homomorphism $\alpha \colon \pi \to \O(Y)^\times$ given by $\gamma \mapsto f_1^{k_1}\cdots f_N^{k_N}$, where $f_j \in \O(Y)^\times$ is given by $f_j(\vrho) = \rho_j(\gamma)$. Then $\VV_{Y,\alpha}$ is a flat vector bundle over $X$, and the fiber over each point is a free $\O(Y)$-module of rank one. If $Y = \{\vrho\}$,  then $\VV_{Y,\alpha}$ is the rank one local system on $X$ with monodromy $\rho_1^{k_1}\cdots \rho_N^{k_N}$. We denote this local system by $\L_{\vrho^\k}$. Corollary \ref{DeRhamSpecialization} implies that there is a canonical isomorphism
$$
E^\bullet(X,\VV_{Y,\alpha}) \otimes_{\O(Y)} \C_\vrho \, \cong \, E^\bullet(X,\L_{\vrho^\k})
$$
of complexes. It induces a homomorphism
$$
H^\bullet(X,\VV_{Y,\alpha}) \otimes_{\O(Y)} \C_\vrho \longrightarrow H^\bullet(X,\L_{\vrho^\k}).
$$
In Section \ref{CohomologySpecialization}, we show that if $X$ deform-retracts onto a finite simplicial complex, then this is an isomorphism for general $\vrho \in Y$. This result therefore holds for the complement of a hyperplane arrangement \cite[Theorem 5.40]{OT}. The proof will use the following proposition.

\begin{proposition}
\label{BigProposition}
If $X$ deform-retracts onto a finite simplicial complex $Z$, then there is a natural isomorphism
$$
H^\bullet(X,\VV_{R,\alpha}) \cong H^\bullet(C^\bullet_\Delta(Z,\VV_{R,\alpha}))
$$
of $R$-modules.
\end{proposition}

\begin{proof}
See \cite{Narkawicz}.
\end{proof}

\subsection{Universal Coefficients and Specialization}
\label{UCSection}

The purpose of this section is to prove several general statements in commutative algebra.
Recall that $Y$ is an irreducible affine variety. Let $\O(Y)$ denote the coordinate ring of $Y$, and let $F$ denote its fraction field. For each $\vrho \in Y$, let $\C_\vrho$ denote the associated residue field. Let $1_F$ and $1_\vrho$ denote the multiplicative identities in $F$ and $\C_\vrho$, respectively. Let $(C_\bullet,\partial_\bullet)$ be a chain complex of $\O(Y)$-modules. In this section, we consider the complex $(C_\bullet \otimes_{\O(Y)} \C_\vrho, \partial_\bullet \otimes 1_\vrho)$. There is a canonical homomorphism
$$
H_\bullet(C_\bullet) \otimes_{\O(Y)} \C_\vrho \longrightarrow H_\bullet(C_\bullet \otimes_{\O(Y)} \C_\vrho),
$$
which is an isomorphism under certain conditions. As in Section \ref{InfiniteDimensionalVec}, associated to any group homomorphism $\alpha \colon \pi_1(X,x_0) \to \O(Y)^\times$ is a flat vector bundle $\VV_{Y,\alpha}$ over $X$ whose fibers are free $\O(Y)$-modules of rank one and whose monodromy homomorphism is $\alpha$. In Section \ref{CohomologySpecialization}, we use the results of the current section to show that if $X$ deform-retracts onto a finite simplicial complex, then the canonical homomorphism
$$
H^\bullet(X,\VV_{Y,\alpha}) \otimes_{\O(Y)} \C_\vrho \longrightarrow H^\bullet(X,\VV_{Y,\alpha} \otimes_{\O(Y)} \C_\vrho)
$$
is an isomorphism for general $\vrho \in Y$.

\begin{remark}
For simplicity of notation, the results in this section are stated in terms of chain complexes. All results hold for cochain complexes as well.
\end{remark}

The next lemma follows from the proofs of Theorems 3.3 and 3.8 of \cite{Rotman}.

\begin{lemma}
\label{LemmaFinGenFree}
If $M$ is a finitely generated $\O(Y)$-module, then there exists a resolution $\dots \to P_1 \to P_0 \to M \to 0$ of $M$ such that each $P_j$ is a finitely generated free $\O(Y)$-module. \qed
\end{lemma}

\begin{lemma}
\label{MTorpLemma}
If $M$ is a finitely generated $\O(Y)$-module, then for each $j \geq 1$, $$\Tor^{\O(Y)}_j(M,\C_\vrho) = 0$$ for generic $\vrho \in Y$.
\end{lemma}

\begin{proof}
Choose a resolution
$$
\dots \longrightarrow P_r \overset{d_r}{\longrightarrow} \dots \overset{d_2}{\longrightarrow} P_1 \overset{d_1}{\longrightarrow} P_0 \overset{d_0}{\longrightarrow} M \longrightarrow 0
$$
of $M$, where each $P_j$ is a finitely generated free $\O(Y)$-module. For $j \geq 0$, the map $d_j$ is represented by a matrix with entries in $\O(Y)$. The complex $P_\bullet \otimes_{\O(Y)} F$ is exact since $F$ is a flat $\O(Y)$-module \cite[Corollary 3.48]{Rotman}. Let $r_j$ denote the rank of the map $d_j \otimes 1_F \colon P_j \otimes_{\O(Y)} F \to P_{j-1} \otimes_{\O(Y)} F$. Then for generic $\vrho \in Y$, the induced map $d_j \otimes 1_\vrho \colon P_j \otimes_{\O(Y)} \C_\vrho \to P_{j-1} \otimes_{\O(Y)} \C_\vrho$ has rank $r_j$. For each $\vrho \in Y$ such that $d_j \otimes 1_\vrho$ has rank $r_j$ and $d_{j+1} \otimes 1_\vrho$ has rank $r_{j+1}$, the sequence
$$
P_{j+1} \otimes_{\O(Y)} \C_\vrho \overset{d_{j+1} \otimes 1_\vrho}{\longrightarrow} P_j \otimes_{\O(Y)} \C_\vrho \overset{d_j \otimes 1_\vrho}{\longrightarrow} P_{j-1} \otimes_{\O(Y)} \C_\vrho
$$
is exact. Thus, $\Tor^{\O(Y)}_j(M,\C_\vrho) = 0$ for such $\vrho$.
\end{proof}

Recall that a property holds for general $\vrho \in Y$ if there are countably many nonempty Zariski open subsets $U_j$ of $Y$ such that the property holds for any $\vrho \in \bigcap_j U_j$. The next lemma follows from Lemma \ref{MTorpLemma} and the fact that the $\Tor$ functor commutes with direct limits.

\begin{lemma}
\label{MTorp}
If $M$ is a countably generated $\O(Y)$-module, then for each $j \geq 1$,
$$
\Tor^{\O(Y)}_j(M,\C_\vrho) = 0
$$
for general $\vrho \in Y$. \qed
\end{lemma}

\begin{theorem}[Important]
\label{GenericCohomDimension}
If $C_\bullet$ is a chain complex of countably generated $\O(Y)$-modules, then the canonical homomorphism
$$
H_\bullet(C_\bullet) \otimes_{\O(Y)} \C_\vrho \longrightarrow H_\bullet(C_\bullet \otimes_{\O(Y)} \C_\vrho)
$$
is an isomorphism for general $\vrho \in Y$.
\end{theorem}

\begin{proof}
Since $\O(Y)$ is a Noetherian ring, submodules of countably generated $\O(Y)$-modules are countably generated. Thus, each $H_j(C_\bullet)$ is a countably generated $\O(Y)$-module. By Lemma \ref{MTorp}, for general $\vrho \in Y$,
\begin{equation}
\label{TorCj}
\Tor^{\O(Y)}_i(C_j,\C_\vrho) = 0
\end{equation}
and
\begin{equation}
\label{TorHj}
\Tor^{\O(Y)}_i(H_j(C_\bullet),\C_\vrho) = 0
\end{equation}
for all $j \geq 0$ and $i \geq 1$. Choose any such $\vrho$. Let $\dots \to P_1 \to P_0 \to \C_\vrho \to 0$ be a projective resolution of $\C_\vrho$ as $\O(Y)$-modules.

Let $E^n$ denote the homology spectral sequence associated to the double complex $C_\bullet \otimes_{\O(Y)} P_\bullet$. It has $E^0$-term $E^0_{s,t} = C_s \otimes_{\O(Y)} P_t$. Its $E^1$ term is given by $E^1_{s,t} = \Tor^{\O(Y)}_t(C_s,\C_\vrho)$. By Equation (\ref{TorCj}), for general $\vrho \in Y$, the $E^1$ term is concentrated in $t = 0$ and is given by the complex $C_s \otimes_{\O(Y)} \C_\vrho$. For such $\vrho$, the $E^2$ term is concentrated in $t = 0$ and is given by $H_s(C_\bullet \otimes_{\O(Y)} \C_\vrho)$.

Every projective module is flat. Thus, each $P_t$ is a flat $\O(Y)$-module. This implies that the $'E^1$ term is given by $'E^1_{s,t} = H_s(C_\bullet) \otimes_{\O(Y)} P_t$. Thus, the $'E^2$ term is given by $'E^2_{s,t} = \Tor^{\O(Y)}_t(H_s(C_\bullet),\C_\vrho)$. By Equation (\ref{TorHj}), for general $\vrho \in Y$, the $'E^2$ term is concentrated in $t = 0$ and is given by $H_s(C_\bullet) \otimes_{\O(Y)} \C_\vrho$. The result follows.
\end{proof}

\subsection{Specialization of Cohomology}
\label{CohomologySpecialization}

\noindent Let $X$ be a smooth manifold such that $H_1(X,\Z)$ is torsion-free. Then $\T = H^1(X,\C^*)$ is a torus. Set $\pi = \pi_1(X,x_0)$. Each element of $\T^N$ can be viewed as a homomorphism $\pi \to (\C^*)^N$. Let $Y$ be an irreducible subvariety of $\T^N$, and let $\O(Y)$ denote its coordinate ring. Let $\alpha \colon \pi \to \O(Y)^\times$ be a group homomorphism. As in Section \ref{InfiniteDimensionalVec}, there is an associated flat vector bundle $\VV_{Y,\alpha}$ over $X$. The fiber over each point is a free $\O(Y)$-module of rank one. Let $\phi_\vrho \colon \O(Y) \to \C_\vrho$ denote the canonical surjection. Let $\L_{\vrho,\alpha}$ denote the rank one local system on $X$ with monodromy $\phi_\vrho \circ \alpha$. Recall that Corollary \ref{DeRhamSpecialization} implies that if $\vrho \in Y$, then the natural homomorphism $$E^\bullet(X,\VV_{Y,\alpha}) \otimes_{\O(Y)} \C_\vrho \longrightarrow E^\bullet(X,\L_{\vrho,\alpha})$$
is an isomorphism of complexes. It induces a homomorphism $H^\bullet(X,\VV_{Y,\alpha}) \otimes_{\O(Y)} \C_\vrho \longrightarrow H^\bullet(X,\L_{\vrho,\alpha})$. The next theorem will follow from Proposition \ref{BigProposition} and Theorem \ref{GenericCohomDimension}.

\begin{theorem}[Important]
\label{ImportantCohomology}
If $X$ deform-retracts onto a finite simplicial complex, then for general $\vrho \in Y$, the natural homomorphism
$$
H^\bullet(X,\VV_{Y,\alpha}) \otimes_{\O(Y)} \C_\vrho \longrightarrow H^\bullet(X,\L_{\vrho,\alpha}).
$$
is an isomorphism.
\end{theorem}

\begin{proof}
Let $K$ be a finite simplicial complex which is deformation retraction of $X$. Corollary \ref{SimpSpec} implies that there is a natural isomorphism
$$
C^\bullet_\Delta(K,\VV_{Y,\alpha}) \otimes_{\O(Y)} \C_\vrho \cong C^\bullet_\Delta(K,\L_{\vrho,\alpha})
$$
of complexes. By Proposition \ref{BigProposition}, it therefore suffices to show that the canonical homomorphism
$$
H^\bullet(C^\bullet_\Delta(K,\VV_{Y,\alpha})) \otimes_{\O(Y)} \C_\vrho \longrightarrow H^j(C^\bullet_\Delta(K,\VV_{Y,\alpha}) \otimes_{\O(Y)} \C_\vrho)
$$
is an isomorphism for general $\vrho \in Y$. Since $C^\bullet_\Delta(K,\VV_{Y,\alpha})$ is a complex of finitely generated $\O(Y)$-modules, this follows from Theorem \ref{GenericCohomDimension}.
\end{proof}

\begin{remark}
The word ``general" in the statement of the theorem cannot be removed. If $X = \C^*$ and $Y = H^1(X,\C^*)$, then there is a homomorphism $\alpha = \pi_1(X,x_0) \to \O(Y)^\times$ given by $\gamma \longmapsto (\vrho \mapsto \vrho(\gamma))$. Thus, $H^0(X,\VV_{Y,\alpha}) = 0$, but for the trivial character $1 \in Y$, we have $H^0(X,\L_{1,\alpha}) = H^0(X,\C) = \C$.
\end{remark}

\begin{corollary}
If $X$ is the complement of an arrangement of hyperplanes in a complex vector space, then for general $\vrho \in Y$, the natural homomorphism
$$
H^\bullet(X,\VV_{Y,\alpha}) \otimes_{\O(Y)} \C_\vrho \longrightarrow H^\bullet(X,\L_{\vrho,\alpha}).
$$
is an isomorphism.
\end{corollary}

\begin{proof}
By \cite[Theorem 5.40]{OT}, there is a strong deformation retraction of $X$ onto a finite simplicial complex. Thus, Theorem \ref{ImportantCohomology} applies.
\end{proof}

\begin{example}
For each $\k \in \Z^N$, there is a homomorphism $\alpha \colon \pi \to \O(Y)^\times$ given by $\gamma \mapsto f_1^{k_1}\cdots f_N^{k_N}$, where $f_j \in \O(Y)^\times$ is defined by $f_j(\vrho) = \rho_j(\gamma)$. The local system $\L_{\vrho,\alpha}$ has monodromy $\vrho_1^{k_1}\cdots \vrho_N^{k_N}$. We therefore denote this local system by $\L_{\vrho^\k}$. If $X$ deform-retracts onto a finite simplicial complex, then for general $\vrho \in Y$, there is a canonical isomorphism
$$
H^\bullet(X,\VV_{Y,\alpha}) \otimes_{\O(Y)} \C_\vrho \cong H^\bullet(X,\L_{\vrho^\k}).
$$
It is natural with respect to cup products.
\end{example}

\section{Constancy of Relative Malcev Completion}
\label{SectionY}

Let $X$ be a smooth manifold such that $H_1(X,\Z)$ is torsion-free. Then $\T = H^1(X,\C^*)$ is a complex torus. Set $\pi = \pi_1(X,x_0)$. Each $\vrho \in \T^N$ can be viewed as a homomorphism $\pi \to (\C^*)^N$. Let $Y$ be an irreducible subvariety of $\T^N$. One natural and interesting choice for $Y$ is an irreducible component of the characteristic variety $\V_{N,m}^i(X)$.

In this section, we will construct an affine group scheme $\S_Y$ over $Y$ and a homomorphism $\theta_Y \colon \pi \to \S_Y(\O(Y))$ using iterated integrals that generalize those of Chen \cite{Chen3} and Hain \cite{HainDeRham}. When $Y = \{\vrho\}$, the group scheme $\S_Y$ is the Malcev completion of $\pi$ relative to $\vrho$. In addition, for each irreducible subvariety $Z$ of $Y$, there is a canonical homomorphism $\S_Z \to \S_Y \otimes_{\O(Y)} \O(Z)$ of group schemes such that the diagram
$$
\xymatrix{\pi \ar[r]^{\theta_Z} \ar[d]_{\theta_Y} & \S_Z(\O(Z)) \ar[d] \\ \S_Y(\O(Y)) \ar[r] & \S_Y(\O(Z))}
$$
commutes.

If $X$ deform-retracts onto a finite simplicial complex (e.g. if $X$ is the complement of a hyperplane arrangement), then the relative Malcev completion $\S_\vrho$ is constant over $Y$ in the following sense. If there exists $\pmb{\varrho} \in Y$ such that the Zariski closure of the image of $\pmb{\varrho}$ in $\G_m^N$ contains $\im\vrho$ for every $\vrho \in Y$, then for general $\vrho \in Y$, the homomorphism $\S_\vrho \to \S_Y \otimes_{\O(Y)} \C_\vrho$ is an isomorphism. This holds for any irreducible subvariety of the character torus $\T$.

\subsection{The Group Schemes $G_Y$ and $D_Y$}
\label{DY}

Recall that $X$ is a smooth manifold such that $H_1(X,\Z)$ is torsion-free, that $\T = H^1(X,\C^*)$, and that $\pi = \pi_1(X,x_0)$. Let $Y$ be an irreducible subvariety of $\T^N$. The coordinate ring of $Y$ is denoted $\O(Y)$. There is a tautological homomorphism
$$
\vrho_Y \colon \pi \to (\O(Y)^\times)^N
$$
into the $\O(Y)$-rational points of $\G_m^N$ defined by $\gamma \mapsto (f_1,\dots,f_N)$, where $f_j(\vrho) = \rho_j(\gamma)$ for $\vrho \in Y$. Via the group isomorphism $\G_m^N(\C_\vrho) \cong (\C^*)^N$, the composition
$$
\pi \overset{\vrho_Y}{\longrightarrow}\G_m^N(\O(Y)) \longrightarrow \G_m^N(\C_\vrho)
$$
is given by $\vrho$.

Define $G_Y$ to be the intersection of all group subschemes $G$ of $\G_m^N$ over $\C$ such that $\im \vrho \subset G(\C)$ for every $\vrho \in Y$. This is a group subscheme of $\G_m^N$ over $\C$. Set
$$
D_Y = G_Y \otimes_{\C} \O(Y).
$$
By definition, $\O(D_Y) = \O(G_Y) \otimes_\C \O(Y)$, and $D_Y$ is a group subscheme of $\G_{m/Y}^N$. The image of $\vrho_Y$ is contained in $G_Y(\O(Y))$, which is equal to $D_Y(\O(Y))$.

\begin{remark}
Suppose that $Y = \{\vrho\}$. Then $\O(Y)$ is the residue field $\C_\vrho$, and $G_Y$ is the Zariski closure of the image of $\vrho$ in $\G_m^N$. That is, if $\vrho \in Y$, then $D_\vrho = G_\vrho$.
\end{remark}

Each character $\alpha \in G_Y^\vee$ induces a homomorphism $D_Y \to \G_{m/Y}$ of affine group schemes over $Y$. This homomorphism corresponds to a Hopf algebra homomorphism $\alpha^* \colon \O(Y)[q^{\pm 1}] \to \O(D_Y)$. There is an $\O(Y)$-module map $\O(Y) \to \O(D_Y)$ given by $1 \mapsto \alpha^*(q)$. This results in a $D_Y$-module $V_{Y,\alpha}$, which is a free $\O(Y)$-module of rank one. There is a homomorphism $\pi \longrightarrow \O(Y)^\times$ given by the composition
$$
\pi \overset{\vrho_Y}{\longrightarrow} G_Y(\O(Y)) \overset{\alpha}{\longrightarrow} \G_m(\O(Y)).
$$
This induces a left action of $\pi$ on $\O(Y)$, where an element $\gamma$ of $\pi$ acts by multiplication by $(\alpha \circ \vrho_Y)(\gamma)$. As in Section \ref{InfiniteDimensionalVec}, there is a left action of $\pi$ on the trivial bundle $\X \times \O(Y) \to \X$ defined by $\gamma \cdot (z,v) = (\gamma \cdot z, \gamma^{-1} \cdot v)$. The quotient
$$
\xymatrix{\VV_{Y,\alpha} \ar@{=}[r] & \X \times \O(Y) \ar[d] \\ & X}
$$
is a flat vector bundle over $X$; each fiber is a free $\O(Y)$-module of rank one. The differential forms $E^\bullet(X,\VV_{Y,\alpha})$ on $X$ with coefficients in $\VV_{Y,\alpha}$ are defined to be sums $\sum_{j \in J} \psi_j \otimes v_j$ taken over any index set $J$, with $\psi_j \in E^\bullet(\X)$ and $v_j \in \O(Y)$, which have the following two properties.
\begin{itemize}
\item Each point of $\X$ has a neighborhood on which all but finitely many $\psi_j$ vanish identically.
\item For each $\gamma \in \pi$, one has $\sum_j \psi_j \otimes v_j = \sum_j (\gamma^{-1})^*\psi_j \otimes \gamma \cdot v_j$.
\end{itemize}
The monodromy representation of $\VV_{Y,\alpha}$ is $\alpha\circ \vrho_Y \colon \pi \to \O(Y)^\times$.

\begin{remark}
Let $Y = \{\vrho\}$ and $\k \in Z^N$. Let the character $\alpha$ of $G_Y$ be the restriction to $G_Y$ of the $\k$-th standard character on $\G_m^N$. Let $\L_{\vrho^\k}$ denote the rank one local system on $X$ with monodromy $\rho_1^{k_1}\cdots \rho_N^{k_N}$. The complex $E^\bullet(X,\VV_{Y,\alpha})$ is the de Rham complex $E^\bullet(X,\L_{\vrho^\k})$.
\end{remark}

\subsection{The Algebra $E^\bullet(X,\O_Y)$}
\label{CommDiffEY}

Recall that $X$ is a smooth manifold such that $H_1(X,\Z)$ is torsion-free, that $\T = H^1(X,\C^*)$, and that $\pi = \pi_1(X,x_0)$. Let $Y$ be an irreducible subvariety of $\T^N$. As in the previous section, each $\alpha \in G_Y^\vee$ gives a $D_Y$-module $V_{Y,\alpha}$ and a flat vector bundle $\VV_{Y,\alpha}$ over $X$. The fiber over each point is a free $\O(Y)$-module of rank one. The monodromy representation of $\VV_{Y,\alpha}$ is $\alpha \circ \vrho_Y \colon \pi \to \O(Y)^\times$, where $\vrho_Y \colon \pi \to G_Y(\O(Y))$ is the tautological homomorphism.

Given $\alpha, \beta \in G_Y^\vee$, the cup product of an element of $E^\bullet(X,\VV_{Y,\alpha})$ with an element of $E^\bullet(X,\VV_{Y,\beta})$ lies in $E^\bullet(X,\VV_{Y,\alpha\beta})$. Via the construction in Section \ref{GenConstructionAlg}, we can define the commutative differential graded $\O(Y)$-algebra $E^\bullet(X,\O_Y)$:
\begin{equation}
\label{CommDiffY}
E^\bullet(X,\O_Y) = \bigoplus_{\alpha \in G_Y^\vee} E^\bullet(X,\VV_{Y,\alpha}) \otimes_{\O(Y)} V_{Y,\alpha^{-1}}.
\end{equation}
The differential is defined componentwise by the differentials on the complexes $E^\bullet(X,\VV_{Y,\alpha})$, and the grading is determined by the degree of differential forms. The $\O(Y)$-algebra structure is determined by the $\O(Y)$-module structure on each $E^\bullet(X,\VV_{Y,\alpha})$. The $V_{Y,\alpha^{-1}}$ in the summand implies that $E^\bullet(X,\O_Y)$ has the structure of a $D_Y$-module.

\begin{example}
If $Y = \{\vrho\}$, then $G_Y$ is the Zariski closure of the image of $\vrho$ in $\G_m^N$. Thus, the algebra $E^\bullet(X,\O_\vrho)$ is the same as the commutative differential graded algebra defined in Section \ref{CommDiffE}. Let $\U_\vrho$ denote the prounipotent radical of the relative Malcev completion $\S_\vrho$, and let $\u_\vrho$ denote the Lie algebra of $\U_\vrho$. This is representation of $D_\vrho$. The commutative differential graded algebra $E^\bullet(X,\O_\vrho)$ determines the Lie algebra $\u_\vrho$ and the action of $D_\vrho$ on it by standard methods of rational homotopy theory.
\end{example}

\begin{example}
Suppose that there exists $\vrho \in Y$ such that $\vrho$ has Zariski dense image in $\G_m^N$. Thus, $G_Y = \G_m^N$ and $D_Y = \G_{m/Y}^N$. As in Section \ref{InfiniteDimensionalVec}, given $\k \in \Z^N$, let $\L_{Y^\k}$ denote the flat vector bundle over $X$ whose fibers are free $\O(Y)$-modules of rank one and whose monodromy representation $\pi \to \O(Y)^\times$ is given by $\gamma \mapsto f_1^{k_1}\cdots f_N^{k_N}$, where $f_j(\vrho) = \rho_j(\gamma)$. Let $q_j$ denote the $j$-th standard character on $\G_{m/Y}^N$. The algebra $E^\bullet(X,\O_Y)$ has the following description as a $\G_{m/Y}^N$-module.
$$
E^\bullet(X,\O_Y) = \bigoplus_{\k \in \Z^N} E^\bullet(X,\L_{Y^\k}) q_1^{-k_1}\cdots q_N^{-k_N}.
$$
The $q_j$ determine the action of $\G_{m/Y}^N$ on $E^\bullet(X,\O(Y))$ via the action on its coordinate ring. For $\vrho \in Y$ and $\k \in \Z^N$, let $\L_{\vrho^\k}$ denote the rank one local system on $X$ with monodromy $\rho_1^{k_1}\cdots \rho_N^{k_N}$. By Corollary \ref{DeRhamSpecialization} , for each $\vrho \in Y$ which has Zariski dense image in $\G_m^N$, there is a canonical isomorphism
$$
E^\bullet(X,\O_Y) \otimes_{\O(Y)} \C_\vrho \cong \bigoplus_{\k \in \Z^N} E^\bullet(X,\L_{\vrho^\k}) q_1^{-k_1}\cdots q_N^{-k_N}
$$
of commutative differential graded $\C$-algebras.
\end{example}

The cohomology of the commutative differential graded algebra $E^\bullet(X,\O_Y)$ is denoted $H^\bullet(X,\O_Y)$. There is a canonical isomorphism
$$
H^\bullet(X,\O_Y) \cong \bigoplus_{\alpha \in G_Y^\vee} H^\bullet(X,\VV_{Y,\alpha}) \otimes_{\O(Y)} V_{Y,\alpha^{-1}}.
$$
of $\O(Y)$-algebras. Thus, the algebra $H^\bullet(X,\O_Y)$ is a $D_Y$-module.

\begin{theorem}
The algebra $E^\bullet(X,\O_Y)$ has connected cohomology:
$$
H^0(X,\O_Y) = \O(Y).
$$
\end{theorem}

\begin{proof}
For the trivial character $1 \colon G_Y \to \G_m$, one always has $H^0(X,\VV_{Y,1}) = \O(Y)$.
It suffices to prove that if $\alpha \in G_Y^\vee$ is nontrivial, then the monodromy representation of $\VV_{Y,\alpha}$ is nontrivial. That is, if $\alpha \colon G_Y \to \G_m$ is nontrivial, then
$$
\alpha \circ \vrho_Y \colon \pi \to \O(Y)^\times
$$
is nontrivial, where $\vrho_Y \colon \pi \to G_Y(\O(Y))$ is the tautological homomorphism. Recall that $G_Y$ is defined to be the intersection of all group subschemes of $\G_m^N$ that contain $\im \vrho$ for every $\vrho \in Y$. Thus, $D_\vrho$ is a group subscheme of $G_Y$ for each $\vrho \in Y$. Thus, the character $\alpha$ of $G_Y$ restricts to a character on $D_\vrho$. The diagram
$$
\xymatrix{\pi \ar[d]_\vrho \ar[r]^-{\vrho_Y} & G_Y(\O(Y)) \ar[r]^-\alpha & \G_{m/Y}(\O(Y)) \ar[d] \\ D_\vrho(\C) \ar[rr]^{\alpha} && \G_m(\C_\vrho)}
$$
commutes.

If $\alpha \circ \vrho_Y$ is trivial, then $\alpha \circ \vrho \colon \pi \to \G_m(\C)$ is trivial for each $\vrho \in Y$. Since the image of $\vrho$ is dense in $D_\vrho$, this implies that $\alpha \colon D_\vrho(\C) \to \G_m(\C_\vrho)$ is trivial for each $\vrho \in Y$. Over $\C$, an algebraic group scheme is uniquely determined by its group of $\C$-rational points. Thus, $D_\vrho \subset \ker \alpha$ for every $\vrho \in Y$. But $\ker \alpha$ is a group subscheme of $G_Y$. By the definition of $G_Y$, this implies that $G_Y = \ker \alpha$. Thus, $\alpha$ is trivial.
\end{proof}

\subsection{Specialization of Coefficients}
\label{CoeffSpecializationV}

\noindent Suppose now that $Z$ is an irreducible subvariety of $Y$. Then $G_Z \subset G_Y$, and the diagram
$$
\xymatrix{\pi \ar[d]_{\rho_Y} \ar[r]^-{\rho_Z} & G_Z(\O(Z))  \ar[d] \\ G_Y(\O(Y)) \ar[r] & G_Y(\O(Z))}
$$
commutes. Suppose that $\alpha$ is a character on $G_Y$. This restricts to a character $\alpha|_Z$ on $G_Z$. Thus, there is a canonical homomorphism $G_Y^\vee \to G_Z^\vee$ of groups, given by restriction. The complex $E^\bullet(X,\VV_{Z,\alpha|_Z})$ is a complex of $\O(Z)$-modules. By Proposition \ref{PropositionRestrictionSurjection}, there is a canonical surjection
$$
E^\bullet(X,\VV_{Y,\alpha}) \longrightarrow E^\bullet(X,\VV_{Z,\alpha|_Z})
$$
of complexes of $\O(Y)$-modules. Corollary \ref{DeRhamSpecialization} implies that it induces a canonical isomorphism
$$
E^\bullet(X,\VV_{Y,\alpha}) \otimes_{\O(Y)} \O(Z) \overset{\cong}{\longrightarrow} E^\bullet(X,\VV_{Z,\alpha|_Z})
$$
of complexes of $\O(Z)$-modules.

Recall that $D_Y = G_Y \otimes_\C \O(Y)$ and $D_Z = G_Z \otimes_\C \O(Z)$. Since $G_Z \subset G_Y$, $D_Z$ is a group subscheme of $D_Y \otimes_{\O(Y)} \O(Z)$. The $\O(Y)$-algebra $E^\bullet(X,\O_Y)$ is a $D_Y$-module. Thus, the $\O(Z)$-algebra $E^\bullet(X,\O_Y) \otimes_{\O(Y)} \O(Z)$ is a $D_Z$-module. The next proposition and corollary follow immediately.

\begin{proposition}
There is a canonical $D_Z$-equivariant homomorphism
$$
E^\bullet(X,\O_Y) \longrightarrow E^\bullet(X,\O_Z)
$$
of commutative differential graded $\O(Y)$-algebras. It is a surjection if $G_Z = G_Y$. \qed
\end{proposition}

\begin{corollary}
\label{YZSpecCorollary}
There is a canonical $D_Z$-equivariant inclusion
$$
E^\bullet(X,\O_Y) \otimes_{\O(Y)} \O(Z) \hookrightarrow E^\bullet(X,\O_Z)
$$
of commutative differential graded $\O(Y)$-algebras. It is an isomorphism if $G_Z = G_Y$. \qed
\end{corollary}

It follows from this corollary that there is a canonical homomorphism
$$
H^\bullet(X,\O_Y) \otimes_{\O(Y)} \O(Z)\, \longrightarrow \, H^\bullet(X,\O_Z)
$$
of graded $\O(Z)$-algebras. The next proposition will follow directly from Theorem \ref{ImportantCohomology} and the fact that $G_Y^\vee$ is countable.

\begin{theorem}
\label{HYpTheorem}
If $X$ deform-retracts onto a finite simplicial complex, then for general $\vrho \in Y$, the natural homomorphism
\begin{equation}
\label{isom2233}
H^\bullet(X,\O_Y) \otimes_{\O(Y)} \C_\vrho \, \longrightarrow \, H^\bullet(X,\O_\vrho)
\end{equation}
is an isomorphism when $D_\vrho = G_Y$.
\end{theorem}

\begin{proof}
Recall that $D_\vrho = G_\vrho$ for every $\vrho \in Y$. We have the following two equalities.
\begin{equation*}
\begin{aligned}
H^\bullet(X,\O_Y) &= \bigoplus_{\alpha \in G_Y^\vee} H^\bullet(X,\VV_{Y,\alpha}) \otimes_{\O(Y)} V_{Y,\alpha^{-1}}\\
H^\bullet(X,\O_\vrho) &= \bigoplus_{\beta \in D_\vrho^\vee} H^\bullet(X,\VV_{\vrho,\alpha}) \otimes_\C V_{\rho,\alpha^{-1}}
\end{aligned}
\end{equation*}
The group scheme $D_\vrho$ is a group subscheme of $D_Y \otimes_{\O(Y)} \C_\vrho$, and the restriction to $D_\vrho$ of the representation $V_{Y,\alpha^{-1}} \otimes_{\O(Y)} \C_\vrho$ of $D_Y \otimes_{\O(Y)} \C_\vrho$ is the representation $V_{\vrho,\beta^{-1}}$, where $\beta$ is the restriction of $\alpha$ to the subscheme $D_\vrho$ of $G_Y$. The homomorphism (\ref{isom2233}) is given on the $\alpha$-th component by the canonical homomorphism
\begin{equation}
\label{Hom2}
H^\bullet(X,\VV_{Y,\alpha}) \otimes_{\O(Y)} \C_\vrho \longrightarrow H^\bullet(X,\VV_{\vrho,\beta}).
\end{equation}
By Theorem \ref{ImportantCohomology}, for a fixed $\alpha \in G_Y^\vee$, the map (\ref{Hom2}) is an isomorphism for general $\vrho \in Y$. Thus, since $G_Y^\vee$ is countable, for general $\vrho \in Y$, the homomorphism (\ref{Hom2}) is an isomorphism for all $\alpha \in G_Y^\vee$. For any such $\vrho$, if $D_\vrho = G_Y$, then $D_\vrho^\vee = G_Y^\vee$ and the natural homomorphism
$$
H^\bullet(X,\O_Y) \otimes_{\O(Y)} \C_\vrho \, \longrightarrow \, H^\bullet(X,\O_\vrho)
$$
is an isomorphism on the $\alpha$-th component for all $\alpha \in G_Y^\vee$. Since $D_\vrho^\vee = G_Y^\vee$, it is therefore an isomorphism.
\end{proof}

\begin{remark}
The statement that $D_\vrho = G_Y$ is equivalent to saying that for every $\pmb{\varrho} \in Y$, $\im\pmb{\varrho} \subset D_\vrho(\C)$.
\end{remark}

The remark following Theorem \ref{ImportantCohomology} implies that we cannot remove the word ``general" in the statement of this theorem.

\subsection{$G_Z = G_Y$ is Not Too Restrictive}
\label{GZGY}

\noindent The results in the previous section indicate that for an irreducible subvariety $Z$ of $Y$ such that $G_Z = G_Y$, the restriction homomorphism $E^\bullet(X,\O_Y) \longrightarrow E^\bullet(X,\O_Z)$ is well behaved. In this case, the canonical homomorphism
$$
E^\bullet(X,\O_Y) \otimes_{\O(Y)} \O(Z) \longrightarrow E^\bullet(X,\O_Z)
$$
is an isomorphism of commutative differential graded $\O(Z)$-algebras. There is an induced homomorphism
$$
H^\bullet(X,\O_Y) \otimes_{\O(Y)} \O(Z) \longrightarrow H^\bullet(X,\O_Z)
$$
of graded $\O(Z)$-algebras.

Let $Y$ denote any irreducible subvariety of the character torus $\T$. If $K$ is a subgroup of $\C^*$ that is not Zariski dense in $\G_m$, then $K$ must be a finite subgroup of the roots of unity. Thus, if $Y$ has positive dimension, then there exists $\vrho \in Y$,
$$
\vrho \colon \pi_1(X,x_0) \longrightarrow \C^*,
$$
that has Zariski dense image in $\G_m$. Hence, $G_Y = \G_m$. Thus, if $Z$ is an irreducible subvariety of $Y$ that has positive dimension, then $G_Z = G_Y$. In particular, we have the following proposition.

\begin{proposition}
If $Y$ is a subvariety of $\T$, then $D_\vrho = G_Y$ for general $\vrho \in Y$. \qed
\end{proposition}

\subsection{The Affine Group Scheme $\S_Y$}
\label{SYDefinition}

Recall that $X$ is a smooth manifold such that $H_1(X,\Z)$ is torsion-free, that $\T = H^1(X,\C^*)$, and that $\pi = \pi_1(X,x_0)$. Suppose that $Y$ is an irreducible subvariety of $\T^N$, where $\T$ is the character torus of $X$. There is a tautological homomorphism
$$
\vrho_Y \colon \pi \to (\O(Y)^\times)^N
$$
into the $\O(Y)$-rational points of $\G_m^N$, defined by $\gamma \mapsto (f_1,\dots,f_N)$, where $f_j(\vrho) = \rho_j(\gamma)$ for $\vrho \in Y$. The image of $\vrho_Y$ is contained in $G_Y(\O(Y))$. For $\vrho \in Y$, the composition
$$
\pi \overset{\vrho_Y}{\longrightarrow}\G_m^N(\O(Y)) \longrightarrow \G_m^N(\C_\vrho)
$$
is given by $\vrho$. In Section \ref{CommDiffEY}, we constructed a commutative differential graded algebra $E^\bullet(X,\O_Y)$ over $\O(Y)$. Recall that $E^\bullet(X,\O_Y)$ is defined to be the direct sum
$$
E^\bullet(X,\O_Y) = \bigoplus_{\alpha \in G_Y^\vee} E^\bullet(X,\VV_{Y,\alpha}) \otimes_{\O(Y)} V_{Y,\alpha^{-1}}.
$$
Each $\alpha \in G_Y^\vee$ induces a homomorphism $D_Y \to \G_{m/Y}$ of affine group schemes over $Y$. This corresponds to a Hopf algebra homomorphism $\alpha^* \colon \O(Y)[q^{\pm 1}] \to \O(D_Y)$. There is an injective group homomorphism $G_Y^\vee \to \O(D_Y)^\times$ that takes $\alpha$ to $\alpha^*(q)$. The action by $D_Y$ on $E^\bullet(X,\O_Y)$ corresponds to the $\O(Y)$-algebra homomorphism
$$
\nu \colon E^\bullet(X,\O_Y) \to E^\bullet(X,\O_Y) \otimes_{\O(Y)} \O(D_Y)
$$
that sends $\psi \in E^\bullet(X,\VV_{Y,\alpha}) \otimes_{\O(Y)} V_{Y,\alpha^{-1}}$ to $\psi \otimes (\alpha^{-1})^*(q)$. It follows from definitions that the map $\nu$ satisfies the hypotheses at the beginning of Section \ref{BarHopfAlgebra}. Thus, we may form the reduced bar construction $B(\O(Y),E^\bullet(X,\O_Y),\O(D_Y))$, which is a differential graded Hopf algebra over $\O(Y)$. Thus, the cohomology $H^0(\O(Y),E^\bullet(X,\O_Y),\O(D_Y))$ is a Hopf algebra over $\O(Y)$. Define the affine group scheme $\S_Y$ over $Y$ by
$$
\S_Y = \Spec H^0B(\O(Y),E^\bullet(X,\O_Y),\O(D_Y)).
$$
As is standard for affine group schemes, we write
$$
\O(\S_Y) = H^0B(\O(Y),E^\bullet(X,\O_Y),\O(D_Y)).
$$

\begin{proposition}
There is a canonical surjection $\S_Y \to D_Y$ of affine group schemes over $Y$.
\end{proposition}

\begin{proof}
This map corresponds to the homomorphism $$\O(D_Y) \to H^0B(\O(Y),E^\bullet(X,\O_Y),\O(D_Y))$$ that sends $\varphi$ to $[\phantom{i}]\varphi$. This Hopf algebra homomorphism is injective.
\end{proof}

If $Z$ is an irreducible subvariety of $Y$, then Corollary \ref{YZSpecCorollary} implies that there is a canonical homomorphism
\begin{equation}
\label{BarSpecYZ1}
B(\O(Y),E^\bullet(X,\O_Y),\O(D_Y)) \otimes_{\O(Y)} \O(Z) \,  \longrightarrow \, B(\O(Z),E^\bullet(X,\O_Z),\O(D_Z))
\end{equation}
Note that Proposition \ref{BarSpecializationProposition} and Corollary \ref{YZSpecCorollary} imply that this is an isomorphism if $G_Z = G_Y$. By the definition of the Hopf algebra $\O(\S_Y)$, there is a canonical homomorphism
\begin{equation}
\label{OSpecialization}
\S_Z \longrightarrow \S_Y \otimes_{\O(Y)} \O(Z)
\end{equation}
of affine group schemes over $Z$.

\begin{proposition}
If $Z$ is an irreducible subvariety of $Y$, then the diagram
$$
\xymatrix{\S_Z \ar[r] \ar[d] & D_Z \ar@{^{ (}->}[d]
\\
\S_Y \otimes_{\O(Y)} \O(Z) \ar[r] & D_Y \otimes_{\O(Y)} \O(Z)}
$$
commutes. \qed
\end{proposition}

Note that $D_Y(\O(Y)) = G_Y(\O(Y))$. Thus, the image of the tautological homomorphism $\vrho_Y \colon \pi \to \G_m^N(\O(Y))$ is contained in $D_Y(\O(Y))$. In Section \ref{subsectionIter}, we show that there is a homomorphism $\pi \to \S_Y(\O(Y))$ into the $\O(Y)$-rational points of $\S_Y$ that lifts $\vrho_Y$:
$$
\xymatrix{\pi \ar[d] \ar[dr]^{\vrho_Y} \\ \S_Y(\O(Y)) \ar[r] & D_Y(\O(Y)).}
$$
If $Z$ is an irreducible subvariety of $Y$, then the diagram
$$
\xymatrix{\pi \ar[r] \ar[d] & \S_Z(\O(Z)) \ar[d] \\ S_Y(\O(Y)) \ar[r] & \S_Y(\O(Z))}
$$
commutes. If $X$ deform-retracts onto a finite simplicial complex, then the homomorphism $\S_\vrho \to \S_Y \otimes_{\O(Y)} \C_\vrho$ is an isomorphism for general $\vrho \in Y$, assuming that $\vrho$ has Zariski dense image in $G_Y$ for general $\vrho \in Y$. As the comments in Section \ref{GZGY} indicate, this restriction is frequently satisfied, and in fact it is always satisfied if $Y$ is an irreducible subvariety of $\T$.

\subsection{Iterated Integrals}
\label{subsectionIter}

\noindent In this section, we introduce iterated integrals that generalize those of Chen \cite{Chen3} and Hain \cite{HainDeRham}. First, recall that $X$ is a smooth manifold such that $H_1(X,\Z)$ is torsion-free, that $\T = H^1(X,\C^*)$, and that $\pi = \pi_1(X,x_0)$. Let $Y$ be an irreducible subvariety of $\T^N$. There is a tautological homomorphism $\vrho_Y \colon \pi \to \G_m^N(\O(Y))$. Its image is contained in $G_Y(\O(Y))$, where $G_Y$ is the group subscheme of $\G_m^N$ defined in Section \ref{DY}. Each element $f$ in $\O(D_Y)$ gives a function $D_Y(\O(Y)) \to \O(Y)$ that sends the $\O(Y)$-algebra homomorphism $\phi \colon \O(D_Y) \to \O(Y)$ to $\phi(f)$. For $\vrho \in Y$, the diagram
$$
\xymatrix{\pi \ar[r]^-{\vrho_Y} \ar[d]^\vrho & \G_m^N(\O(Y)) \ar[d] \\ (\C^*)^N \ar@{=}[r] & (\C_\vrho^\times)^N}
$$
commutes.

Suppose that $\gamma \colon [0,1] \to X$ is a piecewise smooth loop in $X$ with basepoint $x_0$. Let $\tilde{\gamma}$ denote any lift of $\gamma$ to $\widetilde{X}$. If $\psi \in E^1(X,\VV_{Y,\alpha}) \otimes_{\O(Y)} V_{Y,\alpha^{-1}}$, where $\alpha \in G_Y^\vee$, then we define $\int_\gamma \psi = \int_{\tilde{\gamma}} \psi$. This is an element of $V_{Y,\alpha} \otimes_{\O(Y)} V_{Y,\alpha^{-1}}$, which is the trivial $D_Y$-module $\O(Y)$. That is, the integral $\int_\gamma \psi$ is an element of $\O(Y)$. We extend this definition to {\em iterated integrals} as follows.

Suppose that $\psi_1,\dots,\psi_r$ are elements of $E^1(X,\O_Y)$ and that $\varphi \in \O(D_Y)$.
If each $\psi_j \in E^1(X,\VV_{\alpha_j}) \otimes_{\O(Y)} V_{\alpha_j^{-1}}$, where $\alpha_j \in D_Y^\vee$, then we define
$$
\int_\gamma (\psi_1\cdots\psi_r|\varphi) = \varphi(\vrho_Y(\gamma))\int_{\tilde{\gamma}}\psi_1\cdots\psi_r \, \, \in \O(Y).
$$
This definition extends uniquely to the case where $\psi_1,\dots,\psi_r \in E^1(X,\O_Y)$ in such a way that the integral $\int_\gamma (\psi_1\cdots\psi_r|\varphi)$ is $\O(Y)$-multi-linear in the forms $\psi_j$ and in $\varphi$. If $r = 0$, we set $\int_\gamma(\phantom{i}|\varphi) = \varphi(\vrho_Y(\gamma))$.

\begin{definition}
The set $I(X)_Y$ of {\em iterated integrals} with coefficients in $\O(D_Y)$ is defined to be the set of all $\O(Y)$-linear combinations of integrals of the form $\int (\psi_1\cdots\psi_r|\varphi)$, where $r \geq 0$, $\psi_j \in E^1(X,\O_Y)$, and $\varphi \in \O(D_Y)$.
\end{definition}

The elements of $I(X)_Y$ will be regarded as $\O(Y)$-valued functions on the space $\Omega_{x_0} X$ of piecewise smooth loops in $X$.

\begin{definition}
We define $H^0(I(X)_Y)$ to be the subset of $I(X)_Y$ consisting of all elements that are constant on each homotopy class $[\gamma] \in \pi_1(X,x_0)$. We call the elements of $H^0(I(X)_Y)$ {\em locally constant iterated integrals} with coefficients in $\O(D_Y)$.
\end{definition}

Recall that each element of $E^1(X,\VV_{Y,\alpha})$ is a differential form on $\X$ with values in $\O(Y)$. That is, there are open sets $U$ covering $\X$ such that the restriction of any $\zeta \in E^1(X,\VV_{Y,\alpha})$ to $U$ is an element of $E^\bullet(U) \otimes_\C \O(Y)$. Since the image of a  lift $\tilde{\gamma}$ of a piecewise smooth path $\gamma \colon [0,1] \to X$ to $\X$ has compact image, it follows that there is an open set $U$ containing $\tilde{\gamma}$ such that the restriction of $\zeta$ to $U$ is an element of $E^\bullet(U) \otimes_\C \O(Y)$. Thus, the next proposition follows directly from the similar result for ordinary iterated integrals \cite[Equation (1.5.1)]{ChenAlg}.

\begin{proposition}
\label{shuffleY}
For $\psi_1,\dots,\psi_{p+q} \in E^1(X,\O_Y)$ and $\varphi,\theta \in \O(Y)[q_j^{\pm 1}]$, we have
$$
\int (\psi_1\cdots\psi_p|\varphi) \int (\psi_{p+1}\cdots\psi_{p+q} | \theta) = \sum_{\sigma \in Sh(p,q)} \int (\psi_{\sigma(1)}\cdots\psi_{\sigma(p+q)} | \varphi\theta),
$$
where $Sh(p,q)$ denotes the set of shuffles of type $(p,q)$. \qed
\end{proposition}

\begin{corollary}
The sets $I(X)_Y$ and $H^0(I(X)_Y)$ are $\O(Y)$-algebras, where the map $\O(Y) \to H^0(I(X)_Y)$ is given by $1 \mapsto \int (\phantom{i}|1)$. \qed
\end{corollary}

\subsection{The Map $\theta_Y \colon \pi \to \S_Y(\O(Y))$}
\label{thetaY}

\noindent Consider the reduced bar construction $$B(\O(Y),E^\bullet(X,\O_Y),\O(D_Y)),$$ which was examined in Section \ref{SYDefinition}. The degree of the element $[\psi_1|\cdots|\psi_r] \varphi$ is defined to be $\deg(\psi_1)+\cdots+\deg(\psi_r) - r$. Note that
$$
H^0 B(\O(Y),E^\bullet(X,\O_Y),\O(D_Y)) \subset B(\O(Y),E^\bullet(X,\O_Y),\O(D_Y)),
$$
since $B(\O(Y),E^\bullet(X,\O_Y),\O(D_Y))$ is is non-negatively weighted.

\begin{proposition}
\label{IterYMap}
There is an $\O(Y)$-algebra homomorphism $$H^0B(\O(Y),E^\bullet(X,\O_Y),\O(D_Y)) \longrightarrow H^0(I(X)_Y)$$
given by
\begin{equation}
\label{BarIterMap}
[\psi_1|\cdots|\psi_r]\varphi \, \longmapsto \, \int (\psi_1\cdots\psi_r|\varphi).
\end{equation}
\end{proposition}

\begin{remark}
\label{RemarkNatural}
The homomorphism in the theorem is natural with respect to specializations in the following sense. If $Z$ is an irreducible subvariety of $Y$, then the diagram
$$
\xymatrix{H^0B(\O(Y),E^\bullet(X,\O_Y),\O(D_Y)) \ar[r] \ar[d] & H^0(I(X)_Y) \ar[d] \\
H^0B(\O(Z),E^\bullet(X,\O_Z),\O(D_Z)) \ar[r] & H^0(I(X)_Z)}
$$
commutes, where down arrows are the canonical maps.
\end{remark}

\begin{proof}[Proof of Proposition \ref{IterYMap}]
The image of the map (\ref{BarIterMap}) is certainly contained in $I(X)_Y$. Suppose that the image is not contained in $H^0(I(X)_Y)$. Then there exists a closed $\zeta \in B(\O(Y),E^\bullet(X,\O_Y),\O(D_Y))$ of degree $0$ and loops $\gamma,\lambda \in \Omega_{x_0}X$ that have the same equivalence class in $\pi_1(X,x_0)$ such that $\int_\gamma \zeta \neq \int_\lambda \zeta$ as elements of $\O(Y)$. Choose $\vrho \in Y$ such that $\biggl\int_\gamma \zeta \biggr (\vrho) \neq \biggl\int_\lambda\zeta \biggr (\vrho)$.

Consider the canonical homomorphism
\begin{equation}
\label{BarSpecial2}
\phi_\vrho \colon B(\O(Y),E^\bullet(X,\O_Y),\O(D_Y)) \otimes_{\O(Y)} \C_\vrho \, \longrightarrow \, B(\C,E^\bullet(X,\O_\vrho),\O(D_\vrho))
\end{equation}
of differential graded Hopf algebras. By the commutativity of the diagram in Remark \ref{RemarkNatural},
$$
\int_\gamma \phi_\vrho(\zeta \otimes 1) = \biggl\int_\gamma \zeta\biggr (\vrho) \neq \biggl\int_\lambda \zeta\biggr (\vrho) = \int_\lambda \phi_\vrho(\zeta \otimes 1).
$$
Thus, by standard properties of Chen's iterated integrals \cite{Chen3} and Hain's generalization of them \cite{HainDeRham}, $\phi_\vrho(\zeta \otimes 1)$ is not closed in $B(\C,E^\bullet(X,\O_\vrho),\O(D_\vrho))$. This is a contradiction, since $\zeta$ is closed. Thus, the image of (\ref{BarIterMap}) is contained in $H^0(I(X)_Y)$.

To see that it is well defined, suppose that $\xi \in (\bigoplus_{s \geq 0} E^1(X,\O_Y)^{\otimes s}) \otimes \O(D_Y)$, where $\otimes s$ denotes a tensor product taken over $\O(Y)$. Suppose that the equivalence class $[\xi]$ in $B(\O(Y),E^\bullet(X,\O_Y),\O(D_Y)$ is trivial but that $\int \xi \neq 0$ as a function $\Omega_{x_0} X \to \O(Y)$. Then $[\xi] \in H^0B(\O(Y),E^\bullet(X,\O_Y),\O(D_Y))$. Choose $\vrho \in Y$ and $\gamma \in \Omega_{x_0} X$ such that $\biggl\int_\gamma \xi\biggr (\vrho) \neq 0$.

Consider the element $\phi_\vrho([\xi] \otimes 1)$ of $H^0B(\C,E^\bullet(X,\O_\vrho),\O(D_\vrho))$, which is trivial, since $[\xi]$ is trivial. By Remark \ref{RemarkNatural}, it follows that $\int_\gamma \phi_\vrho([\xi] \otimes 1) = \biggl\int_\gamma [\xi] \biggr(\vrho) \neq 0$. By \cite[Proposition 8.1]{HainDeRham}, the element $\phi_\vrho([\xi] \otimes 1)$ is nonzero in $B(\C,E^\bullet(X,\O_\vrho),\O(D_\vrho))$, a contradiction. The map (\ref{BarIterMap}) is therefore well-defined.

Propositions \ref{commprop} and \ref{shuffleY} imply that this map is a homomorphism.
\end{proof}

Recall that the group $\S_Y(\O(Y))$ of $\O(Y)$-rational points of $\S_Y$ is the set of $\O(Y)$-algebra homomorphisms $H^0B(\O(Y),E^\bullet(X,\O_Y),\O(D_Y)) \to \O(Y)$. If $\gamma \in \pi_1(X,x_0)$, then this proposition implies that $\theta_Y(\gamma) = \int_\gamma$ is an element of $\S_Y(\O(Y))$.

\begin{theorem}
\label{homomtheorem}
The map
$$
\pi_1(X,x_0) \, \overset{\theta_Y}{\longrightarrow} \, \S_Y(\O(Y)),
$$
given by $\gamma \mapsto \int_\gamma$, is a homomorphism of groups.
\end{theorem}

\begin{remark}
\label{ThetaYRemark}
The map $\theta_Y$ is natural in the sense that if $Z$ is an irreducible subvariety of $Y$, then the diagram
$$
\xymatrix{\pi_1(X,x_0) \ar[r]^{\theta_Z} \ar[d]_{\theta_Y} & \S_Z(\O(Z)) \ar[d] \\ \S_Y(\O(Y)) \ar[r] & \S_Y(\O(Z))}
$$
commutes.
\end{remark}

\begin{proof}[Proof of Theorem \ref{homomtheorem}]
If $Y = \{\vrho\}$, then this is Theorem \ref{thetavrho}. Suppose that $\gamma, \lambda \in \pi_1(X,x_0)$. Then $\beta_Y = \theta_Y(\gamma\lambda) - \theta_Y(\gamma)\theta_Y(\lambda)$ is a set map $\O(\S_Y) \, \longrightarrow \, \O(Y)$. For each $\vrho \in Y$, there is a canonical homomorphism
$$
\phi \colon \O(\S_Y) \otimes_{\O(Y)} \C_\vrho \, \longrightarrow \, \O(\S_\vrho)
$$
of differential graded algebras. Remark \ref{ThetaYRemark} implies that the diagram
$$
\xymatrix{\O(S_Y) \otimes_{\O(Y)} \C_\vrho \ar[d]_{\beta_Y \otimes 1_\vrho} \ar[r]^-\phi & \O(\S_\vrho) \ar[d]^{\beta_\vrho}\\ \C_\vrho \ar@{=}[r] & \C}
$$
commutes. The map $\beta_\vrho$ is zero by Theorem \ref{thetavrho}. Thus, the map $\beta_Y \otimes 1_\vrho$ must be zero for each $\vrho \in Y$. If the map $\beta_Y$ is not identically zero, choose $\zeta \in \O(\S_Y)$ such that $\beta_Y(\zeta)$ is nonzero as an element of $\O(Y)$. Choose $\vrho \in Y$ such that $\beta_Y(\zeta)(\vrho) \neq 0$. Then the map
$$
\beta_Y \otimes 1_\vrho \colon \O(\S_Y) \otimes_{\O(Y)} \C_\vrho \longrightarrow \C_\vrho
$$
satisfies $(\beta_Y \otimes 1_\vrho)(\zeta \otimes 1) = \beta_Y(\zeta)(\vrho) \neq 0$, a contradiction. Thus, the map $\beta_Y$ is identically zero.
\end{proof}

The next proposition follows directly from definitions.

\begin{proposition}
The homomorphism $\theta_Y$ lifts $\vrho_Y$:
$$
\xymatrix{\pi \ar[d]_{\theta_Y} \ar[dr]^{\vrho_Y} \\ \S_Y(\O(Y)) \ar[r] & D_Y(\O(Y)).}
$$\qed
\end{proposition}

\subsection{Constancy of Relative Completion}
\label{ConstancyY}

\noindent For the convenience of the reader, we recall the material of the previous sections. Recall that $X$ is a smooth manifold such that $H_1(X,\Z)$ is torsion-free, that $\T = H^1(X,\C^*)$, and that $\pi = \pi_1(X,x_0)$. Each element of $\T^N$ can be viewed as a representation $\pi \to (\C^*)^N$. Let $Y$ be an irreducible subvariety of $\T^N$. Define $G_Y$ to be the intersection of all group subschemes of $\G_m^N$ whose group of $\C$-rational points contains $\im \vrho$ for every $\vrho \in Y$. The group scheme $D_Y$ over $Y$ is defined by $D_Y = G_Y \otimes_\C \O(Y)$. This is a group subscheme of $\G_{m/Y}^N$. Recall that $\vrho_Y$ is the tautological homomorphism
$$
\vrho_Y \colon \pi \longrightarrow D_Y(\O(Y)).
$$
The affine group scheme $\S_Y$ over $Y$ satisfies
$$
\O(\S_Y) = H^0B(\O(Y),E^\bullet(X,\O_Y),\O(D_Y)).
$$
There is a canonical surjection $\S_Y \to D_Y$ of affine group schemes over $Y$ and a homomorphism $\theta_Y \colon \pi \to \S_Y(\O(Y))$ that lifts $\vrho_Y$. When $Y = \{\vrho\}$, the group scheme $\S_Y$ is the relative Malcev completion of $\pi$ with respect to $\vrho$. For each irreducible subvariety $Z$ of $Y$, there is a canonical homomorphism $\S_Z \to \S_Y \otimes_{\O(Y)} \O(Z)$ of affine group schemes over $Z$. The diagram
$$
\xymatrix{\pi \ar[d]_{\theta_Y} \ar[r]^{\theta_Z} & \S_Z(\O(Z)) \ar[d] \\ \S_Y(\O(Y)) \ar[r] & \S_Y(\O(Z))}
$$
commutes. In particular, for each $\vrho \in Y$, there is a canonical homomorphism $\S_\vrho \to \S_Y \otimes_{\O(Y)} \C_\vrho$. The next theorem is the main result of Section \ref{SectionY}. The proof is based on the Eilenberg-Moore spectral sequence and relies on Theorem \ref{GenericCohomDimension}.

\begin{theorem}
\label{MyBigTheorem}
Suppose that $X$ deform-retracts onto a finite simplicial complex. If $\vrho$ is Zariski dense in $G_Y$ for general $\vrho \in Y$, then the homomorphism $\S_\vrho \to \S_Y \otimes_{\O(Y)} \C_\vrho$ is an isomorphism of affine group schemes for general $\vrho \in Y$.
\end{theorem}

\begin{remark}
The statement that $\vrho$ is Zariski dense in $G_Y$ means that the image of $\vrho$ in $G_Y(\C)$ is Zariski dense in $G_Y$. If $Y$ is any irreducible subvariety of $\T$, then $\vrho$ is dense in $G_Y$ for general $\vrho \in Y$.
\end{remark}

\begin{remark} If $X$ is the complement of an arrangement of hyperplanes in a complex vector space, then there is a deformation retraction of $X$ onto a finite simplicial complex \cite[Theorem 5.40]{OT}. Thus, Theorem \ref{MyBigTheorem} can be applied to $X$.
\end{remark}

\begin{proof}[Proof of Theorem \ref{MyBigTheorem}]
The homomorphism $\S_\vrho \to \S_Y \otimes_{\O(Y)} \C_\vrho$ corresponds to the Hopf algebra homomorphism $\O(\S_Y) \otimes_{\O(Y)} \C_\vrho \longrightarrow \O(\S_\vrho)$. This is the canonical homomorphism
\begin{equation}
\label{H0Hopfs}
H^0B(\O(Y),E^\bullet(X,\O_Y),\O(D_Y)) \otimes_{\O(Y)} \C_\vrho \longrightarrow H^0B(\C,E^\bullet(X,\O_\vrho),\O(D_\vrho))
\end{equation}
of Hopf algebras. If $\vrho \in Y$ is Zariski dense in $G_Y$, then this homomorphism is induced by the canonical homomorphism
\begin{equation}
\label{22222}
B(\O(Y),E^\bullet(X,\O_Y),\O(D_Y)) \otimes_{\O(Y)} \C_\vrho \longrightarrow B(\C,E^\bullet(X,\O_\vrho),\O(D_\vrho)).
\end{equation}
If $\vrho$ is Zariski dense in $Y$, then Proposition \ref{BarSpecializationProposition} and Corollary \ref{YZSpecCorollary} imply that (\ref{22222}) is an isomorphism. By assumption, (\ref{22222}) is therefore an isomorphism for general $\vrho \in Y$. It suffices to prove that (\ref{H0Hopfs}) is an isomorphism for general $\vrho \in Y$.

For each irreducible subvariety $Z$ of $Y$, which can be $\vrho$ or $Y$ itself, let $E_n(Z)$ denote the Eilenberg-Moore spectral sequence corresponding to the reduced bar construction $B(\O(Z),E^\bullet(X,\O_Z),\O(D_Z))$. Each $E_n(Z)$ is a Hopf algebra over $\O(Z)$. There is a canonical homomorphism
$$
E_n(Y) \otimes_{\O(Y)} \O(Z) \longrightarrow E_n(Z)
$$
of differential graded Hopf algebras over $\O(Z)$. The Hopf algebra homomorphism $E_0(Y) \otimes_{\O(Y)} \C_\vrho \longrightarrow E_0(\vrho)$ is the map (\ref{22222}), and the Hopf algebra homomorphism $E_\infty(Y) \otimes_{\O(Y)} \C_\vrho \longrightarrow E_\infty(\vrho)$ is the map (\ref{H0Hopfs}).

It suffices to prove that $E_\infty(Y) \otimes_{\O(Y)} \C_\vrho \longrightarrow E_\infty(\vrho)$ is an isomorphism for general $\vrho \in Y$. Proposition \ref{EMSS1} implies that
$$
E_1(Y) = B(\O(Y),E^\bullet(X,\O_Y),\O(D_Y)).
$$
Thus,
$$E_1^{-s,t}(Y) = [H^+(X,\O_Y)^{\otimes s}]^t \otimes_{\O(Y)} \O(D_Y),$$ where $[H^+(X,\O_Y)^{\otimes s}]^t$ denotes the degree $t$ part of $H^+(X,\O_Y)^{\otimes s}$.
Note that this implies that each $E_1^{-s,t}(Y)$ is a countably generated $\O(Y)$-module, since $H^\bullet(X,\O_Y)$ is countably generated. Since $\O(Y)$ is Noetherian, submodules of countably generated $\O(Y)$-modules are Noetherian. It follows that $E_n^{-s,t}(Y)$ is a countably generated $\O(Y)$-module for $n \geq 1$ and all $s$ and $t$.

The image of $\vrho$ is Zariski dense in $G_Y$ for general $\vrho \in Y$. For such $\vrho$, the homomorphism $\O(D_Y) \otimes_{\O(Y)} \C_\vrho \to \O(D_\vrho)$ is an isomorphism. Theorem \ref{HYpTheorem} implies that $H^\bullet(X,\O_Y) \otimes_{\O(Y)} \C_\vrho \longrightarrow H^\bullet(X,\O_\vrho)$ is an isomorphism for general $\vrho \in Y$, since the general $\vrho$ is Zariski dense in $G_Y$. Thus, for general $\vrho \in Y$, the canonical homomorphism
$$
E_1(Y) \otimes_{\O(Y)} \C_\vrho \longrightarrow E_1(\vrho)
$$
is an isomorphism.

Suppose now that the canonical homomorphism $E_n(Y) \otimes_{\O(Y)} \C_\vrho \longrightarrow E_n(\vrho)$ is an isomorphism for general $\vrho \in Y$. Since each $E_n^{-s,t}(Y)$ is countably generated, $E_n(Y)$ is a complex of countably generated $\O(Y)$-modules. Theorem \ref{GenericCohomDimension} therefore implies that for general $\vrho \in Y$, the canonical homomorphism
$$
E_{n+1}(Y) \otimes_{\O(Y)} \C_\vrho \longrightarrow E_{n+1}(\vrho)
$$
is an isomorphism. It follows that the canonical homomorphism
$$
E_\infty(Y) \otimes_{\O(Y)} \C_\vrho \longrightarrow E_\infty(\vrho)
$$
is an isomorphism for general $\vrho \in Y$. This completes the proof.
\end{proof}

\end{document}